%% file: hitting.tex
\documentclass{article}

\usepackage[numbers]{natbib}
\usepackage[bookmarkstype=toc,colorlinks=true,citecolor=blue,linkcolor=red]{hyperref}  

\usepackage{graphicx}
\usepackage{amsmath,amssymb,amsfonts,amsthm}

\usepackage{geometry}
\usepackage{enumerate}

\newtheorem{theorem}{Theorem}[section]
\newtheorem{lem}{Lemma}[section]

\newtheorem{cor}{Corollary}[section]
\theoremstyle{definition}
\newtheorem{Def}{Definition}[section]
\newtheorem*{rmk*}{Remark}
\newtheorem{rmk}{Remark}[section]
\newtheorem{example}{Example}[section]

\makeatletter
    \renewcommand{\theequation}{
    \thesection.\arabic{equation}}
    \@addtoreset{equation}{section}
\makeatother

\makeatletter
    \renewcommand*{\section}{\@startsection{section}{1}{\z@}%
    {21pt}{12pt}{\reset@font\normalsize\bfseries}}
\makeatother
    
\makeatletter
    \renewcommand*{\subsection}{\@startsection{subsection}{2}{\z@}%
    {15pt}{6pt}{\reset@font\normalsize\mdseries\itshape}}
\makeatother

\makeatletter
    \renewcommand*{\subsubsection}{\@startsection{subsubsection}{3}{\z@}%
    {15pt}{6pt}{\reset@font\normalsize\mdseries\itshape}}
\makeatother

\makeatletter
\def\@listi{\leftmargin\leftmargini
  \topsep=.5\baselineskip 
  \partopsep=0pt \parsep=0pt \itemsep=0pt}
\let\@listI\@listi
\@listi
\def\@listii{\leftmargin\leftmarginii
  \labelwidth\leftmarginii \advance\labelwidth-\labelsep
  \topsep=0pt \partopsep=0pt \parsep=0pt \itemsep=0pt}
\def\@listiii{\leftmargin\leftmarginiii
  \labelwidth\leftmarginiii \advance\labelwidth-\labelsep
  \topsep=0pt \partopsep=0pt \parsep=0pt \itemsep=0pt}
\def\@listiv{\leftmargin\leftmarginiv
  \labelwidth\leftmarginiv \advance\labelwidth-\labelsep
  \topsep=0pt \partopsep=0pt \parsep=0pt \itemsep=0pt}
\makeatother

\title{Central limit theorems for pre-averaging covariance estimators under endogenous sampling times}
\author{Yuta Koike
\thanks{University of Tokyo, Graduate School of Mathematical Sciences, 3-8-1 Komaba, Meguro-ku, Tokyo 153-8914, Japan, Email: kyuta@ms.u-tokyo.ac.jp}}

\begin{document}

\maketitle 


\input{hitting/hitting_abstract}

\input{hitting/hitting_intro}
\input{hitting/hitting_setting}
\input{hitting/hitting_main}
\input{hitting/hitting_examples}
\input{hitting/hitting_application}

\input{hitting/hitting_simulation}
\newpage

\begin{figure}
\caption{Normal QQ plots of the transformed statistics for Scenario 1}\label{qqplot}
\begin{center}
\includegraphics[scale=0.8]{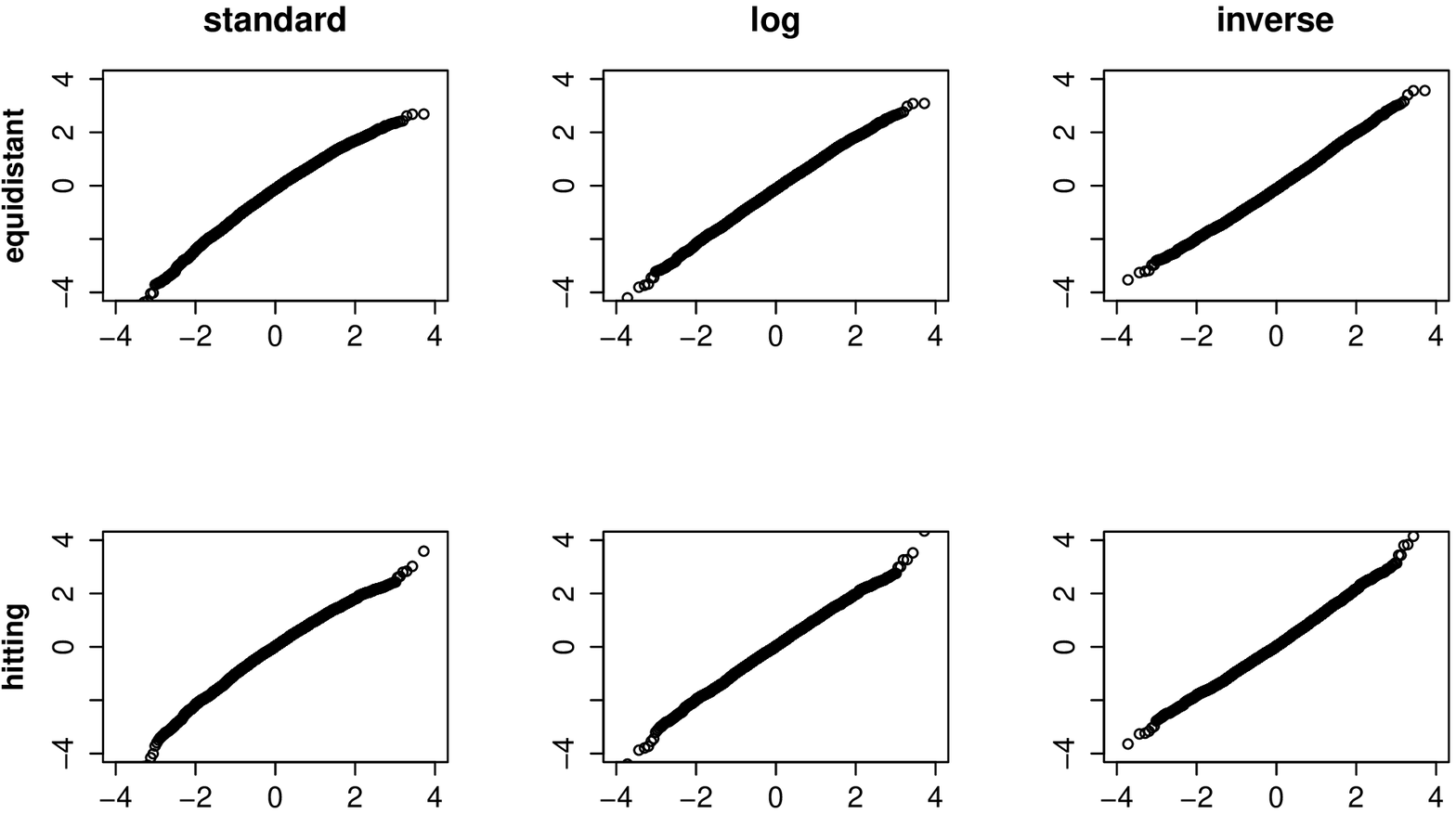}
\parbox{15cm}
{\small
\textit{Note}. We plot the QQ plots of the Studentized statistics for Scenario 1. The upper panels are for the equidistant sampling case and the lower panels are for the hitting sampling. The left panels refer to $S_{\text{PHY}}$, the middle panels refer to $S_{\text{log}}$ and the right panels refer to $S_{\text{inv}}$. 
}
\end{center}
\end{figure}

\appendix

\renewcommand*{\theequation}{\textup{\Alph{section}}.\arabic{equation}}

\section*{Appendix}


\input{hitting/hitting_proof_short}

\input{hitting/hitting_avar}

\input{hitting/hitting_dependent}

\input{hitting/hitting_mrc}

\input{hitting/hitting_acknowledgements}
{\small
\addcontentsline{toc}{section}{References}

\input{hitting.bbl}
}
 
\end{document}

%% file: hitting/hitting_abstract.tex
\begin{abstract}

We consider two continuous It\^o semimartingales observed with noise and sampled at stopping times in a nonsynchronous manner. In this article we establish a central limit theorem for the pre-averaged Hayashi-Yoshida estimator of their integrated covariance in a general endogenous time setting. In particular, we show that the time endogeneity has no impact on the asymptotic distribution of the pre-averaged Hayashi-Yoshida estimator, which contrasts the case for the realized volatility in a pure diffusion setting. We also establish a central limit theorem for the modulated realized covariance, which is another pre-averaging based integrated covariance estimator, and demonstrate the above property seems to be a special feature of the pre-averaging technique.\vspace{3mm}

\noindent \textit{Keywords}: Central limit theorem; Hitting times; Market microstructure noise; Nonsynchronous observations; Pre-averaging; Time endogeneity.
\end{abstract}

%% file: hitting/hitting_intro.tex
\section{Introduction}

Let $X=(X_t)_{t\in\mathbb{R}_+}$ be a continuous semimartingale on a stochastic basis $(\Omega,\mathcal{F},(\mathcal{F}_t),P)$. Suppose that for each $n$ we have a sequence $(t^n_i)_{i\in\mathbb{Z}_+}$ of $(\mathcal{F}_t)$-stopping times such that $t^n_0=0$ and $t^n_i\uparrow\infty$ as $i\to\infty$. Then, as is well known, for any $t\in\mathbb{R}_+$ the quantity $\text{RV}^n_t:=\sum_{i:t^n_i\leq t}(X_{t^n_i}-X_{t^n_{i-1}})^2$ converges to the quadratic variation $[X]_t$ of $X$ in probability as $n\to\infty$, provided that $\Delta_n(t):=\sup_{i}(t^n_i\wedge t-t^n_{i-1}\wedge t)\to^p0$ (see Theorem I-4.47 in \cite{JS} for instance). Here, the notation $\to^p$ means convergence in probability. In recent years this classic result has been highlighted in the context of a high-frequency data analysis. In the econometric literature, the quantities $\text{RV}^n_t$ and $[X]_t$ are called the \textit{realized volatility} (RV) and \textit{integrated volatility} (IV) (up to the time $t$) respectively. Then, the above result is equivalent to say that the RV is a consistent estimator for the IV if $\Delta_n(t)\to^p0$ as $n\to\infty$. Because the importance of the IV as an index of the volatility of assets has been recognized since a series of studies by \citet{AB1997,AB1998} and the increasing availability of high-frequency data in finance makes the assumption $\Delta_n(t)\to^p0$ reliable, the statistical theory for the estimation of the IV has been developed by many authors recently.

One of the interesting topics after the consistency is the asymptotic distribution theory. In the context of the statistical estimation of diffusion parameters, such a theory has already appeared in \citet{Dohnal1987} and \citet{GCJ1993,GCJ1994}. See also the recent works of \citet{UY2013} and \citet{OY2012}. Also, the early limit theory for the RV was developed in \citet{Jacod1994unp}, \citet{JP1998} and \citet{Zhang2001} in different contexts. In the present situation, under some regularity conditions \citet{BNS2002} developed a ``feasible'' central limit theorem
\begin{equation}\label{bnsCLT}
\frac{\text{RV}^n_t-[X]_t}{\sqrt{\frac{2}{3}\text{RQ}^n_t}}\xrightarrow{d}N(0,1)\qquad\text{ as }n\to\infty
\end{equation}
with the regular sampling case $t^n_i=i/n$. Here, the notation $\xrightarrow{d}$ means convergence in distribution and the quantity $\text{RQ}^n_t$ defined by $\text{RQ}^n_t=\sum_{i:t_i\leq t}(X_{t^n_i}-X_{t^n_{i-1}})^4$ is sometimes called the \textit{realized quarticity}. In this case, even the second-order asymptotic expansion of the statistic in the left-hand side of $(\ref{bnsCLT})$ was developed in \citet{Yoshida2012}. 

It is natural to ask what happens when we consider more general stopping times as the sampling times $(t_i^n)$. In fact, \citet{BNS2006tc} and \citet{MZ2006} showed that the convergence $(\ref{bnsCLT})$ is also valid with more general deterministic sampling times. Moreover, even in the case that the sampling times could be random and endogenous (i.e., path-dependent) $(\ref{bnsCLT})$ is still valid as long as $(t^n_i)$ satisfies a kind of strong predictability condition, as shown in \citet{HJY2011}, \citet{HY2011} and \citet{PY2008}. Here, the strong predictability condition intuitively means that the future sampling times are determined with delay. See \cite{HJY2011} and \cite{HY2011} for more precise definitions. Such a kind of condition has already appeared in \cite{GCJ1994} and \cite{Jacod1994unp}. Dropping the strong predictability condition is much difficult. For some special hitting-time-based sampling schemes, $(\ref{bnsCLT})$ was verified by \citet{Fu2010} and \citet{FR2012}. However, the convergence $(\ref{bnsCLT})$ could fail for general endogenous sampling times. In fact, \citet{Fu2009,Fu2010b} showed that the asymptotic distribution of the RV is determined by the asymptotic skewness and kurtosis of observed returns. More precisely, suppose that $X$ is a continuous local martingale with $E[\langle X\rangle_t^6]<\infty$ for simplicity. Then, set $\mathcal{G}^k_{j,n}=E[(X_{t^n_{j+1}}-X_{t^n_j})^k|\mathcal{F}_{t^n_j}]$ for every $j,n$ and each $k=2,\dots,12$ and suppose also that there exist $(\mathcal{F}_t)$-adapted locally bounded left continuous processes $u$ and $v$ such that $\mathcal{G}^3_{j,n}/\mathcal{G}^2_{j,n}=v_{t^n_j}n^{-1/2}+o_p(n^{-1/2})$, $\mathcal{G}^4_{j,n}/\mathcal{G}^2_{j,n}=u^2_{t^n_j}n^{-1}+o_p(n^{-1})$ and $\mathcal{G}^{2k}_{j,n}/\mathcal{G}^2_{j,n}=o_p(n^{-k/2})$ $(k=3,4,6)$ uniformly in $j$ with $t^n_j\leq t$ as $n\to\infty$. Suppose further that $\sum_{j:t_j\leq t}\mathcal{G}^2_{j,n}=O_p(1)$ as $n\to\infty$. Then, the asymptotic distribution of the (scaled) estimation error $\sqrt{n}(\text{RV}^n_t-[X]_t)$ of the RV is given by
\begin{equation}\label{limitbias}
\frac{2}{3}\int_0^t v_s\mathrm{d}X_s+\sqrt{\frac{2}{3}}\int_0^t\sqrt{u_s^2-\frac{2}{3}v_s^2}\mathrm{d}W_{[X]_s},
\end{equation}
where $W$ is a standard Wiener process independent of $\mathcal{F}$. See Theorem 3.10 of \cite{Fu2010b} for details. This type of result was also obtained by \citet{LMRZZ2012}. We call the first integral in $(\ref{limitbias})$ limiting bias, following \cite{FR2012}.

Though the limit theory viewed in the above provides us a beautiful framework for estimating the IV from high-frequency financial data, we encounter another problem called \textit{market microstructure noise} when we focus on ultra high-frequencies. For this reason, recently many authors have proposed alternative estimators for the IV in consideration of microstructure noise e.g., the \textit{two-time scale realized volatility} of \cite{ZMA2005}, \textit{realized kernel} of \cite{BNHLS2008}, \textit{pre-averaging estimator} of \cite{PV2009,JLMPV2009} and \textit{realized quasi-maximum likelihood estimator} of \cite{Xiu2010}. The aim of this article is to answer a natural question that what happens in the asymptotic distribution of such a kind of estimator when sampling times are random and endogenous. This type of problem has been well studied in recent years when sampling times are deterministic or random but independent of observations in connection with the problem of \textit{nonsynchronous observations}, which is another important problem for analyzing high-frequency data of multiple assets. See \cite{BNHLS2011,Bibinger2012,CPV2013,SX2012a} for example. The case that both of the microstructure noise and the time endogeneity are present was considered in \citet{LZZ2012}, and in that article they constructed a new estimator and developed an asymptotic distribution theory of it.

In this article we will focus on the pre-averaging estimators, especially the \textit{pre-averaged Hayashi-Yoshida estimator} (PHY)  proposed in \citet{CKP2010}, which is a pre-averaging version of the \textit{Hayashi-Yoshida estimator} proposed in \citet{HY2005}. For the case with deterministic sampling times, the asymptotic distribution of this estimator was derived in \cite{CPV2013}. The case that a kind of strong predictability condition holds true was also developed in \cite{Koike2012phy}. In both cases no limiting bias appears, which is of course naturally predicted from the RV case. Interestingly, in this article we will show that \textit{nothing} happens even if we drop the strong predictability condition in the above (more precisely, we can replace the strong predictability condition in \cite{Koike2012phy} by a kind of continuity for the conditionally expected durations). That is, the asymptotic distribution of the PHY does not change even in the presence of the time endogeneity, in particular any limiting bias does not appear. This is quite different from the RV. Furthermore, we will show our result in the bivariate setting with nonsynchronous observations because it causes no difficulty compared with the univariate setting. This is completely different from the no-noise case and reflects the fact that the nonsynchronicity of observation times is less important in the presence of noise, as shown in \cite{BHMR2013}.

Compared with the estimator proposed in \cite{LZZ2012}, the PHY has three advantages, except we can apply it to nonsynchronous data. First, it attains the optimal convergence rate. Second, it is robust to a certain kind of autocorrelated noise (see Section \ref{secdepnoise}). Third, we do not need to correct the limiting bias of the estimator, so that it is easier for implementation.

The plan of this article is as follows. Section \ref{setting} presents the mathematical model and the construction of the pre-averaged Hayashi-Yoshida estimator. Section \ref{main} is devoted to the main result of this article. Section \ref{examples} provides some concrete examples of sampling times which are possibly endogenous. Section \ref{application} discusses Studentization, autocorrelated noise and a comparison between some existing approaches, and Section \ref{simulation} uses Monte Carlo simulations to verify the conclusions obtained from the previous sections. Most of the proofs are given in the Appendix.\vspace{2mm}

\noindent{\bf Notation}\vspace{2mm}

We denote by $\mathbb{D}(\mathbb{R}_+)$ the space of c\`adl\`ag functions on $\mathbb{R}_+$ equipped with the Skorokhod topology. A sequence of random elements $X^n$ defined on a probability space $(\Omega,\mathcal{F},P)$ is said to \textit{converge stably in law} to a random element $X$ defined on an appropriate extension $(\tilde{\Omega},\tilde{\mathcal{F}} ,\tilde{P})$ of $(\Omega,\mathcal{F},P)$ if $E[Yg(X^n)]\rightarrow E[Yg(X)]$ for any $\mathcal{F}$-measurable and bounded random variable $Y$ and any bounded and continuous function $g$. We then write $X^n\rightarrow^{d_s}X$. A sequence $(X^n)$ of stochastic processes is said to converge to a process $X$ \textit{uniformly on compacts in probability} (abbreviated \textit{ucp}) if, for each $t>0$, $\sup_{0\leq s\leq t}|X^n_s-X_s|\rightarrow^p0$ as $n\rightarrow\infty$. 

If a process $V$ is (pathwise) absolutely continuous, we denote its density process by $V'$. $|\cdot|$ denotes the Lebesgue measure. For a (random) interval $I$ and a time $t\in\mathbb{R}_+$, we write $I(t)=I\cap[0,t)$.

%% file: hitting/hitting_setting.tex

\section{The setting}\label{setting}

\subsection{Model}

Let $\mathcal{B}^{(0)}=(\Omega^{(0)},\mathcal{F}^{(0)},\mathbf{F}^{(0)}=(\mathcal{F}^{(0)}_t)_{t\in\mathbb{R}_+} ,P^{(0)})$ be a stochastic basis. For any $t\in\mathbb{R}_+$ we have a transition probability $Q_t(\omega^{(0)},\mathrm{d}z)$ from $(\Omega^{(0)},\mathcal{F}^{(0)}_t)$ into $\mathbb{R}^2$, which satisfies $\int z Q_t(\omega^{(0)},\mathrm{d}z)=0.$
We endow the space $\Omega^{(1)}=(\mathbb{R}^2)^{[0,\infty)}$ with the product Borel $\sigma$-field $\mathcal{F}^{(1)}$ and with the probability $Q(\omega^{(0)},\mathrm{d}\omega^{(1)})$ which is the product $\otimes_{t\in\mathbb{R}_+}Q_t(\omega^{(0)},\cdot)$. We also call $(\epsilon_t)_{t\in\mathbb{R}_+}$ the ``canonical process'' on $(\Omega^{(1)},\mathcal{F}^{(1)})$ and the filtaration $\mathcal{F}^{(1)}_t=\sigma(\epsilon_s;s\leq t)$. Then we consider the stochastic basis $\mathcal{B}=(\Omega,\mathcal{F},\mathbf{F}=(\mathcal{F}_t)_{t\in\mathbb{R}_+} ,P)$ defined as follows:
\begin{gather*}
\Omega=\Omega^{(0)}\times\Omega^{(1)},\qquad
\mathcal{F}=\mathcal{F}^{(0)}\otimes\mathcal{F}^{(1)},\qquad
\mathcal{F}_t=\cap_{s>t}\mathcal{F}^{(0)}_s\otimes\mathcal{F}^{(1)}_s,\\
P(\mathrm{d}\omega^{(0)},\mathrm{d}\omega^{(1)})=P^{(0)}(\mathrm{d}\omega^{(0)})Q(\omega^{(0)},\mathrm{d}\omega^{(1)}).
\end{gather*}
Any variable or process which is defined on either $\Omega^{(0)}$ or $\Omega^{(1)}$ can be considered in the usual way as a variable or a process on $\Omega$.

Now we introduce our observation data. Let $X$ and $Y$ be two continuous semimartingales on $\mathcal{B}^{(0)}$. Also, we have two sequences of $\mathbf{F}^{(0)}$-stopping times $(S^i)_{i\in\mathbb{Z}_+}$ and $(T^j)_{j\in\mathbb{Z}_+}$ that are increasing a.s., 
\begin{equation}\label{increase}
S^i\uparrow\infty\qquad \textrm{and}\qquad T^j\uparrow\infty.
\end{equation}
As a matter of convenience we set $S^{-1}=T^{-1}=0$. These stopping times implicitly depend on a parameter $n\in\mathbb{N}$, which represents the frequency of the observations. Denote by $(b_n)$ a sequence of positive numbers tending to 0 as $n\to\infty$ (typically $b_n=n^{-1}$). Let $\xi'$ be a constant satisfying $0<\xi'<1$. In this paper, we will always assume that
\begin{equation}\label{A4}
r_n(t):=\sup_{i\in\mathbb{Z}_+}(S^i\wedge t-S^{i-1}\wedge t)\vee\sup_{j\in\mathbb{Z}_+}(T^j\wedge t-T^{j-1}\wedge t)=o_p(b_n^{\xi'})
\end{equation}
as $n\to\infty$ for any $t\in\mathbb{R}_+$.

The processes $X$ and $Y$ are observed at the sampling times $(S^i)$ and $(T^j)$ with observation errors $(U^X_{S^i})_{i\in\mathbb{Z}_+}$ and $(U^Y_{T^j})_{j\in\mathbb{Z}_+}$ respectively. We assume that the observation errors have the following representations:
\begin{equation*}
U^X_{S^i}=b_n^{-1/2}(\underline{X}_{S^i}-\underline{X}_{S^{i-1}})+\epsilon^X_{S^i},\qquad
U^Y_{T^j}=b_n^{-1/2}(\underline{Y}_{T^j}-\underline{Y}_{T^{j-1}})+\epsilon^Y_{T^j}.
\end{equation*}
Here, $\epsilon_t=(\epsilon^X_t,\epsilon^Y_t)$ for each $t$, while $\underline{X}$ and $\underline{Y}$ are two continuous semimartingales on $\mathcal{B}^{(0)}$. We can take $\underline{X}=\phi^X X$ and $\underline{Y}=\phi^Y Y$ for some constants $\phi^X$ and $\phi^Y$, so that the observation errors can be correlated with the returns of the latent processes $X$ and $Y$. Moreover, $\underline{X}$ and $\underline{Y}$ could also depend on the sampling times. For these reasons we will refer to $(b_n^{-1/2}(\underline{X}_{S^i}-\underline{X}_{S^{i-1}}))_{i\in\mathbb{Z}_+}$ and $(b_n^{-1/2}(\underline{Y}_{T^j}-\underline{Y}_{T^{j-1}}))_{j\in\mathbb{Z}_+}$ as the endogenous noise. The factor $b_n^{-1/2}$ is necessary for the endogenous noise not to degenerate asymptotically. Such a kind of noise appears in e.g., \cite{BNHLS2011} and \cite{KL2008}. After all, we have the observation data $\mathsf{X}=(\mathsf{X}_{S^i})_{i\in\mathbb{Z}_+}$ and $\mathsf{Y}=(\mathsf{Y}_{T^j})_{j\in\mathbb{Z}_+}$ of the forms $\mathsf{X}_{S^i}=X_{S^i}+U^X_{S^i}$ and $\mathsf{Y}_{T^j}=Y_{T^j}+U^Y_{T^j}$.

\subsection{Construction of the estimator}

In this subsection we explain the construction of the pre-averaged Hayashi-Yoshida estimator. First we introduce a concept called the \textit{pre-averaging}, which was originally proposed by \cite{PV2009} and generalized by \cite{JLMPV2009}. We choose a sequence $k_n$ of positive integers and a number $\theta\in(0,\infty)$ satisfying $k_n\sqrt{b_n}=\theta +o(b_n^{1/4})$
as $n\to\infty$ (for example $k_n=\lceil\theta /\sqrt{b_n}\rceil$). We associate the random intervals $I^i=[S^{i-1},S^i)$ and $J^j=[T^{j-1},T^j)$ with the sampling scheme $(S^i)$ and $(T^j)$ and refer to $\mathcal{I}=(I^i)_{i\in\mathbb{N}}$ and $\mathcal{J}=(J^j)_{j\in\mathbb{N}}$ as the sampling designs for $X$ and $Y$. For a function $\alpha$ on $\mathbb{R}_+$, we introduce the \textit{pre-averaging observation data} of $X$ and $Y$ with the weight function $\alpha$ and based on the sampling designs $\mathcal{I}$ and $\mathcal{J}$ respectively as follows:
\begin{align*}
\overline{\mathsf{X}}_\alpha(\mathcal{I})^i=\sum_{p=1}^{k_n-1}\alpha\left (\frac{p}{k_n}\right)\left(\mathsf{X}_{S^{i+p}}-\mathsf{X}_{S^{i+p-1}}\right),\quad
\overline{\mathsf{Y}}_\alpha(\mathcal{J})^j=\sum_{q=1}^{k_n-1}\alpha\left (\frac{q}{k_n}\right)\left(\mathsf{Y}_{T^{j+q}}-\mathsf{Y}_{T^{j+q-1}}\right),\qquad
i,j=0,1,\dots.
\end{align*}
In the following we fix a continuous function $g:[0,1]\rightarrow\mathbb{R}$ which is piecewise $C^1$ with a piecewise Lipschitz derivative $g'$ and satisfies $g(0)=g(1)=0$ and $\psi_{HY}:=\int_0^1 g(x)\mathrm{d}x\neq 0$
(for example $g(x)=x\wedge(1-x)$). 

The following quantity was introduced in Christensen et al. \cite{CKP2010} :
\begin{Def}[Pre-averaged Hayashi-Yoshida estimator]\label{Defphy}
The \textit{pre-averaged Hayashi-Yoshida estimator} (PHY) of $\mathsf{X}$ and $\mathsf{Y}$ associated with sampling designs $\mathcal{I}$ and $\mathcal{J}$ is the process
\begin{equation*}
PHY(\mathsf{X},\mathsf{Y};\mathcal{I},\mathcal{J})^n_t
=\frac{1}{(\psi_{HY}k_n)^2}\sum_{\begin{subarray}{c}
i,j=0\\
S^{i+k_n}\vee T^{j+k_n}\leq t
\end{subarray}}^{\infty}\overline{\mathsf{X}}_g(\mathcal{I})^i\overline{\mathsf{Y}}_g(\mathcal{J})^j 1_{\{[S^i,S^{i+k_n})\cap[T^j,T^{j+k_n})\neq\emptyset\}},\qquad t\in\mathbb{R}_+.
\end{equation*}
\end{Def}

For a technical reason explained in \cite{Koike2012phy}, we modify the above estimator as follows. The following notion was introduced to this area in \citet{BNHLS2011}:
\begin{Def}[Refresh time]
The first refresh time of sampling designs $\mathcal{I}$ and $\mathcal{J}$ is defined as $R^0=S^0\vee T^0$, and then subsequent refresh times as
\begin{align*}
R^k:=\min\{S^i|S^i>R^{k-1}\}\vee\min\{T^j|T^j>R^{k-1}\},\qquad k=1,2,\dots.
\end{align*}
\end{Def}

We introduce new sampling schemes by a kind of the next-tick interpolations to the refresh times. That is, we define $\widehat{S}^0:=S^0$, $\widehat{T}^0:=T^0$, and
\begin{align*}
\widehat{S}^k:=\min\{S^i|S^i>R^{k-1}\},\quad\widehat{T}^k:=\min\{T^j|T^j>R^{k-1}\},\qquad k=1,2,\dots.
\end{align*}
Note that $\widehat{S}^k$ is an $\mathbf{F}^{(0)}$-stopping time because
\begin{equation}\label{refreshrep}
\widehat{S}^k=\inf_{i\in\mathbb{N}} S^i_{\{S^i>R^{k-1}\}}.
\end{equation}
Here, for a stopping time $T$ with respect to filtration $(\mathcal{F}_t)$ and a set $A\in\mathcal{F}_T$, we define $T_A$ by $T_A(\omega)=T(\omega)$ if $\omega\in A$; $T_A(\omega)=\infty$ otherwise (see I-1.15 of \cite{JS}). Similarly $\widehat{T}^k$ is also an $\mathbf{F}^{(0)}$-stopping time, hence so is $R^k$.

Then, we create new sampling designs as follows:
\begin{align*}
\widehat{I}^k:=[\widehat{S}^{k-1},\widehat{S}^k),\qquad\widehat{J}^k:=[\widehat{T}^{k-1},\widehat{T}^k),\qquad
\widehat{\mathcal{I}}:=(\widehat{I}^i)_{i\in\mathbb{N}},\qquad\widehat{\mathcal{J}}:=(\widehat{J}^j)_{j\in\mathbb{N}}.
\end{align*}
For the sampling designs $\widehat{\mathcal{I}}$ and $\widehat{\mathcal{J}}$ obtained in such a manner, we consider the pre-averaging observation data $\overline{\mathsf{X}}(\widehat{\mathcal{I}})^i$ and $\overline{\mathsf{Y}}(\widehat{\mathcal{J}})^j$ of $X$ and $Y$ based on the sampling designs $\widehat{\mathcal{I}}$ and $\widehat{\mathcal{J}}$ respectively i.e.,
\begin{align*}
\overline{\mathsf{X}}_g(\widehat{\mathcal{I}})^i=\sum_{p=1}^{k_n-1}g\left (\frac{p}{k_n}\right)\left(\mathsf{X}_{\widehat{S}^{i+p}}-\mathsf{X}_{\widehat{S}^{i+p-1}}\right),\quad
\overline{\mathsf{Y}}_g(\widehat{\mathcal{J}})^j=\sum_{q=1}^{k_n-1}g\left (\frac{q}{k_n}\right)\left(\mathsf{Y}_{\widehat{T}^{j+q}}-\mathsf{Y}_{\widehat{T}^{j+q-1}}\right),\qquad
i,j=0,1,\dots.
\end{align*}
We refer to these quantities as the \textit{pre-averaging data in refresh time}. Finally, our objective estimator is given by $\widehat{PHY}(\mathsf{X},\mathsf{Y})^n:=PHY(\mathsf{X},\mathsf{Y};\widehat{\mathcal{I}},\widehat{\mathcal{J}})^n$. More precisely, we have
\begin{align*}
\widehat{PHY}(\mathsf{X},\mathsf{Y})^n_t
=\frac{1}{(\psi_{HY}k_n)^2}\sum_{\begin{subarray}{c}
i,j=0\\
\widehat{S}^{i+k_n}\vee \widehat{T}^{j+k_n}\leq t
\end{subarray}}^{\infty}\overline{\mathsf{X}}_g(\widehat{\mathcal{I}})^i\overline{\mathsf{Y}}_g(\widehat{\mathcal{J}})^j 1_{\{[\widehat{S}^i,\widehat{S}^{i+k_n})\cap[\widehat{T}^j,\widehat{T}^{j+k_n})\neq\emptyset\}},\qquad t\in\mathbb{R}_+.
\end{align*}

%% file: hitting/hitting_main.tex

\section{Main results}\label{main}

\subsection{Conditions}

We start with introducing some notation and conditions in order to state our main result. First, for any continuous semimartingale $Z$ on $\mathcal{B}^{(0)}$, we write its canonical decomposition as $Z=A^Z+M^Z$, where $A^Z$ is a continuous $\mathbf{F}^{(0)}$-adapted process with a locally finite variation and $M^Z$ is a continuous $\mathbf{F}^{(0)}$-local martingale.

Next, let $N^n_t=\sum_{k=1}^{\infty}1_{\{R^k\leq t\}}$, $N^{n,1}_t=\sum_{k=1}^{\infty}1_{\{\widehat{S}^k\leq t\}}$ and $N^{n,2}_t=\sum_{k=1}^{\infty}1_{\{\widehat{T}^k\leq t\}}$ for each $t\in\mathbb{R}_+$ and 
\begin{align*}
\Gamma^k=[R^{k-1},R^k),\qquad \check{I}^k:=[\check{S}^k,\widehat{S}^k),\qquad\check{J}^k:=[\check{T}^k,\widehat{T}^k)
\end{align*}
for each $k\in\mathbb{N}$. Here, for each $t\in\mathbb{R}_+$ we write $\check{S}^k=\sup_{S^i<\widehat{S}^k}S^i$ and $\check{T}^k=\sup_{T^j<\widehat{T}^k}T^j$. Note that $\check{S}^k$ and $\check{T}^k$ may not be stopping times.  

Let $\mathbf{H}^n=(\mathcal{H}^n_t)_{t\in\mathbb{R}_+}$ be a sequence of filtrations of $\mathcal{F}^{(0)}$ to which $N^n$, $N^{n,1}$ and $N^{n,2}$ are adapted. For each $n$, we also assume that $A^Z$ and $M^Z$ are adapted to $\mathbf{H}^n$ for every $Z\in\{X,Y,\underline{X},\underline{Y}\}$. Then, for each $n$ and each $\rho\geq0$ we define the processes $\chi^n$, $G(\rho)^n$, $F(\rho)^{n,1}$, $F(\rho)^{n,2}$ and $F(1)^{n,1* 2}$ by
\begin{gather*}
G(\rho)^n_s=E\left[\left(b_n^{-1}|\Gamma^{k}|\right)^\rho\big|\mathcal{H}_{R^{k-1}}^n\right],\quad
F(\rho)^{n,1}_{s}=E\left[\left(b_n^{-1}|\check{I}^{k}|\right)^\rho\big|\mathcal{H}_{\widehat{S}^{k-1}}^n\right],\quad
F(\rho)^{n,2}_{s}=E\left[\left(b_n^{-1}|\check{J}^{k}|\right)^\rho\big|\mathcal{H}_{\widehat{T}^{k-1}}^n\right],\\
\chi^n_{s}=P(\widehat{S}^k=\widehat{T}^k\big|\mathcal{H}_{R^{k-1}}^n),\qquad
F(1)^{n,1*2}_{s}=b_n^{-1}E\left[|\check{I}^k*\check{J}^k|\big|\mathcal{H}_{R^{k-1}}^n\right]
\end{gather*}
when $s\in\Gamma^k$. Here,   $\check{I}^k*\check{J}^k=(\check{I}^{k}\cap\check{J}^k)\cup(\check{I}^{k+1}\cap\check{J}^k)\cup(\check{I}^{k}\cap\check{J}^{k+1})$.

The following condition is necessary to compute the asymptotic variance of the estimation error of our estimator explicitly. 
\begin{enumerate}

\item[{[H1]}] (i) For each $n$, we have a c\`adl\`ag $\mathbf{F}^{(0)}$-adapted process $G^n$ and a random subset $\mathcal{N}^0_n$ of $\mathbb{N}$ such that $(\#\mathcal{N}^0_n)_{n\in\mathbb{N}}$ is tight, $G(1)^n_{R^{k-1}}=G^n_{R^{k-1}}$ for any $k\in\mathbb{N}-\mathcal{N}^0_n$, and there exist a c\`adl\`ag $\mathbf{F}^{(0)}$-adapted process $G$ and a constant $\delta>1-\xi'$ satisfying that $G$ and $G_{-}$ do not vanish and that $b_n^{-\delta}(G^n-G)\xrightarrow{ucp}0$ as $n\to\infty$.

(ii) There exists a constant $\rho>1/\xi'$ such that $\left(\sup_{0\leq s\leq t}G(\rho)^n_{s}\right)_{n\in\mathbb{N}}$ is tight for all $t>0$.

(iii) For each $n$, we have a c\`adl\`ag $\mathbf{F}^{(0)}$-adapted process $\chi^{\prime n}$ and a random subset $\mathcal{N}'_n$ of $\mathbb{N}$ such that $(\#\mathcal{N}'_n)_{n\in\mathbb{N}}$ is tight, $\chi^n_{R^{k-1}}=\chi^{\prime n}_{R^{k-1}}$ for any $k\in\mathbb{N}-\mathcal{N}'_n$, and there exist a c\`adl\`ag $\mathbf{F}^{(0)}$-adapted process $\chi$ and a constant $\delta'>1-\xi'$ satisfying $b_n^{-\delta'}(\chi^{\prime n}-\chi)\xrightarrow{ucp}0$ as $n\to\infty$.

(iv) For each $n$ and $l=1,2,1*2$, we have a c\`adl\`ag $\mathbf{F}^{(0)}$-adapted process $F^{n,l}$ and a random subset $\mathcal{N}^l_n$ of $\mathbb{N}$ such that $(\#\mathcal{N}^l_n)_{n\in\mathbb{N}}$ is tight, $F(1)^{n,l}_{R^{k-1}}=F^{n,l}_{R^{k-1}}$ for any $k\in\mathbb{N}-\mathcal{N}^l_n$, and there exist a c\`adl\`ag $\mathbf{F}^{(0)}$-adapted processes $F^l$ and a constant $\delta^l>1-\xi'$ satisfying $b_n^{-\delta^l}(F^{n,l}-F^l)\xrightarrow{ucp}0$ as $n\to\infty$.

(v) There exists a constant $\rho'>1/\xi'$ such that $\left(\sup_{0\leq s\leq t}F(\rho')^{n,l}_{s}\right)_{n\in\mathbb{N}}$ is tight for all $t>0$ and $l=1,2$.

\end{enumerate}

The following condition is a sufficient one for the condition [H1]:
\begin{enumerate}
\item[{[H1$^{\sharp}$]}]

(i) There exists a number $\bar{\rho}>1/\xi'$ such that for every $\rho\in[0,\bar{\rho}]$ we have a c\`adl\`ag $\mathbf{F}^{(0)}$-adapted process $G(\rho)$ such that $G(\rho)^n\xrightarrow{ucp}G(\rho)$ as $n\to\infty$. Furthermore, $G$ and $G_{-}$ do not vanish and there exists a constant $\delta>1-\xi'$ satisfying $b_n^{-\delta}(G(1)^n-G)\xrightarrow{ucp}0$ as $n\to\infty$ with $G=G(1)$.

(ii) There exist a c\`adl\`ag $\mathbf{F}^{(0)}$-adapted process $\chi$ and a constant $\delta>1-\xi'$ such that $b_n^{-\delta}(\chi^{n}-\chi)\xrightarrow{ucp}0$ as $n\to\infty$.

(iii) There exists a number $\bar{\rho}>1/\xi'$ such that for every $l=1,2$ and every $\rho'\in[0,\bar{\rho}]$ we have a c\`adl\`ag $\mathbf{F}^{(0)}$-adapted process $F(\rho)^l$ such that $F(\rho)^{n,l}\xrightarrow{ucp}F(\rho)^l$ as $n\to\infty$. Furthermore, there exists a constant $\delta>1-\xi'$ satisfying $b_n^{-\delta}(F(1)^{n,l}-F^l)\xrightarrow{ucp}0$ as $n\to\infty$ with $F^l=F(1)^l$ for each $l=1,2$.

(iv) There exist a c\`adl\`ag $\mathbf{F}^{(0)}$-adapted process $F^{1*2}$ and a constant $\delta>1-\xi'$ such that $b_n^{-\delta}(F(1)^{n,1*2}-F^{1*2})\xrightarrow{ucp}0$ as $n\to\infty$.

\end{enumerate}

\begin{rmk}\label{remH1}
An [H1$^{\sharp}$] type condition appears in \citet{HJY2011} (see assumptions E($q$) and E$'(q)$ of \cite{HJY2011}). The reason why we introduce a kind of exceptional sets $\mathcal{N}^l_n$ $(l=0,1,2,1*2,')$ is that the condition [H1] without them is too local. To explain this, we focus on the univariate case. Note that in this case we have $R^k=S^k$ $(k=0,1,2,\dots)$. Let $\tau$ be a positive number and suppose that $(S^i)$ be a sequence of Poisson arrival times whose intensity is $\underline{\lambda}$ before the time $\tau$ and $\overline{\lambda}$ after $\tau$. Then the structure of the process $G(1)^n$ becomes very complex around the time $\tau$ (of course if $\underline{\lambda}\neq\overline{\lambda}$), so that it will be difficult to verify the convergence $G(1)^n\xrightarrow{ucp}G$ because it requires a kind of uniformity. See also Example \ref{HYmodel}.
\end{rmk}

Next we introduce a kind of continuity of a stochastic process which we mentioned in the introduction.
\begin{Def}
Let $\lambda\in[0,1]$ and let $V$ be a c\`adl\`ag $\mathbf{F}^{(0)}$-adapted process. 
\begin{enumerate}[(i)]

\item $V$ is \textit{of class} (A$_\lambda$) if there is a positive constant $C$ satisfying
\begin{equation*}
E\left[|V_{\tau_1}-V_{\tau_2}|^2\big|\mathcal{F}_{\tau_1\wedge\tau_2}\right]\leq CE\left[|\tau_1-\tau_2|^{1-\lambda}\big|\mathcal{F}_{\tau_1\wedge\tau_2}\right]
\end{equation*}
for any bounded $\mathbf{F}^{(0)}$-stopping times $\tau_1$ and $\tau_2$.

\item $V$ is \textit{of class} (AL$_\lambda$) if there is a sequence $(\sigma_k)$ of $\mathbf{F}^{(0)}$-stopping times such that $\sigma_k\uparrow\infty$ as $k\to\infty$ and the stopped process $V^{\sigma_k}$ is of class (A$_\lambda$) for every $k$.

\end{enumerate}
\end{Def}
If both of processes $V$ and $W$ are of class (AL$_\lambda$) for some $\lambda\in[0,1]$, then the process $V+W$ is obviously of class (AL$_\lambda$). Moreover, the class (AL$_\lambda$) is non-increasing in $\lambda$. That is, if $0\leq\lambda_1\leq\lambda_2\leq 1$ and a process $V$ is of class (AL$_{\lambda_1}$), then V is also of class (AL$_{\lambda_2}$). In fact, if $\tau$ is an $\mathbf{F}^{(0)}$-stopping time such that $V^\tau$ is of class (A$_{\lambda_1}$), then $V^{\tau\wedge K}$ is of class (A$_{\lambda_2}$) for any $K>0$. This implies $V$ is of class (AL$_{\lambda_2}$).

\begin{rmk}\label{remAL}
If a c\`adl\`ag $\mathbf{F}^{(0)}$-adapted process $V$ is of class (AL$_\lambda$) for some $\lambda\in[0,1)$, then evidently $V$ satisfies the Aldous tightness criterion condition (this is why we use the letter ``A'' for the definition). More precisely, for all $K>0$ and $\eta>0$ we have
$\lim_{\theta\downarrow0}\sup_{\sigma,\tau\in\mathcal{T}_K:\sigma\leq\tau\leq\sigma+\theta}P(|V_\tau-V_\sigma|\geq\eta)=0,$
where $\mathcal{T}_K$ denotes the set of all $\mathbf{F}^{(0)}$-stopping times bounded by $K$. This also implies that $V$ is quasi-left continuous (see Remark VI-4.7 of \cite{JS}). 
\end{rmk} 

In the following processes of class (AL$_\lambda$) for any $\lambda\in(0,1]$ play an important role. Here we give some examples of such ones.
\begin{example}\label{BVprocess}
If $B$ is an $\mathbf{F}^{(0)}$-adapted process with a locally integrable variation and the predictable compensator of the variation process of $B$ is absolutely continuous with a locally bounded derivative, then $B$ is of class (AL$_0$).
\end{example}

\begin{example}\label{martingale}
If $L$ is a locally square-integrable martingale on $\mathcal{B}^{(0)}$ and its predictable quadratic variation process is absolutely continuous with a locally bounded derivative, then $L$ is of class (AL$_0$).
\end{example}

\begin{example}\label{holder}
For a real-valued function $x$ on $\mathbb{R}_+$, the \textit{modulus of continuity} on $[0,T]$ is denoted by $w(x;\delta,T)=\sup\{|x(t)-x(s)|;s,t\in[0,T],|s-t|\leq\delta\}$ for $T,\delta>0$. Then, if an $\mathbf{F}^{(0)}$-adapted process $V$ satisfies $w(V;h,t)=O_p(h^{\frac{1}{2}-\lambda})$ as $h\to\infty$ for every $t,\lambda\in(0,\infty)$, then $V$ is of class (AL$_{\lambda}$) for any $\lambda\in(0,1]$. An interesting example of such ones which does not belong to the above examples is a class of  fractional Brownian motions with Hurst indices greater than $1/2$. 
\end{example}

Instead of a kind of strong predictability, we impose the following condition on the sampling times:
\begin{enumerate}

\item[{[H2]}] (i) $S^i$ and $T^i$ are $\mathbf{F}^{(0)}$-predictable times for every $i$.

(ii) The process $G$ in the condition [H1] is of the form
$G_t=V^G_t+\sum_{k=1}^{N^G_t}\gamma^G_k,$
where $V^G$ is of class (AL$_\lambda$) for any $\lambda\in(0,1]$, $N^G$ is an adapted point process and $(\gamma^G_k)$ is a sequence of random variables.

(iii) The process $\chi$ in the condition [H1] is of the form
$\chi_t=V^\chi_t+\sum_{k=1}^{N^\chi_t}\gamma^\chi_k,$
where $V^\chi$ is of class (AL$_\lambda$) for any $\lambda\in(0,1]$, $N^\chi$ is an adapted point process and $(\gamma^\chi_k)$ is a sequence of random variables.

(iv) For each $l=1,2,1*2$, the process $F^l$ in the condition [H1] is of the form
$F^l_t=V^{F^l}_t+\sum_{k=1}^{N^{F^l}_t}\gamma^{F^l}_k,$
where $V^{F^l}$ is of class (AL$_\lambda$) for any $\lambda\in(0,1]$, $N^{F^l}$ is an adapted point process and $(\gamma^{F^l}_k)$ is a sequence of random variables.

\end{enumerate}

\begin{rmk}
(i) We will explain why we need the condition [H2](i) in Remark \ref{techrmk}. This condition is not restricted in the framework of continuous processes because hitting times of continuous adapted processes are predictable. Note that $\widehat{S}^k$, $\widehat{T}^k$ and $R^k$ are also $\mathbf{F}^{(0)}$-predictable times under [H2](i) by Eq.~$(\ref{refreshrep})$.

\noindent (ii) The conditions [H2](ii)-(iv) are also not restricted at least in the univariate case (in the univariate case we have $G=F^1=F^2=F^{1*2}$ and $\chi\equiv1$, so that it is sufficient that [H2](ii) holds). For example, renewal sampling schemes satisfy these conditions because the conditionally expected durations of such schemes are constant. Other examples satisfying [H2] are given in Section \ref{examples}. In particular, sampling times generated by hitting barriers satisfy [H2] (see Example \ref{Exhit}) and in this case the asymptotic skewness of returns do not vanish. We involve terms with finite activity jumps such as $\sum_{k=1}^{N^G_t}\gamma^G_k$ in [H2] to treat sampling schemes as stated in Remark \ref{remH1} (see also Example \ref{HYmodel}).

\noindent (iii) We also remark that in the econometric literature conditionally expected durations are often modeled by GARCH-type models (such as the ACD model of \cite{ER1998}) or SV-type models (such as the SCD model of \cite{BV2004}). Since such models can be approximated by It\^o semimartingales (see \cite{Lindner2009} and references therein), [H2] is also not restricted from the econometric point of view in the light of Example \ref{BVprocess}--\ref{martingale}. 
\end{rmk}

The volatility processes should also have a kind of continuity:
\begin{enumerate}
\item[[{H3]}] For each $V,W=X,Y,\underline{X},\underline{Y}$, $[V,W]$ is absolutely continuous with a c\`adl\`ag derivative, and the density process $[V,W]'$ is of class (AL$_\lambda$) for any $\lambda\in(0,1]$.
\end{enumerate}
In consideration of Example \ref{BVprocess}--\ref{holder}, [H3] is standard in the literature; see e.g., \cite{HJY2011} and \cite{HY2011}.

The maximum of the durations need to have a fairly fast convergence speed.
\begin{enumerate}
\item[{[H4]}] $\frac{5}{6}<\xi'<1$ and $(\ref{A4})$ holds for every $t\in\mathbb{R}_+$.
\end{enumerate}
An [H4]-type condition often appears in the literature (e.g., \cite{Bibinger2012,HY2011,LMRZZ2012}). As naturally expected, this condition have a connection with the condition [H1]. To explain this, we introduce an auxiliary condition. Let $\rho$ be a positive number.
\begin{enumerate}
\item[{[K$_\rho$]}] The sequence of the processes $\left(\sup_{0\leq s\leq t}G(\rho)^n_{s}\right)_{n\in\mathbb{N}}$ is tight as $n\to\infty$ for all $t>0$.
\end{enumerate}

\begin{lem}\label{supGamma}
Suppose that $[\mathrm{H}1](\mathrm{i})$ and $[\mathrm{K}_\rho]$ hold for some $\rho\geq 1$. Then $\sup_{0\leq t\leq T}|\Gamma^{N^n_t+1}|=O_p(b_n^{1-1/\rho})$ as $n\to\infty$ for any $T>0$.
\end{lem}

\begin{proof}
By an argument similar to the proof of Lemma 10.4 of \cite{Koike2012phy}, we can show that $N^n_T=O_p(b_n^{-1})$. Therefore, the Lenglart inequality and [K$_\rho$] yield $\sum_{k=1}^{N^n_T+1}|\Gamma^k|^\rho=O_p(b_n^{\rho-1})$. Since $\left\{\sup_{0\leq t\leq T}|\Gamma^{N^n_t+1}|\right\}^\rho$ $\leq\sum_{k=1}^{N^n_T+1}|\Gamma^k|^\rho$, we complete the proof of the lemma.
\end{proof}

Since $r_n(t)\leq 2\sup_{k}|\Gamma^k(t)|\leq 2\sup_{0\leq s\leq t}|\Gamma^{N^n_s+1}|$, we obtain the following result:
\begin{cor}
$[\mathrm{H}4]$ holds true if $[\mathrm{K}_\rho]$ holds for some $\rho>6$. 
\end{cor} 

We impose the following regularity conditions on the drift processes and the noise process:
\begin{enumerate}
\item[[{H5]}] For each $V=A^X,A^Y,A^{\underline{X}},A^{\underline{Y}}$, $V$ is absolutely continuous with a c\`adl\`ag derivative, and the density process $V'$ is of class (AL$_\lambda$) for some $\lambda\in(0,\frac{1}{2})$.

\item[{[H6]}] $(\int |z|^8Q_t(\mathrm{d}z))_{t\in\mathbb{R}_+}$ is a locally bounded process and the covariance matrix process $\Psi_t(\cdot)=\int zz^*Q_t(\cdot,\mathrm{d}z)$ is c\`adl\`ag. Moreover, for every $i,j=1,2$ the process $\Psi^{ij}$ is of class (AL$_\lambda$) for any $\lambda\in(0,1]$.

\end{enumerate}

\begin{rmk}\label{techrmk}
The condition [H2](i) is necessary by the following technical reason. In the proof we will regard the noise process $(\epsilon^X_{\widehat{S}^i})$ as the (martingale) differences of the purely discontinuous locally square-integrable martingale $\sum_{p=1}^{\infty}\epsilon^X_{\widehat{S}^p}1_{\{\widehat{S}^p\leq t\}}$ on $\mathcal{B}$. Then we need to consider the predictable quadratic variation process (with respect to the filtration $\mathbf{F}$) of this process. Since $\Psi$ is quasi-left continuous under [H6] (see Remark \ref{remAL}), [H2](i) ensures it is given by $\sum_{p=1}^{\infty}\Psi^{11}_{\widehat{S}^p}1_{\{\widehat{S}^p\leq t\}}$. We refer to Chapter I of \cite{JS} for more details on the concepts appearing here.
\end{rmk}

Finally, we introduce constants appearing in the representation of the asymptotic variance of our estimator. For any real-valued bounded measurable functions $\alpha,\beta$ on $\mathbb{R}$, we define the function $\psi_{\alpha,\beta}$ on $\mathbb{R}$ by $\psi_{\alpha,\beta}(x)=\int_0^1\int_{x+u-1}^{x+u+1}\alpha(u)\beta(v)\mathrm{d}v\mathrm{d}u$ for every $x\in\mathbb{R}$. Then, we extend the functions $g$ and $g'$ to the whole real line by setting $g(x)=g'(x)=0$ for $x\notin[0,1]$ and put
\begin{gather*}
\kappa:=\int_{-2}^{2}\psi_{g,g}(x)^2\mathrm{d}x,\qquad
\widetilde{\kappa}:=\int_{-2}^{2}\psi_{g',g'}(x)^2\mathrm{d}x,\qquad
\overline{\kappa}:=\int_{-2}^{2}\psi_{g,g'}(x)^2\mathrm{d}x.
\end{gather*}

\subsection{Results}

Now we are ready to state the main theorem of this article.
\begin{theorem}\label{mainthm}

$(\mathrm{a})$ Suppose $[\mathrm{H}1](\mathrm{i})$--$(\mathrm{iii})$, $[\mathrm{H}2](\mathrm{i})$--$(\mathrm{iii})$ and  $[\mathrm{H}3]$--$[\mathrm{H}6]$ are satisfied. Suppose also that $\underline{X}=\underline{Y}=0$. Then
\begin{equation}\label{CLT}
b_n^{-1/4}\{\widehat{PHY}(\mathsf{X},\mathsf{Y})^n-[X,Y]\}\to^{d_s}\int_0^\cdot w_s\mathrm{d}\widetilde{W}_s\qquad\mathrm{in}\ \mathbb{D}(\mathbb{R}_+)
\end{equation}
as $n\to\infty$, where $\tilde{W}$ is a one-dimensional standard Wiener process (defined on an extension of $\mathcal{B}$) independent of $\mathcal{F}$ and $w$ is given by
\begin{align}
w_s^2=\psi_{HY}^{-4}[&\theta\kappa\{[X]'_s[Y]'_s+([X,Y]'_s)^2\}G_s
+\theta^{-3}\widetilde{\kappa}\{\Psi^{11}_s\Psi^{22}_s+\left(\Psi^{12}_s\chi_s\right)^2\}G_s^{-1}\nonumber\\
&+\theta^{-1}\overline{\kappa}\{[X]'_s\Psi^{22}_s+[Y]'_s\Psi^{11}_s+2[X,Y]'_s\Psi^{12}_s\chi_s\}].\label{avar}
\end{align}

\begin{enumerate}
\item[$(\mathrm{b})$] Suppose $[\mathrm{H}1]$--$[\mathrm{H}6]$ are satisfied. Then $(\ref{CLT})$ holds as $n\to\infty$, where $\tilde{W}$ is as in the above and $w$ is given by
\begin{align}
w_s^2=\psi_{HY}^{-4}\bigg[&\theta\kappa\left\{[X]'_s[Y]'_s+([X,Y]'_s)^2\right\}G_s
+\theta^{-3}\widetilde{\kappa}\left\{\overline{\Psi}^{11}_s\overline{\Psi}^{22}_s+\left(\overline{\Psi}^{12}_s\right)^2\right\}G_s^{-1}\nonumber\\
&+\theta^{-1}\overline{\kappa}\left\{[X]'_s\overline{\Psi}^{22}_s+[Y]'_s\overline{\Psi}^{11}_s+2[X,Y]'_s\overline{\Psi}^{12}_s-\left([\underline{X},Y]'_s F^1_s-[X,\underline{Y}]'_s F^2_s\right)^2 G_s^{-1}\right\}\Bigg]\label{avarend}
\end{align}
with $\overline{\Psi}^{11}_s=\Psi^{11}_s+[\underline{X}]'_s F^1_s$, $\overline{\Psi}^{22}_s=\Psi^{22}_s+[\underline{Y}]'_s F^2_s$ and $\overline{\Psi}^{12}_s=\Psi^{12}_s\chi_s+[\underline{X},\underline{Y}]'_s F^{1* 2}_s$.
\end{enumerate} 
\end{theorem}

Proof of this theorem is given in Appendix \ref{proofmainthm}. As was announced in the introduction, the time endogeneity has no impact on the asymptotic distribution of the pre-averaged Hayashi-Yoshida estimator, compared with Theorem 3.1 of \cite{Koike2012phy}. It is also worth noting that the endogeneity of the noise pushes down the asymptotic variance.

In the univariate case, we have $S^k=T^k=R^k$ for all $k$, so that $$\widehat{PHY}(\mathsf{X},\mathsf{X})^n_t=\frac{1}{(\psi_{HY}k_n)^2}\sum_{i,j:|i-j|<k_n}\overline{X}_g(\mathcal{I})^i\overline{X}_g(\mathcal{I})^j$$
for each $t\in\mathbb{R}_+$. Moreover, [H1](i)--(ii) and [H2](ii) implies that [H1](iv)--(v) and [H2](iv) respectively, and [H1](iii) and [H2](iii) are automatically satisfied because $\chi^n\equiv1$. Consequently, we obtain the following result:

\begin{cor}
Suppose $[\mathrm{H}1](\mathrm{i})$--$(\mathrm{ii})$, $[\mathrm{H}2](\mathrm{i})$--$(\mathrm{ii})$ and $[\mathrm{H}3]$--$[\mathrm{H}6]$ are satisfied with taking $S^k=R^k$ for every $k$. Then
\begin{equation*}
b_n^{-1/4}\{\widehat{PHY}(\mathsf{X},\mathsf{X})^n-[X]\}\to^{d_s}\int_0^\cdot w_s\mathrm{d}\widetilde{W}_s\qquad\mathrm{in}\ \mathbb{D}(\mathbb{R}_+)
\end{equation*}
as $n\to\infty$, where $\tilde{W}$ is as in the above and $w$ is given by
\begin{align*}
w_s^2=\frac{2}{\psi_{HY}^{4}}\left[\theta\kappa([X]'_s)^2 G_s
+\theta^{-3}\widetilde{\kappa}\left(\overline{\Psi}^{11}_s\right)^2\frac{1}{G_s}+2\theta^{-1}\overline{\kappa}[X]'_s\overline{\Psi}^{11}_s\right].
\end{align*}
\end{cor}

Interestingly, both of the endogeneity of the sampling times and the noise have no impact on the asymptotic distribution.

\begin{rmk}
(i) A brief explanation of the reason why the asymptotic skewness of returns has no impact on the asymptotic variance of the PHY can be given in the following way. For simplicity we focus on the univariate case without the noise and drift. Then, the predictable quadratic covariation of the estimation error of the PHY and the martingale $X$ is given by the sum of terms like $\sum_{p=0}^{k_n-1}g(\frac{p}{k_n})[X](I^{i+p})\sum_{q=0}^{k_n-1}g(\frac{q}{k_n})X(I^{j+q})$ with $|i-j|<k_n$. In such a term, variables corresponding to the third power of returns (i.e., terms involving variables like $[X](I^k) X(I^k)$) have no impact in the first order. By a similar reason the asymptotic kurtosis of returns also has no impact on the asymptotic variance of the PHY.

\noindent (ii) Due to Lemma 3.1 of \cite{Koike2012phy}, in the estimation error of the PHY we can replace the (pre-averaging version of) Hayashi-Yoshida type sampling design kernel $1_{\{[\widehat{S}^i,\widehat{S}^{i+k_n})\cap[\widehat{T}^j,\widehat{T}^{j+k_n})\neq\emptyset\}}$ by a certain deterministic function. This enables us to handle the nonsynchronous case with no difficulty. This is quite different from the case for the Hayashi-Yoshida estimator in a pure diffusion setting, in which the Hayashi-Yoshida sampling design kernel plays a central role in the first order calculus.   
\end{rmk}

%% file: hitting/hitting_examples.tex

\section{Examples}\label{examples}

\subsection{Univariate case}

\begin{example}[Times generated by hitting barriers]\label{Exhit}
This example was treated in Section 4.4 of \cite{Fu2010b} and Example 4 of \cite{LMRZZ2012}.

Suppose that [H3] is satisfied and both $[X]'$ and $[X]'_{-}$ do not vanish. Define
\begin{equation}\label{defhit}
S^0=0,\qquad S^{i+1}=\inf\left\{t>S^{i}|M^X_t-M^X_{S^i}=-u\sqrt{b_n}\textrm{ or }M^X_t-M^X_{S^i}= v\sqrt{b_n}\right\}
\end{equation}
for positive constants $u,v$. Then, using a representation of a continuous local martingale with Brownian motion, we have
\begin{align*}
P\left( M^X_{S^{i+1}}-M^X_{S^i}=-u\sqrt{b_n}\right)=v/(u+v),\qquad
P\left( M^X_{S^{i+1}}-M^X_{S^i}=v\sqrt{b_n}\right)=u/(u+v).
\end{align*}
Combining the above formula with Proposition 2.1 of \cite{Obloj2004} (again using a representation of a continuous local martingale with Brownian motion), we obtain the following result: for each $r\geq1$ there exists a positive constant $C_r$ such that
$E\left[\left|[X]_{S^{i+1}}-[X]_{S^i}\right|^r\right]\leq C_r b_n^r$
for every $n,i$. In particular, this inequality yields [K$_\rho$] holds for any $\rho>1$ because $\inf_{0\leq s\leq t}[X]'_s>0$ for any $t>0$. Therefore, $(\ref{A4})$ holds for any $\xi'\in(0,1)$ by Lemma \ref{supGamma}. Noting that these results and the condition [H3], we can also show that [H1] holds with $\mathbf{H}^n=\mathbf{F}^{(0)}$ and $G_s=uv/[X]_s$. This result also implies that [H2](ii) holds true. Finally, [H2](i) is also satisfied because $M^X$ is continuous.

Note that in this example the asymptotic skewness of the returns does not vanish if $u\neq v$.  
\end{example}

\begin{rmk}\label{Girsanov}
In Example \ref{Exhit}, the stable convergence results of Theorem \ref{mainthm} still hold when we replace $M^X$ in $(\ref{defhit})$ by $X$. This can be shown by the following way: first, by a localization argument it is sufficient to consider processes stopped at some positive number $T$. Let $Z_t=\exp\left(\int_0^t(A^X)'_s/[X]_s\mathrm{d} M^X_s-\frac{1}{2}A^X_t\right)$.
As is well known, $Z_t$ is a positive continuous local martingale. Therefore, again by a localization argument we may assume that both $Z$ and $1/Z$ are bounded. In particular, $Z$ is a martingale, so that we can define a probability measure $\widetilde{P}_T$ on $(\Omega,\mathcal{F}_T)$ by $\widetilde{P}_T(E)=P(1_E Z_T)$. $\widetilde{P}_T$ is obviously equivalent to the probability measure $P$ restricted to $(\Omega,\mathcal{F}_T)$. Then, by Girsanov's theorem $X$ is a continuous $\mathbf{F}^{(0)}$-local martingale under $\widetilde{P}_T$, hence [H1], [H2] and [H4] hold true under $\widetilde{P}_T$. Moreover, [H3] and [H5]--[H6] are also satisfied under $\widetilde{P}_T$ due to the Bayes rule. Therefore, $(\ref{CLT})$ holds true under $\widetilde{P}_T$. Since the stable convergence is stable by equivalent changes of probability measures, $(\ref{CLT})$ also holds true under the original probability measure $P$. Further, in this case we do not need $(A^X)'$ is of class (AL$_\lambda$) for some $\lambda\in(0,\frac{1}{4})$.
\end{rmk}

\begin{example}[General return distribution]
This example was considered in Section 4.3 of \cite{Fu2009} and Example 5 of \cite{LMRZZ2012}, and can be regarded as a generalization of Example \ref{Exhit}.

Let $W$ be a one-dimensional standard Wiener process on a stochastic basis $(\Omega',\mathcal{F}',(\mathcal{F}'_t),P')$. Suppose that $\Psi$ is adapted to the filtration $(\mathcal{F}'_t)$. Let $\mu$ be a probability measure on $\mathbb{R}$ with mean 0, and suppose that $\mu$ is not a Dirac measure i.e., $\mu(\{0\})<1$. Then, by Lemma 108 in Chapter 1 of \cite{Freedman1983} we can construct an i.i.d.~random vectors $(U_0,V_0),(U_1,V_1),\dots$ on a probability space $(\Omega'',\mathcal{F}'',P'')$ satisfying the following conditions for every $i$:
\begin{enumerate}[(i)]

\item $U_i,V_i>0$ a.s.,

\item For any $x\in\mathbb{R}$
$\mu((-\infty,x])=\int_{\Omega''}G_{U_i(\omega''),V_i(\omega'')}(x)P''(\mathrm{d}\omega''),$
where for $u,v>0$ $G_{u,v}$ is the distribution function of the random variable $\zeta$ such that $P(\zeta=-u)=1-P(\zeta=v)=v/(u+v)$.

\end{enumerate}
Now construct the stochastic basis $\mathcal{B}^{(0)}$ by
\begin{equation}\label{sbasis}
\Omega^{(0)}=\Omega'\times\Omega'',\qquad
\mathcal{F}^{(0)}=\mathcal{F}'\otimes\mathcal{F}'',\qquad
\mathcal{F}^{(0)}_t=\mathcal{F}'_t\otimes\mathcal{F}'',\qquad
P^{(0)}=P'\times P''.
\end{equation}
Then, we define $(S^i)$ sequentially by $S^0=0$ and
\begin{equation*}
S^{i+1}=\inf\left\{t>S^{i}|W_t-W_{S^i}=-U_i\sqrt{b_n}\textrm{ or }W_t-W_{S^i}=V_i\sqrt{b_n}\right\}\qquad i=0,1,\dots.
\end{equation*}
By construction $S^i$ is an $\mathbf{F}^{(0)}$-predictable time for every $i$ and $(W_{S^{i+1}}-W_{S^i})_{i\in\mathbb{Z}_+}$ is a sequence of independent random variables. Furthermore, Lemma 115 in Chapter 1 of \cite{Freedman1983} implies that $(\ref{increase})$, $b_n^{-1/2}(W_{S^{i+1}}-W_{S^i})\sim\mu$ and $E[S^{i+1}-S^i]=b_n\int_{\mathbb{R}}x^2\mu(\mathrm{d}x)$. This is known as the \textit{Skorohod representation}, which is closely related to the so-called \textit{Skorohod stopping problem} (see \cite{Obloj2004} for details). In the present situation we need the predictability of $S^i$, so that we give the precise construction of $S^i$.

Now we verify the conditions [H1]--[H2] and [H4]. For this purpose we need to assume that
$\int_{\mathbb{R}}|x|^k\mu(\mathrm{d}x)<\infty$
for some $k>12$. Then, Proposition 2.1 of \cite{Obloj2004} yields [K$_{k/2}$], so that [H1](ii) holds. Moreover, [H4] is also satisfied by Lemma \ref{supGamma}. On the other hand, letting $\mathbf{H}^n$ being the filtration generated by the processes  $W_t,\Psi_t$ and the process $\sum_i1_{\{S^i\leq t\}}$, [H1] holds true with $G_s\equiv\int_{\mathbb{R}}x^2\mu(\mathrm{d}x)$. Thus [H2] is also satisfied.
\end{example}

\begin{example}[Dynamic Mixed Hitting-Time Model]\label{DMHT}
This model was introduced in \citet{RHW2012} and also discussed in Example 6 of \cite{LMRZZ2012}. 

First we construct the stochastic basis $\mathcal{B}^{(0)}$ which is appropriate for the present situation. Let $(\Omega',\mathcal{F}',(\mathcal{F}'_t)$, $P')$ be a stochastic basis, and suppose that the semimartingales $X$ and $\underline{X}$ are defined on this basis. Suppose also that $\Psi$ is $(\mathcal{F}'_t)$-adapted. Moreover, suppose that there exist a one-dimensional standard Wiener process $W$ and two positive c\`adl\`ag adapted processes $\mu$ and $c$ on $(\Omega',\mathcal{F}',(\mathcal{F}'_t)$, $P')$. On the other hand, let $(\zeta_i)_{i\in\mathbb{Z}_+}$ be positive i.i.d.~random variables with mean 1 on an auxiliary probability space $(\Omega'',\mathcal{F}'',P'')$. Then we define the stochastic basis $\mathcal{B}^{(0)}$ by $(\ref{sbasis})$.

We define the sampling scheme $(S^i)$ sequentially by $S^0=0$ and
\begin{equation*}
S^{i+1}=\inf\left\{t>S^{i}|W_t-W_{S^i}+b_n^{-1/2}\mu_{S^i}(t-S^i)=b_n^{1/2}c_{S^i}\zeta_i\right\}\qquad i=0,1,\dots.
\end{equation*}
By construction $S^i$ is an $\mathbf{F}^{(0)}$-predictable time for every $i$. Moreover, the conditional distribution of $S^{i+1}-S^i$ given $\mathcal{F}^{(0)}_{S^i}$ is the inverse Gaussian distribution $IG(b_n^{1/2}c_{S^i}\zeta_i,b_n^{-1/2}\mu_{S^i})$, where the probability density function of the inverse Gaussian distribution $IG(\delta,\gamma)$ is given by
\begin{align*}
p(z;\delta,\gamma)=\frac{\delta e^{\delta\gamma}}{\sqrt{2\pi}}z^{-3/2}\exp\left\{-\frac{1}{2}\left(\frac{\delta^2}{z}+\gamma^2 z\right)\right\},\qquad z>0.
\end{align*}
In order to verify the conditions [H1]--[H2] and [H4], we additionally make the following assumptions: both $c_-$ and $\mu_-$ do not vanish, $\psi:=c/\mu$ is of class (AL$_\lambda$) for any $\lambda>0$ and $E[|\zeta_i|^\rho]<\infty$ for some $\rho>6$. Then, letting $\mathcal{H}^n_t$ being the $\sigma$-field generated by $\mathcal{F}'_t$ and the random variable $\sum_i1_{\{S^i\leq t\}}$ for each $t\in\mathbb{R}_+$, we have
\begin{align*}
G(\rho)^n_{S^i}=E\left[\frac{K_{\rho-1/2}(c_{S^i}\mu_{S^i}\zeta_i)}{K_{-1/2}(c_{S^i}\mu_{S^i}\zeta_i)}\psi_{S^i}^\rho\zeta^\rho_i\big|\mathcal{H}^n_{S^i}\right],\qquad i=0,1,\dots
\end{align*}
by Eq.~(2.16) of \cite{GIG1982}, where $K_\lambda$ is the modified Bessel function of the third kind and with index $\lambda$. Now we notice that the following properties of the function $K_\lambda$. First, Theorem 1.2 of \cite{LN2010} implies that for any $\lambda>0$ there exist a positive constant $C_\lambda$ such that $K_\lambda(x)/K_{\lambda-1}(x)<C_\lambda(1+x^{-1})$ for any $x>0$. Second, for any $x>0$, $K_\lambda(x)$ is strictly increasing in $\lambda$ for $\lambda>0$. This follows form Eq.~(2.12) of \cite{LN2010}. These facts yields the condition [K$_\rho$], hence Lemma \ref{supGamma} implies that [H4] holds true. Moreover, by construction $\mathcal{H}^n_{S^i}$ is independent of $\zeta_i$, hence we have [H1] with $G=\psi$. From this we also obtain [H2]. From this model we can obtain endogenous sampling times by giving a correlation between $X$ and $W$.
\end{example}

\subsection{Nonsynchronous case}

\begin{example}[Poisson sampling with a random change point]\label{HYmodel}
This example is a version of the model discussed in Section 8.3 of \citet{HY2011}.

As in the preceding example, we first construct an appropriate stochastic basis $\mathcal{B}^{(0)}$. Let $(\Omega',\mathcal{F}',(\mathcal{F}'_t)$, $P')$ be a stochastic basis, and suppose that the semimartingales $X$, $Y$, $\underline{X}$ and $\underline{Y}$ are defined on this basis. Suppose also that $\Psi$ is $(\mathcal{F}'_t)$-adapted. Furthermore, on an auxiliary probability space $(\Omega'',\mathcal{F}'',P'')$, there are mutually independent standard Poisson processes $(\underline{N}^k_t)$, $(\overline{N}^k_t)$ $(k=1,2)$. Then we construct $\mathcal{B}^{(0)}$ by $(\ref{sbasis})$.

Next we construct our sampling schemes. For each $k=1,2$, let $\underline{p}^k,\overline{p}^k\in(0,\infty)$ and let $\tau^k$ be an $(\mathcal{F}'_t)$-stopping time. Define $(\underline{S}^i)$ and $(\overline{S}^i)$ each as the arrival times of the point processes $\underline{N}^{n,1}=(\underline{N}^1_{n\underline{p}^1t})$ and $\overline{N}^{n,1}=(\overline{N}^1_{n\overline{p}^1t})$ respectively. Then, we define $(S^i)$ sequentially by $S^0=0$ and
\begin{align*}
S^i=\inf_{l,m\in\mathbb{N}}\left\{\underline{S}^l_{\{S^{i-1}<\underline{S}^l<\tau^1\}},(\tau^1+\overline{S}^m)_{\{S^{i-1}<\tau^1+\overline{S}^m\}}\right\},\qquad i=1,2,\dots.
\end{align*}
$(T^j)$ is defined in the same way using $\underline{N}^{n,2}=(\underline{N}^2_{n\underline{p}^2t})$, $\overline{N}^{n,2}=(\overline{N}^2_{n\overline{p}^2t})$ and $\tau^2$ instead of $\underline{N}^{n,1}$, $\overline{N}^{n,1}$ and $\tau^1$ respectively.

Let $\mathbf{H}^n$ be the filtration generated by the $\sigma$-field $\mathcal{F}'$ and the processes $N^{n,1},N^{n,2}$. Then, in a similar manner to Section 5.2 of \cite{Koike2012phy} we can show that [H1] and [H4] are satisfied with $b_n=n^{-1}$, $\chi\equiv0$ and
\begin{align*}
&G_s=\left(\frac{1}{\underline{p}^1}+\frac{1}{\underline{p}^2}-\frac{1}{\underline{p}^1+\underline{p}^2}\right)1_{\{s<\tau^1\wedge\tau^2\}}
+\left(\frac{1}{\overline{p}^1}+\frac{1}{\underline{p}^2}-\frac{1}{\overline{p}^1+\underline{p}^2}\right)1_{\{\tau^1\leq s<\tau^2\}}+\nonumber\\
&\hphantom{G_s=(}\left(\frac{1}{\underline{p}^1}+\frac{1}{\overline{p}^2}-\frac{1}{\underline{p}^1+\overline{p}^2}\right)1_{\{\tau^2\leq s<\tau^1\}}
+\left(\frac{1}{\overline{p}^1}+\frac{1}{\overline{p}^2}-\frac{1}{\overline{p}^1+\overline{p}^2}\right)1_{\{\tau^1\vee\tau^2\leq s\}}\\
\intertext{and}
&\begin{array}{l}\displaystyle
F^1_s=\frac{1}{\underline{p}^1}1_{\{s<\tau^1\}}+\frac{1}{\overline{p}^1}1_{\{\tau^1\leq s\}},\qquad\qquad
F^2_s=\frac{1}{\underline{p}^2}1_{\{s<\tau^2\}}+\frac{1}{\overline{p}^2}1_{\{\tau^2\leq s\}},\\ 
\displaystyle F^{1*2}_s=\frac{2}{\underline{p}^1+\underline{p}^2}1_{\{s<\tau^1\wedge\tau^2\}}
+\frac{2}{\overline{p}^1+\underline{p}^2}1_{\{\tau^1\leq s<\tau^2\}}
+\frac{2}{\underline{p}^1+\overline{p}^2}1_{\{\tau^2\leq s<\tau^1\}}
+\frac{2}{\overline{p}^1+\overline{p}^2}1_{\{\tau^1\vee\tau^2\leq s\}}.
\end{array}
\end{align*}
From these formulae [H2](ii)--(iv) also holds true. Finally, [H2](i) is also satisfied because $\underline{S}^l$, $\overline{S}^l$, $\underline{T}^l$ and $\overline{T}^l$ are $\mathbf{F}^{(0)}$-predictable times for all $l$. 
\end{example}

\begin{example}[Continuous-time analog of the Lo-MacKinlay model]
In \citet{LM1990} they regarded the nonsynchronicity of observations as a kind of missing observations. That is, they first considered a series $\mathbf{r}_1,\mathbf{r}_2,\dots$ of completely synchronous latent returns. Here, the random vector $\mathbf{r}_t=(r_t^1,\dots,r_t^d)$ represents the vector constituted of the latent returns of $d$ assets for the $t$-th period. In each period $t$ there is some chance that the transaction of the $k$-th asset does not occur with certain probability $p^k$. If it does not occur, the observed return $(r^k_t)^o$ of the $k$-the asset for the $t$-th period is simply 0. On the other hand, if its transaction occurs in the $t$-th period, $(r^k_t)^o$ becomes the sum of its latent returns for all past conservative periods in which its transaction has not occurred. Mathematically speaking, we have $(r^k_t)^o=\sum_{i=0}^\infty X^k_t(i)r^k_{t-i}$, where $X^k_t(i)=(1-\delta^k_t)\prod_{j=1}^i\delta^k_{t-j}$ and $(\delta_t)$ is an i.i.d.~sequence of Bernoulli variables with probabilities $p^k$ and $1-p^k$ of taking values 1 and 0, which are independent of the latent returns. In the following we consider a continuous-time analog of this model.

Let $(\Omega',\mathcal{F}',(\mathcal{F}'_t)$, $P')$ be a stochastic basis as in the preceding example. Suppose that there exists a sequence $(\tau_m)_{m\in\mathbb{Z}_+}$ of $(\mathcal{F}'_t)$-predictable times (which will depend on $n$) such that $\tau_0=0$ and $\tau_m\uparrow\infty$ as $m\to\infty$. On the other hand, suppose that we have two sequences $(\delta^1_m)_{m\in\mathbb{Z}_+}$ and $(\delta^2_m)_{m\in\mathbb{Z}_+}$ of i.i.d.~random variables on a probability space $(\Omega'',\mathcal{F}'',P'')$. Suppose also that they are mutually independent and $P(\delta^k_m=1)=1-P(\delta^k_m=0)=p^k\in[0,1)$ for each $k=1,2$ and $m\in\mathbb{Z}_+$. Then we define the stochastic basis $\mathcal{B}^{(0)}$ by $(\ref{sbasis})$.

Let $M(i)^k=\min\{M|\sum_{m=0}^M(1-\delta^k_m)=i\}$. Then we define $(S^i)$ and $(T^j)$ by $S^i=\tau_{M(i)^1}$ and $T^j=\tau_{M(j)^2}$. By construction $S^i$ and $T^j$ are $\mathbf{F}^{(0)}$-predictable times and satisfy $(\ref{increase})$. For example, if we take $\tau_m=mb_n$, then $(S^i)$ and $(T^j)$ becomes mutually independent Bernoulli sampling schemes (i.e. discretized Poisson sampling schemes). In this case it can be easily shown that [H1]--[H2] and [H4] holds with $G_s=(1-p^1)^{-1}+(1-p^2)^{-1}-(1-p^1p^2)^{-1}$, $\chi_s=(1-p^1)(1-p^2)$, $F^1_s=(1-p^1)^{-1}$, $F^2_s=(1-p^2)^{-1}$ and $F^{1*2}_s=(2-(1-p^1)(1-p^2))(1-p^1p^2)^{-1}$.

By including endogeneity in $(\tau_m)$, we can also obtain endogenous sampling times. For example, let $W$ be a one-dimensional Wiener process on $(\Omega',\mathcal{F}',(\mathcal{F}'_t)$, $P')$ and let $(\zeta_m)_{m\in\mathbb{Z}_+}$ be a sequence of i.i.d.~positive random variables on $(\Omega'',\mathcal{F}'',P'')$. We assume they are independent of $((\delta^1_m,\delta^2_m))_{m\in\mathbb{Z}_+}$ and have mean 1. Further, suppose that $E[|\zeta_i|^\rho]<\infty$ for some $\rho>6$. Then we define $(\tau_m)$ sequentially by $\tau_0=0$ and $\tau_{m+1}=\inf\{t>\tau_m|W_t-W_{\tau_m}+b_n^{-1/2}\mu(t-\tau_m)=b_n^{1/2}c\zeta_m\}$ $(m=0,1,\dots)$. Here, $\mu$ and $c$ are some positive constants. This is a simple mixed hitting-time model considered in Example \ref{DMHT}. Then, by a simple computation similar to that of Example \ref{DMHT} we can show that [H1]--[H2] and [H4] holds with
\begin{gather*}
G_s=\left(\frac{1}{1-p^1}+\frac{1}{1-p^2}-\frac{1}{1-p^1p^2}\right)\psi,\qquad
\chi_s=(1-p^1)(1-p^2),\\
F^1_s=\frac{\psi}{1-p^1},\qquad
F^2_s=\frac{\psi}{1-p^2},\qquad
F^{1*2}_s=\frac{2-(1-p^1)(1-p^2)}{1-p^1p^2}\psi,
\end{gather*}
where $\psi=c/\mu$.
\end{example}

%% file: hitting/hitting_application.tex
\section{Application and discussion}\label{application}

\subsection{Asymptotic variance estimation}\label{avarest}

The central limit theorem derived in Section \ref{main} is infeasible in the sense that the asymptotic variance of the estimation error is unobservable. In order to derive a feasible central limit theorem, we therefore need an estimator for the asymptotic variance. In this subsection we implement this with a \textit{kernel-based approach} as in Section 8.2 of \cite{HY2011} and the second estimator in Section 4 of \cite{CPV2013}.

For this purpose we need to construct global estimators for (i) the integrated (co-)volatility processes $[X]$, $[Y]$ and $[X,Y]$, (ii) the covariance matrix process $\overline{\Psi}$ of the noise process and (iii) the asymptotic variance process $\int_0^\cdot\left([\underline{X},Y]'_s F^1_s-[X,\underline{Y}]'_s F^2_s\right)^2G_s^{-1}\mathrm{d}s$ due to the presence of the endogenous noise. We have already established a class of estimators for the case (i) i.e., the PHY, so that in the following we construct estimators for the cases (ii) and (iii).

For the case (ii) we know several consistent estimators when the noise process is i.i.d. One of the most familiar estimators in such ones is the re-scaled realized covariance: $\sum_{k:R^k\leq t}(\mathsf{X}_{\widehat{S}^k}-\mathsf{X}_{\widehat{S}^{k-1}})(\mathsf{Y}_{\widehat{T}^k}-\mathsf{Y}_{\widehat{T}^{k-1}})/(2N^n_t)$. In the present situation, however, this estimator is not appropriate because $\mathsf{X}_{\widehat{S}^k}$ and $\mathsf{Y}_{\widehat{T}^{k-1}}$ are possibly correlated due to the endogenous noise, for instance. Instead, we use the (scaled) symmetric first-order realized autocovariance estimator proposed by \citet{Oomen2006}:
\begin{align*}
\gamma^n_t(1)^{12}=-\frac{1}{2 k_n^2}\sum_{k:R^{k+1}\leq t}\left\{(\mathsf{X}_{\widehat{S}^k}-\mathsf{X}_{\widehat{S}^{k-1}})(\mathsf{Y}_{\widehat{T}^{k+1}}-\mathsf{Y}_{\widehat{T}^{k}})+(\mathsf{X}_{\widehat{S}^{k+1}}-\mathsf{X}_{\widehat{S}^{k}})(\mathsf{Y}_{\widehat{T}^{k}}-\mathsf{Y}_{\widehat{T}^{k-1}})\right\},\qquad t\in\mathbb{R}_+.
\end{align*}
Similarly we define
\begin{align*}
\gamma^n_t(1)^{11}=-\frac{1}{k_n^2}\sum_{k:\widehat{S}^{k+1}\leq t}(\mathsf{X}_{\widehat{S}^k}-\mathsf{X}_{\widehat{S}^{k-1}})(\mathsf{X}_{\widehat{S}^{k+1}}-\mathsf{X}_{\widehat{S}^{k}}),\quad
\gamma^n_t(1)^{22}=-\frac{1}{k_n^2}\sum_{k:\widehat{T}^{k+1}\leq t}(\mathsf{Y}_{\widehat{T}^k}-\mathsf{Y}_{\widehat{T}^{k-1}})(\mathsf{Y}_{\widehat{T}^{k+1}}-\mathsf{Y}_{\widehat{T}^{k}})
\end{align*}
for each $t\in\mathbb{R}_+$. For the case (iii), we use a pre-averaging based estimator. First, for measurable bounded functions $\alpha,\beta$ on $\mathbb{R}$ we introduce the following quantity:
\begin{align*}
\Xi_{\alpha,\beta}(\mathsf{X},\mathsf{Y})^n_t=\frac{1}{k_n}\sum_{i:R^i\leq t}\overline{\mathsf{X}}_\alpha(\widehat{\mathcal{I}})^i\overline{\mathsf{Y}}_\beta(\widehat{\mathcal{J}})^i,\qquad t\in\mathbb{R}_+.
\end{align*}
Next, fix a $C^2$ function $f$ on $[0,1]$ satisfying $f(0)=f(1)=f'(0)=f'(1)=0$ (e.g., $f(x)=x^2(1-x)^2$ for $x\in[0,1]$) and extend it to the whole real line by setting $f(x)=0$ for $x\notin[0,1]$. Then define
$\Xi[f]^n=\{\Xi_{f',f}(\mathsf{X},\mathsf{Y})^n-\Xi_{f,f'}(\mathsf{X},\mathsf{Y})^n\}/(2\|f'\|^2_2)$, where $\|f'\|^2_2=\int_0^1f'(x)^2\mathrm{d}x$.

Now we construct local estimators for the quantities appearing in the asymptotic variance $(\ref{avarend})$ from the global ones introduced in the above. Let $(h_n)_{n\in\mathbb{N}}$ be a sequence of positive numbers tending to 0 as $n\to\infty$, and define
\begin{align*}
&\widehat{[X]'}_s=h_n^{-1}\left(\widehat{PHY}(\mathsf{X},\mathsf{X})^n_s-\widehat{PHY}(\mathsf{X}, \mathsf{X})^n_{(s-h_n)_+}\right),\qquad
\widehat{[Y]'}_s=h_n^{-1}\left(\widehat{PHY}(\mathsf{Y},\mathsf{Y})^n_s-\widehat{PHY}(\mathsf{Y}, \mathsf{Y})^n_{(s-h_n)_+}\right),\\
&\widehat{[X,Y]'}_s=h_n^{-1}\left(\widehat{PHY}(\mathsf{X},\mathsf{Y})^n_s-\widehat{PHY}(\mathsf{X}, \mathsf{Y})^n_{(s-h_n)_+}\right),\qquad
\partial\gamma^n_s(1)^{11}=h_n^{-1}\left\{\gamma^n_s(1)^{11}-\gamma^n_{(s-h_n)_+}(1)^{11}\right\},\\
&\partial\gamma^n_s(1)^{22}=h_n^{-1}\left\{\gamma^n_s(1)^{22}-\gamma^n_{(s-h_n)_+}(1)^{22}\right\},\qquad
\partial\gamma^n_s(1)^{12}=h_n^{-1}\left\{\gamma^n_s(1)^{12}-\gamma^n_{(s-h_n)_+}(1)^{12}\right\},\\
&\partial\Xi[f]^n_s=h_n^{-1}\left(\Xi[f]^n_s-\Xi[f]^n_{(s-h_n)_+}\right)
\end{align*}
for each $s\in\mathbb{R}_+$. Then, we obtain the following result:

\begin{lem}\label{localest}
Suppose $[\mathrm{H}1]$--$[\mathrm{H}6]$ and $[\mathrm{K}_{4/3}]$ are satisfied. Suppose also that $h_n^{-1}b_n^{1/4}\to0$ as $n\to\infty$. Then
\begin{enumerate}[\normalfont (a)]

\item $\widehat{[X]'}_s\to^p [X]'_{s-}$, $\widehat{[Y]'}_s\to^p [Y]'_{s-}$ and $\widehat{[X,Y]'}_s\to^p [X,Y]'_{s-}$ as $n\to\infty$ for any $s\in\mathbb{R}_+$. Furthermore, $\sup_{0\leq s\leq t}|\widehat{[X]'}_s|$, $\sup_{0\leq s\leq t}|\widehat{[Y]'}_s|$ and $\sup_{0\leq s\leq t}|\widehat{[X,Y]'}_s|$ are tight for any $t>0$.

\item $\partial\gamma^n_s(1)^{11}\to^p\theta^{-2}\overline{\Psi}^{11}_{s-}/G_{s-}$, $\partial\gamma^n_s(1)^{22}\to^p\theta^{-2}\overline{\Psi}^{22}_{s-}/G_{s-}$, and $\partial\gamma^n_s(1)^{12}\to^p\theta^{-2}\overline{\Psi}^{12}_{s-}/G_{s-}$ as $n\to\infty$ for any $s\in\mathbb{R}_+$. Furthermore, $\sup_{0\leq s\leq t}|\partial\gamma^n_s(1)^{11}|$, $\sup_{0\leq s\leq t}|\partial\gamma^n_s(1)^{22}|$ and $\sup_{0\leq s\leq t}|\partial\gamma^n_s(1)^{12}|$ are tight for any $t>0$.

\item $\partial\Xi[f]^n_s\to^p([\underline{X},Y]'_{s-}F^1_{s-}-[X,\underline{Y}]_{s-}F^2_{s-})/\theta G_{s-}$ as $n\to\infty$ for any $s\in\mathbb{R}_+$. Furthermore, $\sup_{0\leq s\leq t}|\partial\Xi[f]^n_s|$ is tight for any $t>0$.

\end{enumerate}
\end{lem}

We give a proof of Lemma \ref{localest} in Appendix \ref{prooflocalest}. According to the above proposition, we can construct a kernel-based estimator for the asymptotic variance as follows. Set
\begin{align}
\widehat{w^2}_{R^k}
&=k_n\psi_{HY}^{-4}\left[\kappa\left\{\widehat{[X]'}_{R^{k}}\widehat{[Y]'}_{R^{k}}+\left(\widehat{[X,Y]'}_{R^{k}}\right)^2\right\}
+\widetilde{\kappa}\left\{\partial\gamma^n_{R^{k}}(1)^{11}\partial\gamma^n_{R^{k}}(1)^{22}+\left(\partial\gamma^n_{R^{k}}(1)^{12}\right)^2\right\}\right.\nonumber\\
&\hphantom{=k_n\psi_{HY}^{-4}[}\left.+\overline{\kappa}\left\{\widehat{[X]'}_{R^{k}}\partial\gamma^n_{R^{k}}(1)^{22}+\widehat{[Y]'}_{R^{k}}\partial\gamma^n_{R^{k}}(1)^{11}+2\widehat{[X,Y]'}_{R^{k}}\partial\gamma^n_{R^{k}}(1)^{12}-\left(\partial\Xi[f]^n_{R^k}\right)^2\right\}\right]|\Gamma^{k}||\Gamma^{k+1}|\label{what}
\end{align}
for each $k\in\mathbb{N}$, and define $\widehat{\mathbf{AVAR}}^n_t=\sum_{k:R^k\leq t}\widehat{w^2}_{R^k}$ for every $t\in\mathbb{R}_+$. Then we obtain the following result:

\begin{theorem}\label{propavar}
Suppose $[\mathrm{H}1]$--$[\mathrm{H}6]$ and $[\mathrm{K}_{4/3}]$ are satisfied. Suppose also that $h_n^{-1}b_n^{1/4}\to0$ as $n\to\infty$. Then we have $b_n^{-1/2}\widehat{\mathbf{AVAR}}^n\xrightarrow{ucp}\int_0^\cdot w_s^2\mathrm{d}s$ as $n\to\infty$, where $w_s$ is defined by $(\ref{avarend})$.
\end{theorem}

\begin{proof}
Since $\int_0^\cdot w_s^2\mathrm{d}s$ is a continuous non-decreasing process, it is sufficient to prove the pointwise convergence. By Lemma \ref{useful} in Appendix \ref{proofmainthm} (or Lemma 10.2(a) of \cite{Koike2012phy}), we have $b_n^{-1/2}\widehat{\mathbf{AVAR}}^n_t-b_n^{1/2}\sum_{k:R^k\leq t}\widetilde{w^2}_{R^k}$ as $n\to\infty$ for every $t$, where $\widetilde{w^2}_{R^k}$ is defined by $(\ref{what})$ with replacing $|\Gamma^{k+1}|$ by $G^n_{R^k}$. Then, we obtain the desired result by Lemma \ref{localest} and the dominated convergence theorem.
\end{proof}

Combining Theorem \ref{propavar} with Theorem \ref{mainthm}, we obtain the following feasible central limit theorem:
\begin{cor}
Under the same assumptions as those of Proposition \ref{propavar}, we have
\begin{equation*}
\frac{\widehat{PHY}(\mathsf{X},\mathsf{Y})^n_t-[X,Y]_t}{\sqrt{\widehat{\mathbf{AVAR}}^n_t}}\to^{d_s}N(0,1)
\end{equation*}
as $n\to\infty$ for any $t\in\mathbb{R}_+$ whenever $\int_0^tw_s^2\mathrm{d}s>0$ a.s. 
\end{cor}


\subsection{Autocorrelated noise}\label{secdepnoise}

We have so far assumed that the observation noise is not autocorrelated, conditionally on $\mathcal{F}^{(0)}$. In empirical studies of financial high-frequency data, however, there are a lot of evidence that microstructure noise is autocorrelated (see \cite{HL2006} and \cite{UO2009} for instance). In this subsection we shall briefly discuss the case that the observation noise is autocorrelated conditionally on $\mathcal{F}^{(0)}$.

We focus on the synchronous case. That is, we assume that $S^i=T^i$ for all $i$. Note that in this case it holds that $\widehat{S}^k=\widehat{T}^k=R^k=S^k$ for all $k$. Let $(\lambda^l_u)_{u\in\mathbb{Z}_+}$ and $(\mu^l_u)_{u\in\mathbb{Z}_+}$ $(l=1,2)$ be four sequences of real numbers such that
\begin{equation}\label{weakdep}
\sum_{u=1}^{\infty}u|\lambda^l_u|<\infty
\quad\mathrm{and}\quad
\sum_{u=1}^{\infty}u|\mu^l_u|<\infty.
\end{equation}
We assume that the observation data $(\mathsf{X}_{S^i})$ and $(\mathsf{Y}_{S^i})$ are of the form
\begin{equation}\label{depmodel}
\left.\begin{array}{l}
\displaystyle\mathsf{X}_{S^i}=X^1_{S^i}+\sum_{u=0}^i\lambda^1_u\epsilon^X_{S^{i-u}}+b_n^{-1/2}\sum_{u=0}^i\mu^1_u(\underline{X}_{S^{i-u}}-\underline{X}_{S^{i-u-1}}),\\
\displaystyle\mathsf{Y}_{S^i}=Y_{S^i}+\sum_{u=0}^i\lambda^2_u\epsilon^Y_{S^{i-u}}+b_n^{-1/2}\sum_{u=0}^i\mu^2_u(\underline{Y}_{S^{i-u}}-\underline{Y}_{S^{i-u-1}}).
\end{array}\right\}
\end{equation}
In other words, the observation noise follows a kind of linear processes. Under such a situation the asymptotic mixed normality of our estimators is still valid: 
\begin{theorem}\label{depCLT}
Suppose that $S^i=T^i$ for every $i$. Suppose also $[\mathrm{H}1](\mathrm{i})$--$(\mathrm{ii})$, $[\mathrm{H}2](\mathrm{i})$--$(\mathrm{ii})$, $[\mathrm{H}3]$--$[\mathrm{H}6]$, $(\ref{weakdep})$ and $(\ref{depmodel})$ are satisfied. Then $(\ref{CLT})$ holds true as $n\to\infty$ with that $\tilde{W}$ is the same one in Theorem \ref{mainthm} and $w$ is given by
\begin{align*}
w_s^2=\psi_{HY}^{-4}\bigg[&\theta\kappa\left\{[X]'_s[Y]'_s+([X,Y]'_s)^2\right\}G_s
+\theta^{-3}\widetilde{\kappa}\left\{\widetilde{\Psi}^{11}_s\widetilde{\Psi}^{22}_s+\left(\widetilde{\Psi}^{12}_s\right)^2\right\}G_s^{-1}\\
&+\theta^{-1}\overline{\kappa}\left\{[X]'_s\widetilde{\Psi}^{22}_s+[Y]'_s\widetilde{\Psi}^{11}_s+2[X,Y]'_s\widetilde{\Psi}^{12}_s-\left(\tilde{\mu}^1_0[\underline{X},Y]'_s-\tilde{\mu}^2_0[X,\underline{Y}]'_s\right)^2 G_s\right\}\Bigg]
\end{align*}
with $\widetilde{\Psi}^{11}_s=\left(\tilde{\lambda}^1_0\right)^2\Psi^{11}_s+\left(\tilde{\mu}^1_0\right)^2[\underline{X}]'_s G_s$, $\widetilde{\Psi}^{22}_s=\left(\tilde{\lambda}^2_0\right)^2\Psi^{22}_s+\left(\tilde{\mu}^2_0\right)^2[\underline{Y}]'_s G_s$ and $\widetilde{\Psi}^{12}_s=\tilde{\lambda}^1_0\tilde{\lambda}^2_0\Psi^{12}_s+\tilde{\mu}^1_0\tilde{\mu}^2_0[\underline{X},\underline{Y}]'_s G_s$, where $\tilde{\lambda}^l_0=\sum_{u=0}^\infty\lambda^l_u$ and $\tilde{\mu}^l_0=\sum_{u=0}^\infty\mu^l_u$ for each $l=1,2$.
\end{theorem}
We give a proof of Theorem \ref{depCLT} in Appendix \ref{proofdepCLT}. The proof is based on a Beveridge-Nelson type  decomposition for the noise. In the nonsynchronous case, we will need to model the autocorrelation structure of the noise on the time dependence in calender time (as \cite{UO2009} did) rather than tick time (as in the above). This is because we have two axes of tick time, $(S^i)$ and  $(T^j)$, in the nonsynchronous case and this fact complicates the analysis of our estimator. However, this topic is beyond the scope of this paper, so that we postpone it to further research.

\subsection{Comparison to other approaches}\label{comparison}

In this subsection we shall briefly discuss the behaviors of other noise-robust volatility estimators in the presence of time endogeneity.

First we focus on the \textit{modulated realized covariance} (MRC) proposed in \citet{CKP2010}, which is another pre-averaging based covariance estimator. It can be basically considered as a pre-averaging version of the realized covariance i.e., $\Xi_{g,g}(\mathsf{X},\mathsf{Y})^n$ (with an appropriate scaling). This quantity, however, has a bias which is given by the covariance of the noise multiplied by some constant, so that we need to involve a bias correction term. For the reason presented in Section \ref{avarest}, we use the estimator $\gamma_t^n(1)^{12}$ for estimating the covariance of the noise. More precisely, we define the process $\operatorname{MRC}(\mathsf{X},\mathsf{Y})^n$ by
\begin{align*}
\operatorname{MRC}(\mathsf{X},\mathsf{Y})^n_t=\frac{1}{\psi_2}\Xi_{g,g}(\mathsf{X},\mathsf{Y})^n_t-\frac{\psi_1}{\psi_2}\gamma^n_t(1)^{12},\qquad t\in\mathbb{R}_+,
\end{align*}
where $\psi_1=\int_0^1 g'(s)^2\mathrm{d}s$ and $\psi_2=\int_0^1 g(s)^2\mathrm{d}s$.

For any $\alpha,\beta\in\Upsilon$, we define the function $\phi_{\alpha,\beta}$ on $\mathbb{R}$ by $\phi_{\alpha,\beta}(s)=\int_s^1\alpha(u-s)\beta(u)\mathrm{d}u$. After that, we set $\Phi_{11}=\int_0^1\phi_{g',g'}(s)^2\mathrm{d}s$, $\Phi_{22}=\int_0^1\phi_{g,g}(s)^2\mathrm{d}s$ and $\Phi_{12}=\int_0^1\phi_{g,g}(s)\phi_{g',g'}(s)\mathrm{d}s$. Then we obtain the following result:
\begin{theorem}\label{mrcthm}
Suppose $[\mathrm{H}1]$--$[\mathrm{H}6]$ are satisfied. Then
\begin{align*}
b_n^{-1/4}\{\operatorname{MRC}(\mathsf{X},\mathsf{Y})^n-[X,Y]\}\to^{d_s}\int_0^\cdot w_s\mathrm{d}\widetilde{W}_s\qquad\mathrm{in}\ \mathbb{D}(\mathbb{R}_+)
\end{align*}
as $n\to\infty$, where $\tilde{W}$ is the same one in Theorem \ref{mainthm} and $w$ is given by
\begin{align}
w_s^2=2\psi_2^{-2}\bigg[&\theta\Phi_{22}\left\{[X]'_s[Y]'_s+([X,Y]'_s)^2\right\}G_s
+\theta^{-3}\Phi_{11}\left\{\overline{\Psi}^{11}_s\overline{\Psi}^{22}_s+\left(\overline{\Psi}^{12}_s\right)^2\right\}G_s^{-1}\nonumber\\
&+\theta^{-1}\Phi_{12}\left\{[X]'_s\overline{\Psi}^{22}_s+[Y]'_s\overline{\Psi}^{11}_s+2[X,Y]'_s\overline{\Psi}^{12}_s-\left([\underline{X},Y]'_s F^1_s-[X,\underline{Y}]'_s F^2_s\right)^2 G_s^{-1}\right\}\Bigg].\label{mrcavar}
\end{align}
\end{theorem}
The proof is in Appendix \ref{proofmrc}. The above result tells us that the time endogeneity also has no impact on the first order asymptotic property of the MRC. In this sense the MRC is better than the PHY because the former has usually smaller asymptotic variance than the later. However, the MRC is not robust to autocorrelated noise, so that in this article we mainly focus on the PHY for practical application.

On the other hand, in general the time endogeneity seems to have some impacts on the asymptotic distribution of noise-robust volatility estimators. That is, the ``robustness to the time endogeneity'' is presumably a special feature of the pre-averaging technique. One of the evidences for this conjecture is the analysis in \cite{LZZ2012}. We give another heuristic evidence in the following. 

For simplicity we focus on the univariate case and suppose that $U^X_{S^i}\equiv0$. We also assume that [H1](i) holds and $S^0,S^1,\dots$ are independent of $X$. We further assume that there exists a constant $\bar{\rho}>2$ such that for each $\rho\in[0,\bar{\rho}]$ the processes $G(\rho)^n$ converges to a c\`adl\`ag process $G(\rho)$ uniformly on compacts in probability as $n\to\infty$. We shall consider the \textit{multiscale realized volatility} (MSRV) proposed in \citet{Z2006}. The MSRV is given by the following formula:
\begin{equation}\label{defmsrv}
\widehat{[X,X]}^{multi}_1=\sum_{i=1}^{M_n}\frac{\alpha_{i,M_n}}{i}\sum_{j=i}^N(\mathsf{X}_{S^j}-\mathsf{X}_{S^{j-i}})^2,
\end{equation}
where $N=\#\{i\in\mathbb{N}|S^i\leq 1\}$, $M_n=\lceil c_{multi}\sqrt{N}\rceil$ with a positive constant $c_{multi}$ and
\begin{align*}
\alpha_{i,M_n}=\frac{12 i^2}{M_n^3-M_n}-\frac{6 i}{M_n^2-1}-\frac{6 i}{M_n^3-M_n}.
\end{align*}
This specification follows from \citet{Bibinger2012}. Then, by Lemma 2.2 of \cite{HJY2011} and Theorem 2 of \cite{Bibinger2012} the asymptotic distribution of the (scaled) estimation error $N^{1/4}(\widehat{[X,X]}^{multi}_1-[X,X]_1)$ of the MSRV is given by $$\zeta\sqrt{c_{multi}\frac{52}{35}\int_0^1([X]'_s)^2\frac{G(2)_s}{G(1)_s}\mathrm{d}s}$$
under some regularity conditions, where $\zeta$ is a standard normal random variable independent of $\mathcal{F}$. The term $G(2)_s/G(1)_s$ appearing in the above seems to reflect the asymptotic kurtosis of returns, so that time endogeneity seems to have an impact on the asymptotic distribution of the MSRV. Although the verification of this conjecture is left to future research, we will examine it numerically by a Monte Carlo study in the next section.

\begin{rmk}
(i) The term $G(2)_s/G(1)_s$ also appears in the asymptotic variance of the realized kernel; see \cite{BNHLS2011} for details. On the other hand, it is known that the asymptotic distribution of the realized quasi-maximum likelihood estimator is not affected by the irregularity of sampling times at least for renewal sampling schemes independent of observed processes, which is similar to our approach; see Section 4.3.2 of \cite{Xiu2010} and Corollary 1 of \cite{SX2012a}.

\noindent (ii) The ``robustness to the time endogeneity'' property in the above has a good aspect and bad aspect. The good aspect is that we do not need to pay attention to the structure of sampling times precisely for statistical application of the estimator. The bad aspect is that we miss a chance to seek more efficient sampling schemes, as manifested in Section 5 of \cite{Fu2010b} for the RV case. This can be seen as a common trade-off between efficiency and robustness.

\noindent (iii) It would be interesting to compare our approach with the \textit{model with uncertainty zones} of \cite{RR2012}, which is another approach allowing us to deal with endogenous noise, endogenous times and nonsynchronous observations simultaneously.
\end{rmk}  

%% file: hitting/hitting_simulation.tex
\section{Simulation study}\label{simulation}

In this section we conduct a simulation analysis to illustrate the finite sample accuracy of some of the asymptotic results developed above. We focus on the univariate PHY.

\subsection{Simulation design}

We simulate data for one day ($t\in[0,1]$). Following \cite{LMRZZ2012,LZZ2012}, we consider sampling times generated by hitting barriers illustrated in Example \ref{Exhit}. Specifically, we define the sampling times $(S^i)$ by Eq.~$(\ref{defhit})$ with setting $u=0.01$, $v=0.04$, $b_n=n^{-1}=3600^{-1}$ and $M^X_t=\sigma W_t$, where $\sigma=0.02$ and $W_t$ is a one-dimensional Wiener process. This specification follows from Section 5 of \cite{LMRZZ2012}. For a comparison we also consider an equidistant sampling scheme i.e., $S^i=i/n$. We will refer to the former as the \textit{hitting sampling} and the later as the \textit{equidistant sampling}, respectively.

One of the novel findings of this article is that the PHY has no limiting bias even if the asymptotic skewness of the latent returns does not vanish, which contrasts the realized volatility. Therefore, it will be a good illustration to consider a situation that the limiting bias of the realized volatility significantly affects its asymptotic distribution. For this reason we adopt the bridge setting for the latent process, following \cite{LZZ2012}. Specifically, $X$ is generated by a Brownian bridge with between 0 and $x_1$. An SDE specification for $X$ can be written as
\begin{align*}
\mathrm{d}X_t=\frac{x_1-X_t}{1-t}\mathrm{d}t+\sigma\mathrm{d}W_t.
\end{align*}
While the limiting bias of the realized volatility is proportionate to the value of $x_1$ in the light of $(\ref{limitbias})$, an overlarge value will cause a significant bias due to the drift term. In consideration of this trade-off, we set $x_1=\sigma/2$.

To generate the microstructure noise process $(U^X_{S^i})$, we consider the following three scenarios:
\begin{enumerate}[{Scenario} 1:]

\item $U^X_{S^i}\equiv0$ i.e., the microstructure noise is absent.

\item $U^X_{S^i}\overset{i.i.d.}{\sim}N(0,\gamma\sigma^2)$, where we set $\gamma=0.001$.

\item $U^X_{S^i}=\delta\sqrt{n}(X_{S^i}-X_{S^{i-1}})$, where we set $\delta=-\sqrt{0.001}$.

\end{enumerate}
The choices of the parameters in the above reflect the empirical findings reported in \citet{HL2006}. That is, the variance of the noise is at most 0.1\% of the integrated volatility and the noise is negatively correlated with the latent returns. Simulation results are based on 5000 Monte Carlo iterations for each model. 

The implementation of the PHY is as follows. Following \cite{CPV2013}, we use $\theta=0.15$ and $g(x)=x\wedge(1-x)$ and set $k_n=\lceil\theta\sqrt{N}\rceil$ for pre-averaging. Here, $N$ represents the number of the observed returns. We also computed  the Studentized statistic
\begin{equation*}
S_{\text{PHY}}:=\left(\widehat{PHY}(\mathsf{X},\mathsf{X})^n_1-\sigma^2\right)\left/\sqrt{\widehat{\mathbf{AVAR}}^n_1}\right.,
\end{equation*}
where the estimator $\widehat{\mathbf{AVAR}}^n_1$ of the asymptotic variance is constructed as in Section \ref{avarest} with using $h_n=N^{-0.2}$. Note that we do not need to specify the function $f$ for computing it because the process $\partial\Xi[f]^n$ is identical to 0 in the univariate case.

For a comparison purpose we also computed the RV and the MSRV defined by $(\ref{defmsrv})$ as well as their Studentization. The Studentization $S_{\text{RV}}$ of the RV was computed by the left-hand side of Eq.~$(\ref{bnsCLT})$ (with $t=1$). The tuning parameter $c_{multi}$ and the estimator $\widehat{\mathbf{AVAR}}_{multi}$ for the asymptotic variance of the MSRV were computed on the basis of Algorithm 2 of \cite{Bibinger2012}. Then the Studentization of the MSRV is given by
\begin{equation*}
S_{\text{MSRV}}:=N^{\frac{1}{4}}\left(\widehat{[X,X]}^{multi}_1-\sigma^2\right)\left/\sqrt{\widehat{\mathbf{AVAR}}_{multi}}\right..
\end{equation*}

\subsection{Simulation results} 

Table \ref{bias} reports the relative bias and the root mean squared error (rmse) of each estimator. That is, we report the sample mean and root mean squared error divided by $\sigma$ of each estimator in Table \ref{bias}. Since the RV is inconsistent in the presence of noise, it does not perform at all in Scenario 2 and 3. As expected from the discussions until the preceding sections, the PHY has the smallest bias in the presence of the time endogeneity. Interestingly, even in the absence of noise the PHY has the smaller bias than the RV when the sampling times are endogenous. On the other hand, in each scenario the difference between the rmse values of two sampling schemes for each estimator (of course except for the RV in the noisy settings) is small, at least compared with those of the bias values. This is also implied by the theory developed above and the formula $(\ref{limitbias})$.    

Next we turn to the accuracy of the asymptotic approximation, which is the main theme of this article. In Figure \ref{density} we plot the kernel densities of the Studentized statistics $S_{\text{PHY}}$, $S_{\text{RV}}$ and $S_{\text{MSRV}}$ for Scenario 1 with the equidistant sampling (on the left panel) and the hitting sampling (on the right panel). In the equidistant sampling case, all of the standard normal approximations perform fairly well. As expected from the asymptotic theory, $S_{\text{RV}}$ offers the best approximation. On the other hand, the standard normal approximations of $S_{\text{RV}}$ and $S_{\text{MSRV}}$ completely fail in the hitting sampling case. In fact, their densities shift to the right and become long and narrow to the lengthwise direction. This is exactly as expected from the asymptotic distribution $(\ref{limitbias})$ for the RV, while it is conjectured from the discussion in Section \ref{comparison} for the MSRV. In contrast, the approximation of $S_{\text{PHY}}$ still perform fairly well. This is in line with the theory developed in this article.

To test the normality of the Studentized statistics quantitatively, we compare their quantiles with those of the standard normal distribution for each scenario. The results are reported in Table \ref{quantile.noise0}--\ref{quantile.noise2}. We also report the sample mean and standard deviation as well as the 95\% coverage. Note that for Scenario 2 and 3 we do not report the results for the RV because of the lack of the consistency. As the tables reveal, we can again observe that the distributions of $S_{\text{RV}}$ and $S_{\text{MSRV}}$ shift to the right in the hitting sampling cases. Also, the quantiles of $S_{\text{PHY}}$ (and $S_{\text{MSRV}}$ for the equidistant sampling case) don't look good enough, but this is not surprising because they have rather slower convergence speeds than $S_{\text{RV}}$. In fact, such an observation has already achieved in the literature (e.g., \cite{BNHLS2008} and \cite{JLMPV2009}). It is worth mentioning that the 95\% coverage of $S_{\text{PHY}}$ seems to be fairly good in practice. It is also interesting to observe that the performance of $S_{\text{PHY}}$ in the hitting sampling case is superior to that in the equidistant sampling case.     

Now we make some efforts to improve the finite sample performance of the asymptotic approximation for the estimation error of the PHY. For this purpose we consider the log-transform
\begin{equation*}
S_{\text{log}}:=\frac{\log\left\{\widehat{PHY}(\mathsf{X},\mathsf{X})^n_1\right\}-\log(\sigma^2)}{\sqrt{\widehat{\mathbf{AVAR}}^n_1}\left/\widehat{PHY}(\mathsf{X},\mathsf{X})^n_1\right.}.
\end{equation*}
By the delta method we have $S_{\text{log}}\xrightarrow{d}N(0,1)$ as $n\to\infty$. It is well-known that this type of transformation often improves the finite sample performance of asymptotic approximations for volatility estimators based on high-frequency data (see \cite{GM2011} and \cite{BNHLS2008} for instance). In fact, this phenomenon can be explained theoretically by higher-order asymptotic properties in some cases; see \cite{GM2011} for details. Furthermore, \cite{GM2011} pointed out that there exist alternative transforms outperforming the log-transform. Motivated by this study, we also consider the inverse transform
\begin{equation*}
S_{\text{inv}}:=-\left(\widehat{PHY}(\mathsf{X},\mathsf{X})^n_1\right)^2\frac{1/\widehat{PHY}(\mathsf{X},\mathsf{X})^n_1-1/\sigma^2}{\sqrt{\widehat{\mathbf{AVAR}}^n_1}},
\end{equation*}
following the suggestion of \cite{GM2011} for the RV. We show the results for these statistics in Table \ref{log.noise0}--\ref{log.noise2}. We can see that the accuracy of asymptotic approximation is surprisingly improved across all the scenarios, compared with the raw statistic case. Further, in the equidistant sampling case $S_{\text{inv}}$ seems to work better than $S_{\text{log}}$ as predicted from the study of \cite{GM2011}, while this observation looks reverse in the hitting sampling case. To understand these findings theoretically, we are likely to need a higher-order asymptotic theory for the PHY. This is more involved and of course beyond the scope of this article.

Finally, in Figure \ref{qqplot} we present the QQ plots of the statistics $S_{\text{PHY}}$, $S_{\text{log}}$ and $S_{\text{inv}}$ for Scenario 1 to complement these results visually.

{\linespread{1.3}
\begin{table}
\caption{Relative bias and root mean squared error}
\begin{center}
\begin{tabular}{l|ccc}
 & PHY & RV & MSRV \\[1pt] \hline
& \multicolumn{3}{c}{Equidistant sampling} \\
Scenario 1  & -0.008 (0.089) & 0.000 (0.023) & -0.001 (0.044) \\
Scenario 2  & -0.005 (0.094) & 7.204 (7.208) & -0.006 (0.074) \\
Scenario 3 & -0.005 (0.091) & 3.407 (3.409) & -0.002 (0.063) \\[3pt]
 & \multicolumn{3}{c}{Hitting sampling} \\
Scenario 1 & 0.006 (0.092) & 0.013 (0.023) & 0.012 (0.040) \\
Scenario 2 & 0.007 (0.097) & 6.971 (6.976) & 0.009 (0.075) \\
Scenario 3 & 0.006 (0.094) & 3.460 (3.461) & 0.012 (0.065) \\ \hline
\end{tabular}
\label{bias}\vspace{3mm}

\parbox{12cm}{\small
\textit{Note}. We report the relative bias and rmse of the estimators included in the simulation study. The number reported in parenthesis is rmse.
}
\end{center}
\end{table}
}

\begin{figure}
\caption{Kernel densities of the Studentized statistics for Scenario 1}\label{density}
\begin{center}
\includegraphics[scale=0.75]{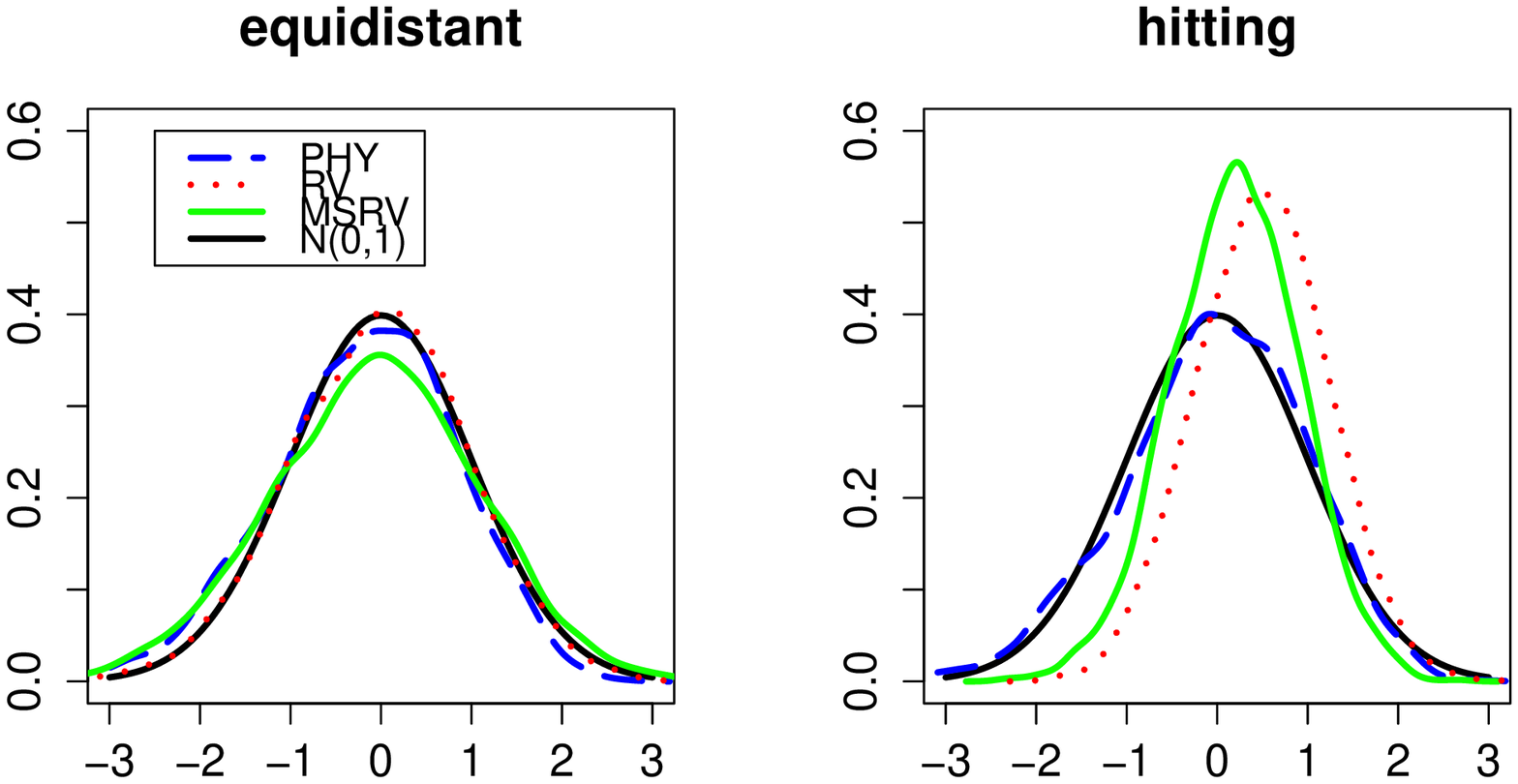}
\parbox{15cm}
{\small
\textit{Note}. We plot the kernel densities of the Studentized statistics for Scenario 1. The left panel is for the equidistant sampling case and the right panel is for the hitting sampling case. The blue dashed line refers to $S_{\text{PHY}}$, the red doted line refers to $S_{\text{RV}}$, the green solid line refers to $S_{\text{MSRV}}$ and the black solid line refers to $N(0,1)$.
}
\end{center}
\end{figure}

{\linespread{1.3}
\begin{table}
\caption{Comparisons of quantiles of Studentized statistics with $N(0,1)$ (Scenario 1)}
\begin{center}
\begin{tabular}{lccccccccc}\hline
 & Mean & SD & 0.5\% & 2.5\% & 5\% & 95\% & 97.5\% & 99.5\% & Cove.~(95\%) \\[1pt] \hline
 \multicolumn{10}{c}{Equidistant sampling} \\
$S_{\text{PHY}}$ & -0.19 & 1.03 & 1.82\% & 4.88\% & 8.90\% & 97.38\% & 99.18\% & 99.94\% & 94.30\% \\
$S_{\text{RV}}$ & -0.02 & 0.99 & 0.78\% & 2.80\% & 5.32\% & 95.74\% & 98.14\% & 99.68\% & 95.34\% \\
$S_{\text{MSRV}}$ & -0.07 & 1.15 & 1.88\% & 5.54\% & 8.90\% & 93.82\% & 96.54\% & 99.12\% & 91.00\% \\[3pt]
 \multicolumn{10}{c}{Hitting sampling} \\
$S_{\text{PHY}}$ & -0.03 & 1.01 & 1.08\% & 3.48\% & 6.68\% & 96.42\% & 98.46\% & 99.88\% & 94.98\% \\
$S_{\text{RV}}$ & 0.50 & 0.73 & 0.00\% & 0.00\% & 0.16\% & 94.42\% & 97.92\% & 99.86\% & 97.92\% \\
$S_{\text{MSRV}}$ & 0.20 & 0.71 & 0.00\% & 0.22\% & 0.66\% & 97.94\% & 99.44\% & 99.94\% & 99.22\% \\ \hline
\end{tabular}
\label{quantile.noise0}
\end{center}
\end{table}
}

{\linespread{1.3}
\begin{table}
\caption{Comparisons of quantiles of Studentized statistics with $N(0,1)$ (Scenario 2)}
\begin{center}
\begin{tabular}{lccccccccc}\hline
 & Mean & SD & 0.5\% & 2.5\% & 5\% & 95\% & 97.5\% & 99.5\% & Cove.~(95\%) \\[1pt] \hline
   \multicolumn{10}{c}{Equidistant sampling} \\
$S_{\text{PHY}}$ & -0.15 & 1.03 & 1.82\% & 4.98\% & 8.36\% & 97.58\% & 98.94\% & 99.90\% & 93.96\% \\
$S_{\text{MSRV}}$ & -0.13 & 1.18 & 2.30\% & 6.32\% & 9.66\% & 93.66\% & 96.56\% & 99.12\% & 90.24\% \\[3pt]
  \multicolumn{10}{c}{Hitting sampling} \\
$S_{\text{PHY}}$ & -0.02 & 1.02 & 1.24\% & 4.00\% & 6.46\% & 96.34\% & 98.66\% & 99.86\% & 94.66\% \\
$S_{\text{MSRV}}$ & 0.08 & 0.85 & 0.12\% & 1.04\% & 2.48\% & 96.80\% & 98.92\% & 99.84\% & 97.88\% \\ \hline
\end{tabular}
\label{quantile.noise1}
\end{center}
\end{table}
}  

{\linespread{1.3}
\begin{table}
\caption{Comparisons of quantiles of Studentized statistics with $N(0,1)$ (Scenario 3)}
\begin{center}
\begin{tabular}{lccccccccc}\hline
 & Mean & SD & 0.5\% & 2.5\% & 5\% & 95\% & 97.5\% & 99.5\% & Cove.~(95\%) \\[1pt] \hline
    \multicolumn{10}{c}{Equidistant sampling} \\
$S_{\text{PHY}}$ & -0.15 & 1.03 & 1.60\% & 5.18\% & 8.26\% & 97.24\% & 99.06\% & 99.96\% & 93.88\% \\
$S_{\text{MSRV}}$ & -0.07 & 1.14 & 1.70\% & 5.28\% & 8.50\% & 94.06\% & 96.68\% & 99.16\% & 91.40\% \\[3pt]
  \multicolumn{10}{c}{Hitting sampling} \\
$S_{\text{PHY}}$ & -0.03 & 1.02 & 1.10\% & 3.52\% & 6.74\% & 96.20\% & 98.42\% & 99.88\% & 94.90\% \\
$S_{\text{MSRV}}$ & 0.13 & 0.82 & 0.14\% & 0.74\% & 1.68\% & 96.82\% & 98.72\% & 99.96\% & 97.98\% \\ \hline
\end{tabular}
\label{quantile.noise2}
\end{center}
\end{table}
}

{\linespread{1.3}
\begin{table}
\caption{Comparisons of quantiles of transformed statistics with $N(0,1)$ (Scenario 1)}
\begin{center}
\begin{tabular}{lccccccccc}\hline
 & Mean & SD & 0.5\% & 2.5\% & 5\% & 95\% & 97.5\% & 99.5\% & Cove.~(95\%) \\[1pt] \hline
    \multicolumn{10}{c}{Equidistant sampling} \\
$S_{\text{log}}$ & -0.14 & 1.01 & 0.94\% & 3.62\% & 7.14\% & 96.32\% & 98.52\% & 99.80\% & 94.90\% \\
$S_{\text{inv}}$ & -0.09 & 1.00 & 0.42\% & 2.40\% & 5.36\% & 95.10\% & 97.56\% & 99.48\% & 95.16\% \\[3pt]
  \multicolumn{10}{c}{Hitting sampling} \\
$S_{\text{log}}$ & 0.03 & 1.01 & 0.70\% & 3.00\% & 5.30\% & 94.54\% & 97.72\% & 99.56\% & 94.72\% \\
$S_{\text{inv}}$ & 0.08 & 1.02 & 0.26\% & 1.90\% & 4.04\% & 93.06\% & 96.40\% & 99.10\% & 94.50\% \\ \hline
\end{tabular}
\label{log.noise0}
\end{center}
\end{table}
}  

{\linespread{1.3}
\begin{table}
\caption{Comparisons of quantiles of transformed statistics with $N(0,1)$ (Scenario 2)}
\begin{center}
\begin{tabular}{lccccccccc}\hline
 & Mean & SD & 0.5\% & 2.5\% & 5\% & 95\% & 97.5\% & 99.5\% & Cove.~(95\%) \\[1pt] \hline
    \multicolumn{10}{c}{Equidistant sampling} \\
$S_{\text{log}}$ & -0.10 & 1.01 & 1.22\% & 3.60\% & 6.74\% & 96.52\% & 98.44\% & 99.78\% & 94.84\% \\
$S_{\text{inv}}$ & -0.05 & 1.00 & 0.40\% & 2.42\% & 5.00\% & 95.46\% & 97.68\% & 99.38\% & 95.26\% \\[3pt]
  \multicolumn{10}{c}{Hitting sampling} \\
$S_{\text{log}}$ & 0.03 & 1.01 & 0.70\% & 3.00\% & 5.30\% & 94.54\% & 97.72\% & 99.56\% & 94.72\% \\
$S_{\text{inv}}$ & 0.08 & 1.02 & 0.26\% & 1.90\% & 4.04\% & 93.06\% & 96.40\% & 99.10\% & 94.50\% \\ \hline
\end{tabular}
\label{log.noise1}
\end{center}
\end{table}
}  

{\linespread{1.3}
\begin{table}
\caption{Comparisons of quantiles of transformed statistics with $N(0,1)$ (Scenario 3)}
\begin{center}
\begin{tabular}{lccccccccc}\hline
 & Mean & SD & 0.5\% & 2.5\% & 5\% & 95\% & 97.5\% & 99.5\% & Cove.~(95\%) \\[1pt] \hline
    \multicolumn{10}{c}{Equidistant sampling} \\
$S_{\text{log}}$ & -0.11 & 1.01 & 0.80\% & 3.60\% & 6.90\% & 96.26\% & 98.12\% & 99.74\% & 94.52\% \\
$S_{\text{inv}}$ & -0.06 & 1.00 & 0.26\% & 2.24\% & 5.32\% & 95.24\% & 97.48\% & 99.38\% & 95.24\% \\[3pt]
  \multicolumn{10}{c}{Hitting sampling} \\
$S_{\text{log}}$ & 0.02 & 1.01 & 0.64\% & 2.36\% & 5.34\% & 95.02\% & 97.52\% & 99.70\% & 95.16\% \\
$S_{\text{inv}}$ & 0.07 & 1.01 & 0.24\% & 1.50\% & 3.72\% & 93.28\% & 96.34\% & 99.12\% & 94.84\% \\ \hline
\end{tabular}
\label{log.noise2}
\end{center}
\end{table}
}

%% file: hitting/hitting_proof_short.tex

\section{Proof of Theorem \ref{mainthm}}\label{proofmainthm}

We start by introducing some notation. Firstly we explain some generic notation. For processes $V$ and $W$, $V\bullet W$ denotes the integral (either stochastic or ordinary) of $V$ with respect to $W$. For any semimartingale $V$ and any (random) interval $I$, we define the processes $V(I)_t$ and $I_t$ by $V(I)_t=\int_0^t1_I(s-)\mathrm{d}V_s$ and $I_t=1_I(t)$ respectively. We denote by $\Upsilon$ the set of all real-valued piecewise Lipschitz functions $\alpha$ on $\mathbb{R}$ satisfying $\alpha(x)=0$ for any $x\notin[0,1]$. For a function $\alpha$ on $\mathbb{R}$ we write $\alpha^n_p=\alpha(p/k_n)$ for each $n\in\mathbb{N}$ and $p\in\mathbb{Z}$. For any semimartingale $V$, any sampling design $\mathcal{D}=(D^i)_{i\in\mathbb{N}}$ and any $\alpha\in\Upsilon$, we define the process $\bar{V}(\mathcal{D})^i_t$ for each $i\in\mathbb{N}$ by
$\bar{V}_\alpha(\mathcal{D})^i_t=\sum_{p=0}^{k_n-1}\alpha^n_p V(D^{i+p})_t.$
Then, for any semimartingales $V,W$ and any $\alpha,\beta\in\Upsilon$, set
\begin{align*}
\bar{L}_{\alpha,\beta}(V,W)^{ij}=\bar{V}_\alpha(\widehat{\mathcal{I}}^i)_-\bullet \bar{W}_\beta(\widehat{\mathcal{J}}^j)+\bar{W}_\beta(\widehat{\mathcal{J}}^j)_-\bullet \bar{V}_\alpha(\widehat{\mathcal{I}}^i)
\end{align*}
for each $i,j\in\mathbb{N}$. Furthermore, for any locally square-integrable martingales $M$, $N$, $M'$, $N'$ and any $\alpha,\beta,\alpha',\beta'\in\Upsilon$, set
\begin{align*}
&V^{iji'j'}_{\alpha,\beta;\alpha',\beta'}(M,N;M',N')_t\\
=&\langle\bar{M}_{\alpha}(\widehat{\mathcal{I}})^i,\bar{M}'_{\alpha'}(\widehat{\mathcal{I}})^{i'}\rangle_t\langle\bar{N}_\beta(\widehat{\mathcal{J}})^j,\bar{N}'_{\beta'}(\widehat{\mathcal{J}})^{j'}\rangle_t+\langle\bar{M}_\alpha(\widehat{\mathcal{I}})^i,\bar{N}'_{\beta'}(\widehat{\mathcal{J}})^{j'}\rangle_t\langle\bar{M}'_{\alpha'}(\widehat{\mathcal{I}})^{i'},\bar{N}_\beta(\widehat{\mathcal{J}})^{j}\rangle_t.
\end{align*}

Secondly we introduce some notation related to the noise processes. Let
\begin{align*}
\mathfrak{E}^X_t=-\frac{1}{k_n}\sum_{p=1}^{\infty}\epsilon^X_{\widehat{S}^p}1_{\{\widehat{S}^p\leq t\}},\qquad
\mathfrak{E}^Y_t=-\frac{1}{k_n}\sum_{q=1}^{\infty}\epsilon^Y_{\widehat{T}^q}1_{\{\widehat{T}^q\leq t\}}.
\end{align*}
$\mathfrak{E}^X$ and $\mathfrak{E}^Y$ are obviously purely discontinuous locally square-integrable martingales on $\mathcal{B}$ if [H6] holds (note that both $(\widehat{S}^i)$ and $(\widehat{T}^j)$ are $\mathbf{F}^{(0)}$-stopping times). Furthermore, if $\Psi$ is c\`adl\`ag, quasi-left continuous and both $(S^i)$ and $(T^j)$ are $\mathbf{F}^{(0)}$-predictable times, then we have
\begin{gather*}
\langle\mathfrak{E}^X\rangle_t=\frac{1}{k_n^2}\sum_{p=1}^{\infty}\Psi^{11}_{\widehat{S}^p}1_{\{\widehat{S}^p\leq t\}},\qquad
\langle\mathfrak{E}^Y\rangle_t=\frac{1}{k_n^2}\sum_{q=1}^{\infty}\Psi^{22}_{\widehat{T}^q}1_{\{\widehat{T}^q\leq t\}},\qquad
\langle\mathfrak{E}^X,\mathfrak{E}^Y\rangle_t=\frac{1}{k_n^2}\sum_{p,q=1}^{\infty}\Psi^{12}_{\widehat{S}^p}1_{\{\widehat{S}^p=\widehat{T}^q\leq t\}}.
\end{gather*}
On the other hand, though $\check{S}^k$ and $\check{T}^k$ may not be stopping times, we have the following result:
\begin{lem}\label{check}
The random variables $\check{I}^k_t$ and $\check{J}^k_t$ are $\mathbf{F}^{(0)}_t$-measurable for every $k,t$.
\end{lem}

\begin{proof}
Since $\{\check{I}^k_t=1\}=\{\check{S}^k\leq t<\widehat{S}^k\}=\bigcap_i[\{S^i\leq t<\widehat{S}^k\}\cup\{t<\widehat{S}^k\leq S^i\}]$, we obtain $\{\check{I}^k_t=1\}\in\mathcal{F}^{(0)}_t$ and thus $\check{I}^k_t$ is $\mathcal{F}^{(0)}_t$-measurable. Similarly we can show that $\check{J}^k_t$ is $\mathcal{F}^{(0)}_t$-measurable.
\end{proof}

Due to the above lemma, both of the processes $\mathfrak{I}_t:=\sum_{p=1}^\infty\check{I}^p_t$ and $\mathfrak{J}_t:=\sum_{q=1}^\infty\check{J}^q_t$ are $\mathbf{F}^{(0)}$-adapted. Therefore, we can define the following processes:
\begin{gather*}
\underline{\mathfrak{X}}_t=-\mathfrak{I}_-\bullet \underline{X}_t,\qquad
\underline{\mathfrak{Y}}_t=-\mathfrak{J}_-\bullet \underline{Y}_t,\qquad
\mathfrak{M}^{\underline{X}}_t=-\mathfrak{I}_-\bullet M^{\underline{X}}_t,\qquad
\mathfrak{M}^{\underline{Y}}_t=-\mathfrak{J}_-\bullet M^{\underline{Y}}_t,\\
\mathfrak{A}^{\underline{X}}_t=-\mathfrak{I}_-\bullet A^{\underline{X}}_t,\qquad
\mathfrak{A}^{\underline{Y}}_t=-\mathfrak{J}_-\bullet A^{\underline{Y}}_t.
\end{gather*}
Then we set $\mathfrak{U}^X=\mathfrak{E}^X+(k_n\sqrt{b_n})^{-1}\underline{\mathfrak{X}}$, $\mathfrak{U}^Y=\mathfrak{E}^Y+(k_n\sqrt{b_n})^{-1}\underline{\mathfrak{Y}}$, $\widetilde{\mathfrak{U}}^X=\mathfrak{E}^X+(k_n\sqrt{b_n})^{-1}\mathfrak{M}^{\underline{X}}$ and $\widetilde{\mathfrak{U}}^Y=\mathfrak{E}^Y+(k_n\sqrt{b_n})^{-1}\mathfrak{M}^{\underline{Y}}$.

Finally, for every $i,j$ we define the process $\bar{K}^{ij}_t$ by $\bar{K}^{ij}=1_{\{[\widehat{S}^i,\widehat{S}^{i+k_n})\cap[\widehat{T}^j,\widehat{T}^{j+k_n})\cap[0,t)\neq\emptyset\}}$. Then define processes $\mathbf{M}^n_t$ and $\widetilde{\mathbf{M}}^n_t$ by
\begin{align*}
\mathbf{M}^n_t&=\mathbf{M}_{g,g}(X,Y)^n_t+\mathbf{M}_{g',g'}(\mathfrak{U}^X,\mathfrak{U}^Y)^n_t
+\mathbf{M}_{g,g'}(X,\mathfrak{U}^Y)^n_t+\mathbf{M}_{g',g}(\mathfrak{U}^X,Y)^n_t,\\
\widetilde{\mathbf{M}}^n_t&=\mathbf{M}_{g,g}(M^X,M^Y)^n_t+\mathbf{M}_{g',g}(\widetilde{\mathfrak{U}}^X,\widetilde{\mathfrak{U}}^Y)^n_t
+\mathbf{M}_{g,g'}(M^X,\widetilde{\mathfrak{U}}^Y)^n_t+\mathbf{M}_{g',g}(\widetilde{\mathfrak{U}}^X,M^Y)^n_t,
\end{align*}
where we set
\begin{align*}
\mathbf{M}_{\alpha,\beta}(V,W)^n_t=\frac{1}{(\psi_{HY}k_n)^2}\sum_{i,j=1}^{\infty}\bar{K}^{ij}_t\bar{L}_{\alpha,\beta}(V,W)^{ij}_t
\end{align*}
for any semimartingales $V,W$ and any $\alpha,\beta\in\Upsilon$. Note that the process $\widetilde{\mathbf{M}}^n_t$ is a locally square-integrable martingale with respect to the filtration $\mathbf{F}$ under the condition [H6] due to Lemma 4.3 of \cite{Koike2012phy}.

In the next step we will strengthen some assumptions by a localization argument. First of all, we remark the following lemma:
\begin{lem}
Suppose that $V$ is a c\`adl\`ag $\mathbf{F}^{(0)}$-adapted process and of class $(\mathrm{AL}_\lambda)$ for some $\lambda\in[0,1]$. Then the process $V-V_0$ is locally bounded.
\end{lem}

\begin{proof}
Without loss of generality, we may assume $V_0=0$. Define the process $H=(H_t)_{t\in\mathbb{R}_+}$ by $H_t=\sup_{0\leq s\leq t}E\left[|V_s|\big|\mathcal{F}^{(0)}_0\right]$. Evidently $H$ is an $\mathcal{F}^{(0)}$-predictable increasing process and satisfies $E[|V_T|]\leq E[|H_T|]$ for every bounded $\mathbf{F}^{(0)}$-stopping time $T$. The fact that $V$ is of class (AL$_\lambda$) for some $\lambda\in[0,1]$ implies the locally boundedness of $H$, so that $V$ is locally bounded because of the Lenglart inequality. 
\end{proof}

Note that for any $\mathcal{F}^{(0)}_0$-measurable random variable $V_0$ and any positive number $K$, $1_{\{|V_0|\leq K\}}X$ and $1_{\{|V_0|\leq K\}}Y$ are also continuous semimartingales on $\mathcal{B}^{(0)}$ and satisfy [H3] and [H5] as far as $X$ and $Y$ do so. Moreover, $P(X\neq1_{\{|V_0|\leq K\}}X~\mathrm{or}~Y\neq1_{\{|V_0|\leq K\}}Y)\leq P(|V_0|>K)\to0$ as $K\to\infty$. As a consequence, a standard localization argument allows us to systematically replace the conditions [H2]--[H3] and [H5]--[H6] by the following strengthened versions:
\begin{enumerate}

\item[{[SH2]}] (i) $S^i$ and $T^i$ are $\mathbf{F}^{(0)}$-predictable times for every $i$.

(ii) [H2](ii) holds and the processes $G$, $V^G$ and $N^G$ are bounded. Moreover, $V^G$ is of class (A$_\lambda$) for any $\lambda\in(0,1]$.

(iii) [H2](iii) holds and the processes $\chi$, $V^\chi$ and $N^\chi$ are bounded. Moreover, $V^\chi$ is of class (A$_\lambda$) for any $\lambda\in(0,1]$.

(iv) [H2](iv) holds and the processes $F^l$, $V^{F^l}$ and $N^{F^l}$ are bounded for every $l=1,2,1*2$. Moreover, $V^{F^l}$ is of class (A$_\lambda$) for any $\lambda\in(0,1]$ and every $l=1,2,1*2$.

\item[{[SH3]}] [H3] holds true. Moreover, for each $V,W=X,Y,\underline{X},\underline{Y}$ the density process $f=[V,W]'$ is bounded and of class (A$_\lambda$) for any $\lambda\in(0,1]$.

\item[[{SH5]}] [H5] holds true. Moreover, for each $V=A^X,A^Y,A^{\underline{X}},A^{\underline{Y}}$ the density process $f=V'$ is bounded and of class (A$_\lambda$) for some $\lambda\in(0,\frac{1}{2})$.

\item[{[SH6]}] [H6] holds and $(\int |z|^8Q_t(\mathrm{d}z))_{t\in\mathbb{R}_+}$ is a bounded process. Moreover, for every $i,j=1,2$ the process $\Psi^{ij}$ is of class (A$_\lambda$) for any $\lambda\in(0,1]$.
 
\end{enumerate}

Our proof is based on the following lemma:
\begin{lem}\label{jacod2}
{\normalfont (a)} Suppose that $[\mathrm{H}1](\mathrm{i})$--$(\mathrm{iii})$, $[\mathrm{SH}3]$ and $[\mathrm{SH}5]$--$[\mathrm{SH}6]$ hold. Suppose also that $\underline{X}=\underline{Y}=0$. Then we have $(\ref{CLT})$ with that $\widetilde{W}$ is the same one in Theorem \ref{mainthm} and $w$ is given by $(\ref{avar})$, provided that the following three conditions are satisfied:

{\normalfont (I)} $b_n^{-1/4}(\mathbf{M}^n-\widetilde{\mathbf{M}}^n)\xrightarrow{ucp}0$ as $n\to\infty$,

{\normalfont (II)} $b_n^{-1/4}\langle\widetilde{\mathbf{M}}^n,N\rangle_t\to^p0$ as $n\to\infty$ for every $t$ and any $N\in\{M^X,M^Y,M^{\underline{X}},M^{\underline{Y}}\}$.

{\normalfont (III)} For any $M,M'\in\{X,\mathfrak{E}^X,\mathfrak{M}^{\underline{X}}\}$, any $N,N'\in\{Y,\mathfrak{E}^Y,\mathfrak{M}^{\underline{Y}}\}$ and any $\alpha,\beta,\alpha',\beta'\in\Upsilon$,
\begin{align}
&b_n^{-1/2}\sum_{i,j,i',j'}(\bar{K}^{ij}_-\bar{K}^{i'j'}_-)\bullet\langle\bar{L}_{\alpha,\beta}^{ij}(M,N),\bar{L}_{\alpha',\beta'}^{i'j'}(M',N')\rangle_t\nonumber\\
=&b_n^{-1/2}\sum_{i,j,i',j'}(\bar{K}^{ij}_-\bar{K}^{i'j'}_-)\bullet V^{iji'j'}_{\alpha,\beta;\alpha',\beta'}(M,N;M',N')_t+o_p\left( k_n^4\right)\label{B2}
\end{align}
as $n\to\infty$ for every $t\in\mathbb{R}_+$.

{\normalfont (b)} Suppose that $[\mathrm{H}1]$, $[\mathrm{SH}3]$ and $[\mathrm{SH}5]$--$[\mathrm{SH}6]$ hold. Then we have $(\ref{CLT})$ with that $\widetilde{W}$ is as in the above and $w$ is given by $(\ref{avarend})$, provided that the above three conditions $(\mathrm{I})$--$(\mathrm{III})$ are satisfied. 
\end{lem}

\begin{proof}
First, it can be easily shown that we can replace the condition [A4] in the assumptions of Proposition 4.3 and Lemma 4.6 of \cite{Koike2012phy} by the condition [H4] (in fact, $\xi'>1/2$ is sufficient). Therefore, it is sufficient to show that $b_n^{-1/4}(\widehat{PHY}(\mathsf{X},\mathsf{Y})^n-\widetilde{\mathbf{M}}^n)\xrightarrow{ucp}0$ as $n\to\infty$ by the assumptions of the lemma and the conditions (II)-(III), and this convergence follows from Lemma 4.2 of \cite{Koike2012phy} and the condition (I). 
\end{proof}
According to the above lemma, it is sufficient to show that the conditions (I)--(III).

In addition to the above localization procedure, we introduce two other ones. The first one is based on the following lemma:
\begin{lem}\label{lem1}
Suppose that $[\mathrm{H}3]$ holds true. Then a.s. we have
\begin{align*}
\limsup_{\delta\rightarrow +0}\sup_{
\begin{subarray}{c}
s,u\in[0,t]\\
|s-u|\leq\delta
\end{subarray}}
\frac{|M^Z_s-M^Z_u|}{\sqrt{2\delta\log\frac{1}{\delta}}}\leq \sup_{0\leq s\leq t}|[Z]'_s|
\end{align*}
for any $t\in\mathbb{R}_+$ and every $Z\in\{X,Y,\underline{X},\underline{Y}\}$.
\end{lem}

\begin{proof}
Combining a representation of a continuous local martingale with Brownian motion and L\'{e}vy's theorem on the uniform modulus of continuity of Brownian motion, we obtain the desired result.
\end{proof}

We can strengthen Lemma \ref{lem1} by a localization in the following way. Suppose there exists a positive constant $K$ such that $\max_{Z\in\{X,Y,\underline{X},\underline{Y}\}}|[Z]'|_t\leq K$ for all $t>0$. For each $k\in\mathbb{N}-\{1\}$, set
\begin{equation*}
\tau_k=\inf\left\{t\in(0,\infty)\bigg|\max_{Z\in\{X,Y,\underline{X},\underline{Y}\}}\sup_{
\begin{subarray}{c}
s,u\in[0,t]\\
|s-u|\leq k^{-1}
\end{subarray}}
\frac{|M^Z_s-M^Z_u|}{\sqrt{2 k^{-1}\log k}}>K+1\right\}.
\end{equation*}
Then $\tau_k$ is a stopping time since $M$ is continuous and adapted, and $\tau_k\uparrow\infty$ a.s. by Lemma \ref{lem1}. This implies that if we have [SH3] then we can always assume that there exist positive constants $K$ and $\delta$ such that 
\begin{equation}\label{absmod}
\max_{Z\in\{X,Y,\underline{X},\underline{Y}\}}\sup_{
\begin{subarray}{c}
s,u\in[0,t]\\
|s-u|\leq \delta
\end{subarray}}
\frac{|M^{Z}_s(\omega)-M^{Z}_u(\omega)|}{\sqrt{2\delta|\log\delta|}}\leq K
\end{equation}
for all $\omega\in\Omega$ localized by $(\tau_k)$. In the remainder of this section, we always assume that we have postive constants $K$ and $\delta$ satisfying $(\mathrm{\ref{absmod}})$, if we have [SH3]. Moreover, whenever we assume $\xi'>1/2$, we only consider sufficiently large $n$ such that $4k_n\bar{r}_n<\delta$, where we write $\bar{r}_n=b_n^{\xi'}$.

On the other hand, the second one is as follows. Fix a positive number $\gamma$, and define $\upsilon_n$ by
\begin{align*}
\upsilon_n=\inf\{t|r_n(t)>\bar{r}_n\}\wedge\inf\{t|N^n_t>b_n^{-1-\gamma}\}.
\end{align*}
By construction each $\upsilon_n$ is an $\mathbf{F}^{(0)}$-stopping time. Moreover, since [H1](i)-(ii) imply that 
\begin{equation}\label{C3}
N^n_t=O_p(b_n^{-1})\qquad\textrm{for any}\ t\in\mathbb{R}_+
\end{equation}
due to Lemma 10.4 of \cite{Koike2012phy}, we have $P(\upsilon_n\leq T)\to0$ as $n\to\infty$ under the assumptions of the theorem. In the following we will only consider processes stopped at time $\upsilon_n$, so that we always assume that
\begin{equation}\label{SA4}
r_n(t)\leq\bar{r}_n\qquad\textrm{for any}\ t\in\mathbb{R}_+\textrm{ and any }n\in\mathbb{N}
\end{equation}
and
\begin{equation}\label{SC3}
N^n_t\leq b_n^{-1-\gamma}\qquad\textrm{for any}\ t\in\mathbb{R}_+\textrm{ and any }n\in\mathbb{N}.
\end{equation}
Moreover, $\gamma$ is taken from the interval $(0,\frac{3}{4}\left(\xi'-\frac{5}{6}\right)\wedge\frac{1}{24})$. This is always possible under the condition [H4].

Now we proceed to derive some consequences of these assumptions. In the following we fix $M,M'\in\{(M^X)^{\upsilon_n},(\mathfrak{E}^X)^{\upsilon_n},(\mathfrak{M}^{\underline{X}})^{\upsilon_n}\}$ and   $N,N'\in\{(M^Y)^{\upsilon_n},(\mathfrak{E}^Y)^{\upsilon_n},(\mathfrak{M}^{\underline{Y}})^{\upsilon_n}\}$. Then we set
\begin{align*}
M^{p,q}=M(\widehat{I}^p)M'(\widehat{I}^{q})-\langle M(\widehat{I}^{p}),M'(\widehat{I}^{q})\rangle,\qquad
L^{p,q}=M(\widehat{I}^p) N'(\widehat{J}^{q})-\langle M(\widehat{I}^{p}),N'(\widehat{J}^{q})\rangle
\end{align*}
for each $p,q$.

Note that $(\ref{absmod})$, $(\ref{SA4})$ and [SH6] implies that for any $r\in[0,8]$ and $t\in\mathbb{R}_+$ there exists a positive constant $C$ such that
\begin{equation}\label{absmod2}
E_0\left[\sup_{0\leq s\leq t}|M(\widehat{I}^p)_s|^r\right]
+E_0\left[\sup_{0\leq s\leq t}|M'(\widehat{I}^p)_s|^r\right]
+E_0\left[\sup_{0\leq s\leq t}|N(\widehat{J}^p)_s|^r\right]
+E_0\left[\sup_{0\leq s\leq t}|N'(\widehat{J}^p)_s|^r\right]
\leq C(\bar{r}_n|\log b_n|)^{r/2}
\end{equation}
for every $p\in\mathbb{Z}_+$, where we denote by $E_0$ the conditional expectation given $\mathcal{F}^{(0)}$, i.e., $E_0[\cdot]:=E[\cdot|\mathcal{F}^{(0)}]$.

For any $\alpha,\beta\in\Upsilon$ and $p,q\in\mathbb{N}$, set
\begin{align*}
c_{\alpha,\beta}(p,q)=\frac{1}{k_n^2}\sum_{i=(p-k_n+1)\vee1}^p\sum_{j=(q-k_n+1)\vee1}^q\alpha^n_{p-i}\beta^n_{q-j}1_{\{[\widehat{S}^i,\widehat{S}^{i+k_n})\cap[\widehat{T}^j,\widehat{T}^{j+k_n})\neq\emptyset\}}.
\end{align*}

The following lemma is useful for obtaining various estimates in the proof. Throughout the discussions, for (random) sequences $(x_n)$ and $(y_n)$, $x_n\lesssim y_n$ means that there exists a (non-random) constant $C\in[0,\infty)$ such that $x_n\leq Cy_n$ for large $n$.
\begin{lem}\label{maest}
Suppose that $[\mathrm{SH}3]$, $[\mathrm{H}4]$ and $[\mathrm{SH}6]$ are satisfied. Let $\alpha,\beta,\alpha',\beta'\in\Upsilon$, $\varpi\in[1,8]$ and $t>0$. Then
\begin{enumerate}[\upshape (a)]

\item There exists a positive constant $C_1$ such that
{\small \begin{align}
E_0\left[\sup_{0\leq s\leq t}\left|\sum_{p:p<r}(\psi_{\alpha,\beta})^n_{q-p} M(\widehat{I}^p)_s\right|^\varpi\right]
+E_0\left[\sup_{0\leq s\leq t}\left|\sum_{p:p<r}c_{\alpha,\beta}(p,q) M(\widehat{I}^p)_s\right|^\varpi\right]
&\leq C_1\left(k_n\bar{r}_n|\log b_n|\right)^{\varpi/2},\label{eqmaest1}\\
E_0\left[\sup_{0\leq s\leq t}\left|\sum_{p:p<r}(\psi_{\alpha,\beta})^n_{q-p} N(\widehat{J}^p)_s\right|^\varpi\right]
+E_0\left[\sup_{0\leq s\leq t}\left|\sum_{p:p<r}c_{\alpha,\beta}(p,q) N(\widehat{J}^p)_s\right|^\varpi\right]
&\leq C_1\left(k_n\bar{r}_n|\log b_n|\right)^{\varpi/2}\label{eqmaest2}
\end{align}}
for any $q,r\in\mathbb{Z}_+$.

\item There exists a positive constant $C_2$ such that
{\small \begin{align*}
E_0\left[\sup_{0\leq s\leq t}\left|\sum_{p:p<r}(\psi_{\alpha,\beta})^n_{q-p} M^{p,p'}_s\right|^\varpi\right]+
E_0\left[\sup_{0\leq s\leq t}\left|\sum_{p:p<r} c_{\alpha,\beta}(p,q)M^{p,p'}_s\right|^\varpi\right]
\leq C_2\left(\sqrt{k_n}\bar{r}_n|\log b_n|\right)^{\varpi}\\
E_0\left[\sup_{0\leq s\leq t}\left|\sum_{p:p<r}(\psi_{\alpha,\beta})^n_{q-p} L^{p,p'}_s\right|^\varpi\right]+
E_0\left[\sup_{0\leq s\leq t}\left|\sum_{p:p<r} c_{\alpha,\beta}(p,q)L^{p,p'}_s\right|^\varpi\right]
\leq C_2\left(\sqrt{k_n}\bar{r}_n|\log b_n|\right)^{\varpi}
\end{align*}}
for any $p',q,r\in\mathbb{Z}_+$, provided that $\varpi\leq 4$.

\end{enumerate}
\end{lem}

\begin{proof}
(a) Note that
\begin{align*}
\sum_{p:p<r}(\psi_{\alpha,\beta})^n_{q-p} M(\widehat{I}^p)_s=
\sum_{p=(q-2k_n+1)_+\wedge r}^{(q+2k_n-1)\wedge r}(\psi_{\alpha,\beta})^n_{q-p} M(\widehat{I}^p)_s
\end{align*}
because $\psi_{\alpha,\beta}$ is equal to 0 outside of the interval $(-2,2)$.

First suppose that $M\in\{(M^X)^{\upsilon_n},(\mathfrak{M}^{\underline{X}})^{\upsilon_n}\}$. Then, Abel's partial summation formula yields
\begin{align*}
\sum_{p:p<q}(\psi_{\alpha,\beta})^n_{q-p} M(\widehat{I}^p)_s
&=\sum_{p=(q-2k_n+1)_+\wedge r}^{(q+2k_n-1)\wedge r}\left\{(\psi_{\alpha,\beta})^n_{q-p} -(\psi_{\alpha,\beta})^n_{q-p-1} \right\}
\left(M_{\widehat{S}^p\wedge s}-M_{\widehat{S}^{(q-2k_n+1)_+\wedge r-1}\wedge s}\right)\\
&\qquad +(\psi_{\alpha,\beta})^n_{q-(q+2k_n-1)\wedge r} \left(M_{\widehat{S}^{(q+2k_n-1)\wedge r}\wedge s}-M_{\widehat{S}^{(q-2k_n+1)_+\wedge r-1}\wedge s}\right),
\end{align*}
hence by $(\ref{absmod})$ and $(\ref{SA4})$ we have
\begin{align*}
E_0\left[\sup_{0\leq s\leq t}\left|\sum_{p:p<r}(\psi_{\alpha,\beta})^n_{q-p} M(\widehat{I}^p)_s\right|^\varpi\right]
\lesssim \left(\sqrt{2\cdot 4 k_n\bar{r}_n|\log(4k_n\bar{r}_n)|}\right)^\varpi
\lesssim \left(k_n\bar{r}_n|\log b_n|\right)^{\varpi/2}.
\end{align*}

Next suppose that $M=(\mathfrak{E}^{X})^{\upsilon_n}$. Then, the Burkholder-Davis-Gundy inequality yields
\begin{align*}
&E_0\left[\sup_{0\leq s\leq t}\left|\sum_{p:p<r}(\psi_{\alpha,\beta})^n_{q-p} M(\widehat{I}^p)_s\right|^\varpi\right]
\lesssim E_0\left[\left\{\frac{1}{k_n^2}\sum_{p:p<r}|(\psi_{\alpha,\beta})^n_{q-p}|^2|\epsilon^X_{\widehat{S}^p}|^21_{\{\widehat{S}^p\leq t\}}\right\}^{\varpi/2}\right].
\end{align*}
Suppose that $\varpi\leq 2$. Then, the Lyapunov inequality implies that
\begin{align*}
&E_0\left[\sup_{0\leq s\leq t}\left|\sum_{p:p<r}(\psi_{\alpha,\beta})^n_{q-p} M(\widehat{I}^p)_s\right|^\varpi\right]
\lesssim \left\{E_0\left[\frac{1}{k_n^2}\sum_{p:p<r}|(\psi_{\alpha,\beta})^n_{q-p}|^2|\epsilon^X_{\widehat{S}^p}|^21_{\{\widehat{S}^p\leq t\}}\right]\right\}^{\varpi/2},
\end{align*}
hence, by [SH6] and the fact that $\psi_{\alpha,\beta}$ is equal to 0 outside of $(-2,2)$ we obtain
\begin{equation}\label{resmaest1}
E_0\left[\sup_{0\leq s\leq t}\left|\sum_{p:p<r}(\psi_{\alpha,\beta})^n_{q-p} M(\widehat{I}^p)_s\right|^\varpi\right]
\lesssim k_n^{-\varpi/2}.
\end{equation}
On the other hand, if $\varpi>2$, then the Jensen inequality and the fact that $\psi_{\alpha,\beta}$ is equal to 0 outside of $(-2,2)$ imply that
\begin{align*}
&E_0\left[\sup_{0\leq s\leq t}\left|\sum_{p:p<r}(\psi_{\alpha,\beta})^n_{q-p} M(\widehat{I}^p)_s\right|^\varpi\right]
\lesssim k_n^{-\varpi/2} E_0\left[\frac{1}{k_n}\sum_{p:p<r}\left|(\psi_{\alpha,\beta})^n_{q-p}\right|^\varpi|\epsilon^X_{\widehat{S}^p}|^\varpi1_{\{\widehat{S}^p\leq t\}}\right],
\end{align*}
hence, again by [SH6] and the fact that $\psi_{\alpha,\beta}$ is equal to 0 outside of $(-2,2)$ we obtain $(\ref{resmaest1})$. Consequently, we conclude that $E_0\left[\sup_{0\leq s\leq t}\left|\sum_{p:p<r}(\psi_{\alpha,\beta})^n_{q-p} M(\widehat{I}^p)_s\right|^\varpi\right]\lesssim\left(k_n\bar{r}_n|\log b_n|\right)^{\varpi/2}$. Similarly we can also show that $E_0\left[\sup_{0\leq s\leq t}\left|\sum_{p:p<r}c_{\alpha,\beta}(p,q) M(\widehat{I}^p)_s\right|^\varpi\right]
\lesssim\left(k_n\bar{r}_n|\log b_n|\right)^{\varpi/2}$, and thus we obtain $(\ref{eqmaest1})$. $(\ref{eqmaest2})$ can be shown in a similar manner.

(b) The claim immediately follows from (a), [SH3], [SH6], $(\ref{absmod2})$ and the Schwarz inequality.
\end{proof}

The following lemma is a generalization of Lemma 2.3 of \citet{Fu2010b}.
\begin{lem}\label{useful}
Consider a sequence $\overline{\mathbf{F}}^n=(\overline{\mathcal{F}}^n_j)_{j\in\mathbb{Z}_+}$ of filtrations and random variables $(\zeta^n_j)_{j\in\mathbb{N}}$ adapted to the filtration $\overline{\mathbf{F}}^n$ for each $n$. Let $N^n(\lambda)$ be an $\overline{\mathbf{F}}^n$-stopping time for each $n\in\mathbb{N}$ and $\lambda$ which is an element of a set $\Lambda$. If it holds that there exists an element $\lambda_0\in\Lambda$ such that $N^n(\lambda)\leq N^n(\lambda_0)$ a.s. for all $\lambda\in\Lambda$. Let $\varpi\in(1,2]$. Then
\begin{enumerate}[\normalfont (a)]

\item if $\sum_{j=1}^{N^n(\lambda_0)}E\left[\left|\zeta^n_j\right|^\varpi\big|\overline{\mathcal{F}}^n_{j-1}\right]\to^p 0$ as $n\to\infty$, then $\sup_{\lambda\in\Lambda}\left|\sum_{j=1}^{N^n(\lambda)}\left\{\zeta^n_j-E\left[\zeta^n_j\big|\overline{\mathcal{F}}^n_{j-1}\right]\right\}\right|\to^p 0$ as $n\to\infty$. 

\item if $\sum_{j=1}^{N^n(\lambda_0)}E\left[\left|\zeta^n_j\right|^\varpi\big|\overline{\mathcal{F}}^n_{j-1}\right]=O_p(1)$ as $n\to\infty$, then $\sup_{\lambda\in\Lambda}\left|\sum_{j=1}^{N^n(\lambda)}\left\{\zeta^n_j-E\left[\zeta^n_j\big|\overline{\mathcal{F}}^n_{j-1}\right]\right\}\right|=O_p(1)$ as $n\to\infty$. 
\end{enumerate}
\end{lem}

\begin{proof}
Note that
\begin{align*}
\sup_{\lambda\in\Lambda}\left|\sum_{j=1}^{N^n(\lambda)}\left\{\zeta^n_j-E\left[\zeta^n_j\big|\overline{\mathcal{F}}^n_{j-1}\right]\right\}\right|\leq\sup_{1\leq k\leq N^n(\lambda_0)}\left|\sum_{j=1}^k\eta^n_j\right|,
\end{align*}
where $\eta^n_j=\zeta^n_j-E\left[\zeta^n_j\big|\overline{\mathcal{F}}^n_{j-1}\right]$.

Let $T$ be a bounded stopping time with respect the filtration $\overline{\mathbf{F}}^n$. Then the Burkholder-Davis-Gundy inequality and the $C_p$ inequality yield
\begin{align*}
E\left[\left|\sum_{j=1}^k\eta^n_j\right|^\varpi\right]\leq 
C E\left[\sum_{j=1}^{T}\left\{|\zeta^n_j|^\varpi+\left|E\left[\zeta^n_j\big|\overline{\mathcal{F}}^n_{j-1}\right]\right|^\varpi\right\}\right]
\end{align*}
for some positive constant $C$ independent of $n$. Since $E\left[\sum_{j=1}^{T}|\zeta^n_j|^\varpi\right]=E\left[\sum_{j=1}^{T}E\left[|\zeta^n_j|^\varpi\big|\overline{\mathcal{F}}^n_{j-1}\right]\right]$ by the optional stopping theorem and $\left|E\left[\zeta^n_j\big|\overline{\mathcal{F}}^n_{j-1}\right]\right|^\varpi\leq E\left[|\zeta^n_j|^\varpi\big|\overline{\mathcal{F}}^n_{j-1}\right]$ by the H\"older inequality, we obtain
\begin{align*}
E\left[\left|\sum_{j=1}^{T}\eta^n_j\right|^\varpi\right]
\leq 2 C E\left[\sum_{k=1}^{T}E\left[|\zeta^n_j|^\varpi\big|\overline{\mathcal{F}}^n_{j-1}\right]\right].
\end{align*}
Therefore, we obtain the desired result due to the Lenglart inequality.
\end{proof}

Now we cope with the main body of the proof. The following lemma is a version of Lemma 12.1 of \cite{Koike2012phy}.
\begin{lem}\label{HYlem13.1and13.2}
Suppose that $[\mathrm{H}1](\mathrm{i})$--$(\mathrm{ii})$, $[\mathrm{SH}3]$, $[\mathrm{H}4]$, $[\mathrm{SH}5]$ and $[\mathrm{SH}6]$ are satisfied. Let $A\in\{(A^X)^{\upsilon_n},(\mathfrak{A}^{\underline{X}})^{\upsilon_n}\}$ and define  
\begin{align*}
\mathbb{I}_t=\sum_{i,j=1}^{\infty}\bar{K}^{ij}_-\bullet\{\bar{A}_\alpha(\widehat{\mathcal{I}})^i_-\bullet\bar{M}_\beta(\widehat{\mathcal{J}})^j\}_t,\qquad
\mathbb{II}_t=\sum_{i,j=1}^{\infty}\bar{K}^{ij}_-\bullet\{\bar{M}_\beta(\widehat{\mathcal{J}})^j_-\bullet\bar{A}_\alpha(\widehat{\mathcal{I}})^i\}_t.
\end{align*}
for each $t\in\mathbb{R}_+$. Then
\begin{enumerate}[\normalfont (a)]
\item $b_n^{-1/4}\sup_{0\leq t\leq T}|\mathbb{I}_{s}|=o_p(k_n^2)$ for every $T>0$.
\item If $A=(A^X)^{\upsilon_n}$ and $[\mathrm{SH}2](\mathrm{ii})$ holds, we have $b_n^{-1/4}\sup_{0\leq t\leq T}|\mathbb{II}_{s}|=o_p(k_n^2)$ for every $T>0$.
\item Suppose that $[\mathrm{H}1](\mathrm{iv})$--$(\mathrm{v})$ and $[\mathrm{SH}2](\mathrm{iv})$ are satisfied. Suppose also that $A=(\mathfrak{A}^{\underline{X}})^{\upsilon_n}$. Then we have $b_n^{-1/4}\sup_{0\leq t\leq T}|\mathbb{II}_{s}|=o_p(k_n^2)$ for every $T>0$.
\end{enumerate}
\end{lem}

\begin{proof}
(a) By an argument similar to the proof of Lemma 12.1(a) of \cite{Koike2012phy} we can prove $b_n^{-1/4}\sup_{0\leq t\leq T}|\mathbb{I}_{t}|$ $=o_p(k_n^2)$. Note that for the proof we do not need the strong predictability condition [A2] of \cite{Koike2012phy} and it is sufficient to hold that $\xi'>3/4$. 

(b) By an argument similar to the proof of Lemma 4.3 of \cite{Koike2012phy}, we can show that
\begin{align*}
\bar{K}^{ij}_-\bullet\{\bar{M}_\beta(\widehat{\mathcal{J}})^j_-\bullet\bar{A}_\alpha(\widehat{\mathcal{I}})^i\}_t
=\bar{K}^{ij}\bar{M}_\beta(\widehat{\mathcal{J}})^j_-\bullet\bar{A}_\alpha(\widehat{\mathcal{I}})^i_t.
\end{align*}
Therefore, we obtain
\begin{align*}
\mathbb{II}_t
=\sum_{i,j=1}^{\infty}\bar{K}^{ij}\bar{M}_\beta(\widehat{\mathcal{J}})^j_-\bullet\bar{A}_\alpha(\widehat{\mathcal{I}})^i_t
=k_n^2\sum_{p,q=1}^\infty c_{\alpha,\beta}(p,q)M(\widehat{J}^{q})_-\widehat{I}^{p}_-\bullet A_t,
\end{align*}
hence it is sufficient to show that $\sup_{0\leq t\leq T}|\tilde{\mathbb{II}}_t|=o_p(b_n^{1/4})$, where
$\tilde{\mathbb{II}}_t=\sum_{p,q=1}^\infty c_{\alpha,\beta}(p,q)M(\widehat{J}^{q})_-\widehat{I}^{p}_-\bullet A_t.$

First, since $M(\widehat{J}^{q})_s=0$ if $s\leq\widehat{T}^{q-1}$, we have
$\tilde{\mathbb{II}}_t=\sum_{p}\sum_{q:q\leq p+1} c_{\alpha,\beta}(p,q)M(\widehat{J}^{q})_-\widehat{I}^{p}_-\bullet A_t.$
Moreover, $(\ref{absmod2})$ yields
\begin{align*}
E_0\left[\sup_{0\leq t\leq T}\left|\sum_{p}\sum_{q:p-1\leq q\leq p+1} c_{\alpha,\beta}(p,q)M(\widehat{J}^{q})_-\widehat{I}^{p}_-\bullet A_t\right|\right]
\lesssim \sqrt{\bar{r}_n|\log b_n|},
\end{align*}
hence we obtain $\tilde{\mathbb{II}}_t=\sum_{p}\sum_{q:q<p-1} c_{\alpha,\beta}(p,q)M(\widehat{J}^{q})_-\widehat{I}^{p}_-\bullet A_t+o_p(b_n^{1/4})$ uniformly in $t\in[0,T]$.

Next we show that
\begin{equation}\label{psi0}
\sup_{0\leq t\leq T}\left|\tilde{\mathbb{II}}_t-\sum_{p}\sum_{q:q<p-1}(\psi_{\alpha,\beta})^n_{q-p} M(\widehat{J}^{q})_-\widehat{I}^{p}_-\bullet A_t\right|=o_p(b_n^{1/4}).
\end{equation}
Since $c_{\alpha,\beta}(p,q)=(\psi_{\alpha,\beta})^n_{q-p}=0$ if $|q-p|\geq 2k_n$, we have
\begin{align*}
&E_0\left[\sup_{0\leq t\leq T}\left|\sum_p\sum_{q:q<(p-1)\wedge k_n}\left\{ c_{\alpha,\beta}(p,q)-(\psi_{\alpha,\beta})^n_{q-p}\right\}M(\widehat{J}^{q})_-\widehat{I}^{p}_-\bullet A_t\right|\right]\\
\lesssim &k_n\bar{r}_n\cdot\sqrt{k_n\bar{r}_n|\log b_n|}
\lesssim b_n^{\frac{3}{2}\left(\xi'-\frac{1}{2}\right)}\sqrt{|\log b_n|}=o_p(b_n^{1/4})
\end{align*}
by Lemma \ref{maest}(a), $(\ref{SA4})$ and [SH5]. In addition, we also have
\begin{align*}
&E_0\left[\sup_{0\leq t\leq T}\left|\sum_p\sum_{q:k_n\leq q<p-1}\left\{ c_{\alpha,\beta}(p,q)-(\psi_{\alpha,\beta})^n_{q-p}\right\}M(\widehat{J}^{q})_-\widehat{I}^{p}_-\bullet A_t\right|\right]\\
\leq&\sup_{p,q\geq k_n}\left|c_{\alpha,\beta}(p,q)-(\psi_{\alpha,\beta})^n_{q-p}\right|\sum_{p,q:|p-q|\leq k_n}\int_0^T E_0\left[\left|M(\widehat{J}^{q})_t\right|\right]\widehat{I}^{p}_t|A'_t|\mathrm{d}t\\
=&O_p(b_n^{1/2}\cdot k_n\cdot \sqrt{\bar{r}_n|\log b_n|})=o_p(b_n^{1/4})
\end{align*}
by $(\ref{absmod2})$, [SH5] and Lemma 3.1 of \cite{Koike2012phy}. Consequently, we conclude that $(\ref{psi0})$ holds true. Therefore, by integration by parts we obtain 
$\tilde{\mathbb{II}}_t=\sum_{p}\sum_{q:q<p-1}(\psi_{\alpha,\beta})^n_{q-p} M(\widehat{J}^{q})_t A(\widehat{I}^p)_t+o_p(b_n^{1/4})$ uniformly in $t\in[0,T]$.
Moreover, since Abel's partial summation formula yields
\begin{align*}
&\sum_{p}\sum_{q:q<p-1}(\psi_{\alpha,\beta})^n_{q-p} M(\widehat{J}^{q})_t\left\{ A(\widehat{I}^p)_t-A(\Gamma^p)_t\right\}\\
=&\sum_{p}\sum_{q:q<p-1}\left\{(\psi_{\alpha,\beta})^n_{q-p}-(\psi_{\alpha,\beta})^n_{q-p-1}\right\} M(\widehat{J}^{q})_t\left( A_{\widehat{S}^p\wedge t}-A_{R^p\wedge t}\right),
\end{align*}
we have
\begin{equation}\label{synchro}
\sup_{0\leq t\leq T}\left|\tilde{\mathbb{II}}_t-\sum_{p}\sum_{q:q<p-1}(\psi_{\alpha,\beta})^n_{q-p} M(\widehat{J}^{q})_t A(\Gamma^p)_t\right|=o_p(b_n^{1/4})
\end{equation}
due to $(\ref{absmod2})$ and the Lipschitz continuity of $\psi_{\alpha,\beta}$. 

Now we show that
\begin{equation}\label{shiftA}
\sup_{0\leq t\leq T}\left|\tilde{\mathbb{II}}_t-\sum_{p}\sum_{q:q<p-1}(\psi_{\alpha,\beta})^n_{q-p} M(\widehat{J}^{q})_t A'_{R^{p-1}}|\Gamma^p(t)|\right|=o_p(b_n^{1/4}).
\end{equation}
Lemma \ref{maest}(a) and the Schwarz inequality yield
\begin{align*}
&E\left[\sup_{0\leq t\leq T}\left|\sum_{p}\sum_{q:q<p-1}(\psi_{\alpha,\beta})^n_{q-p} M(\widehat{J}^{q})_t\left\{ A(\Gamma^p)_t-A'_{R^{p-1}}|\Gamma^p(t)|\right\}\right|\right]\\
\lesssim &\sqrt{k_n\bar{r}_n|\log b_n|}E\left[\sum_{p}\int_{R^{p-1}(T)}^{R^{p}(T)}\left|A'_s-A'_{R^{p-1}}\right|\mathrm{d}s\right]\\
\leq& \sqrt{k_n\bar{r}_n|\log b_n|}T^{1/2}\left\{E\left[\sum_{p}\int_{R^{p-1}(T)}^{R^{p}(T)}\left|A'_s-A'_{R^{p-1}}\right|^2\mathrm{d}s\right]\right\}^{1/2}.
\end{align*}
Further, $(\ref{SA4})$, [SH5] and $(\ref{SC3})$ imply that
\begin{align*}
E\left[\sum_{p}\int_{R^{p-1}(T)}^{R^{p}(T)}\left|A'_s-A'_{R^{p-1}}\right|^2\mathrm{d}s\right]
\leq E\left[\sum_{p}\int_{R^{p-1}(T)}^{R^{p-1}(T)+2\bar{r}_n}E\left[\left|A'_s-A'_{R^{p-1}}\right|^2\big|\mathcal{F}_{R^{p-1}}\right]\mathrm{d}s\right]
\lesssim \bar{r}_n^{2-\lambda}b_n^{-1-\gamma}
\end{align*}
for some $\lambda\in(0,\frac{1}{2})$, hence we obtain
\begin{align*}
E\left[\sup_{0\leq t\leq T}\left|\sum_{p}\sum_{q:q<p-1}(\psi_{\alpha,\beta})^n_{q-p} M(\widehat{J}^{q})_t\left\{ A(\Gamma^p)_t-A'_{R^{p-1}}|\Gamma^p(t)|\right\}\right|\right]
\lesssim b_n^{\frac{\xi'}{2}\left(3-\lambda\right)-1-\gamma}\cdot b_n^{1/4}\sqrt{|\log b_n|}.
\end{align*}
Since $\frac{\xi'}{2}\left(3-\lambda\right)>\frac{25}{24}$ by [H4] and $\gamma<\frac{1}{24}$, we conclude that $(\ref{shiftA})$ holds true. On the other hand, since by Lemma \ref{maest}(a) and [SH5] we have
\begin{align*}
&E_0\left[\sup_{0\leq t\leq T}\left|\sum_{p}\sum_{q:q<p-1}(\psi_{\alpha,\beta})^n_{q-p} M(\widehat{J}^{q})_t A'_{R^{p-1}}\left\{|\Gamma^p(t)|-|\Gamma^p|1_{\{R^{p-1}\leq t\}}\right\}\right|\right]
\lesssim \sqrt{k_n\bar{r}_n|\log b_n|}\sup_{0\leq t\leq T}|\Gamma^{N^n_t+1}|,
\end{align*}
Lemma \ref{supGamma} implies
\begin{align*}
\sup_{0\leq t\leq T}\left|\sum_{p}\sum_{q:q<p-1}(\psi_{\alpha,\beta})^n_{q-p} M(\widehat{J}^{q})_t A'_{R^{(p-2k_n)_+}\wedge t}\left\{|\Gamma^p(t)|-|\Gamma^p|1_{\{R^{p-1}\leq t\}}\right\}\right|=O_p(b_n^{\frac{\xi'}{2}+\frac{3}{4}-\frac{1}{\rho}}|\log b_n|).
\end{align*}
Since $\rho(\xi'+1)>2$, we conclude that
\begin{equation}\label{shiftGamma}
\sup_{0\leq t\leq T}\left|\tilde{\mathbb{II}}_t-\sum_{p=1}^{N^n_t+1}\sum_{q:q<p-1}(\psi_{\alpha,\beta})^n_{q-p} M(\widehat{J}^{q})_t A'_{R^{p-1}}|\Gamma^p|\right|=o_p(b_n^{1/4}).
\end{equation}
Furthermore, with taking $\varpi=\rho\wedge 2$ and $\overline{\mathcal{H}}^n_t=\mathcal{H}^n_t\vee\mathcal{F}^{(1)}_t$ for each $t\in\mathbb{R}_+$, we have
\begin{align*}
&E_0\left[\sup_{0\leq t\leq T}b_n^{-\frac{\varpi}{4}}\sum_{p=1}^{N^n_t+1}E\left[\left|\sum_{q:q<p-1}(\psi_{\alpha,\beta})^n_{q-p} M(\widehat{J}^{q})_t A'_{R^{p-1}} |\Gamma^p|\right|^\varpi\big|\overline{\mathcal{H}}^n_{R^{p-1}}\right]\right]\\
\lesssim & b_n^{\frac{3}{4}\varpi}\left(k_n\bar{r}_n|\log b_n|\right)^{\frac{\varpi}{2}}(N^n_T+1)\sup_{0\leq t\leq T}G(\varpi)^n_t
\lesssim b_n^{\frac{(\xi'+1)}{2}\varpi}|\log b_n|^{\frac{\varpi}{2}}(N^n_T+1)\sup_{0\leq t\leq T}G(\varpi)^n_t
\end{align*}
by Lemma \ref{maest}(a) and [SH5]. Since $(\xi'+1)\varpi/2>1$, [H1](ii) and $(\ref{C3})$ imply that
\begin{align*}
\sup_{0\leq t\leq T}b_n^{-\frac{\varpi}{4}}\sum_{p=1}^{N^n_t+1}E\left[\left|\sum_{q:q<p-1}(\psi_{\alpha,\beta})^n_{q-p} M(\widehat{J}^{q})_t A'_{R^{p-1}} |\Gamma^p|\right|^\varpi\big|\overline{\mathcal{H}}^n_{R^{p-1}}\right]\to^p 0,
\end{align*}
hence, note that $M(\widehat{J}^{q})_t=M_{\widehat{T}^q}-M_{\widehat{T}^{q-1}}$ if $q<p\leq N^n_t+1$, by Lemma \ref{useful} we obtain
\begin{equation*}
\sup_{0\leq t\leq T}\left|\tilde{\mathbb{II}}_t-b_n\sum_{p=1}^{N^n_t+1}\sum_{q:q<p-1}(\psi_{\alpha,\beta})^n_{q-p} M(\widehat{J}^{q})_t A'_{R^{p-1}}G(1)^n_{R^{p-1}}\right|=o_p(b_n^{1/4}).
\end{equation*}
Therefore, [H1](i), Lemma \ref{maest}(a), [SH5] and the fact that $N^n_T=O_p(b_n^{-1})$ yield
\begin{equation}\label{usefularg}
\sup_{0\leq t\leq T}\left|\tilde{\mathbb{II}}_t-b_n\sum_{p=1}^{N^n_t+1}\sum_{q:q<p-1}(\psi_{\alpha,\beta})^n_{q-p} M(\widehat{J}^{q})_t A'_{R^{p-1}}G_{R^{p-1}}\right|=o_p(b_n^{1/4}).
\end{equation}
Moreover, note that $M(\widehat{J}^{q})_t=0$ if $t\leq\widehat{T}^{q-1}$, $R^{k}<\widehat{T}^{k+1}$ and the fact that $\psi_{\alpha,\beta}$ is equal to 0 out side of $(-2,2)$, Lemma \ref{maest}(a), [SH2](ii) and [SH5] imply that
\begin{align*}
&E_0\left[\sup_{0\leq t\leq T}\left|b_n\sum_{p}\sum_{q:q<p-1}(\psi_{\alpha,\beta})^n_{q-p} M(\widehat{J}^{q})_t A'_{R^{p-1}\wedge T}G_{R^{p-1}\wedge T}1_{\{ R^{p-1}>t\}}\right|\right]\\
\lesssim &\sqrt{k_n\bar{r}_n|\log b_n|}\sup_{0\leq t\leq T}b_n\sum_{p}1_{\{R^{p-2k_n-1}< t<R^{p-1}\}}
\lesssim b_n^{1/4}\sqrt{\bar{r}_n|\log b_n|},
\end{align*}
hence we obtain
\begin{equation}\label{endshift}
\sup_{0\leq t\leq T}\left|\tilde{\mathbb{II}}_t-b_n\sum_{p}\sum_{q:q<p-1}(\psi_{\alpha,\beta})^n_{q-p} M(\widehat{J}^{q})_t A'_{R^{p-1}\wedge T}G_{R^{p-1}\wedge T}\right|=o_p(b_n^{1/4}).
\end{equation} 

Here we show that
\begin{equation}\label{contG}
\sup_{0\leq t\leq T}\left|b_n\sum_{p}\sum_{q:q<p-1}(\psi_{\alpha,\beta})^n_{q-p} M(\widehat{J}^{q})_t\left(F_{R^{p-1}\wedge T}-F_{R^{(p-k'_n)_+}}\right)\right|=o_p(b_n^{1/4}),
\end{equation}
where $k'_n=2k_n+1$ and $F=A' G$. Let $\tau_k=\inf\{s\in\mathbb{R}_+|N^G_s=k\}$ $(k=1,2,\dots)$ and set $\mathbb{T}=\{\tau_k|k=1,\dots,N^G_T\}$. Then, by Lemma \ref{maest}, [SH2](ii), [SH5], $(\ref{SA4})$ and $(\ref{SC3})$ we have
\begin{align*}
&E\left[\sup_{0\leq t\leq T}\left|b_n\sum_{p}\sum_{q:q<p-1}(\psi_{\alpha,\beta})^n_{q-p} M(\widehat{J}^{q})_t\left(F_{R^{p-1}\wedge T}-F_{R^{(p-k'_n)_+}}\right)\right|\right]\\
\lesssim &\sqrt{k_n\bar{r}_n|\log b_n|}E\left[b_n\sum_{p}\left| F_{R^{p-1}\wedge T}-F_{R^{(p-k'_n)_+}}\right|1_{\{R^{(p-k'_n)_+}\leq T\}}\right]\\
\lesssim &\sqrt{k_n\bar{r}_n|\log b_n|}\left\{(k_n\bar{r}_n)^{1/2-\lambda}b_n^{-\gamma}+b_n E\left[\#\mathbb{I}^n\right]\right\}
\end{align*}
for some $\lambda\in(0,\frac{1}{4})$, where $\mathbb{I}^n=\{q\in\mathbb{N}|\mathbb{T}\cap[R^{(q-k'_n)_+}\wedge T,R^{q-1}\wedge T)\neq\emptyset\}$. Since for sufficiently large $n$ we have $\#\mathbb{I}^n\leq 4k_n N^G_T$, we obtain
\begin{align*}
&\sup_{0\leq t\leq T}\left|b_n\sum_{p}\sum_{q:q<p-1}(\psi_{\alpha,\beta})^n_{q-p} M(\widehat{J}^{q})_t \left(F_{R^{p-1}\wedge T}-F_{R^{(p-k'_n)_+}}\right)\right|\\
=&O_p\left(\left\{(k_n\bar{r}_n)^{1-\lambda}b_n^{-\gamma}+b_n k_n^{3/2}\sqrt{\bar{r}_n}\right\}\sqrt{|\log b_n|}\right)
=o_p(b_n^{1/4})
\end{align*}
because $\gamma<\frac{3}{4}\xi'-\frac{5}{8}$, and thus $(\ref{contG})$ holds. After all, it is sufficient to show that $\sup_{0\leq t\leq T}|\mathbb{A}_t|\to^p0$ as $n\to\infty$, where
$\mathbb{A}_t=b_n^{\frac{3}{4}}\sum_{p}\sum_{q:q<p-1}(\psi_{\alpha,\beta})^n_{q-p} M(\widehat{J}^{q})_t F_{R^{p-k'_n}}.$

Let $H^q=\sum_{p:p>q+1}(\psi_{\alpha,\beta})^n_{q-p}F_{R^{p-k'_n}}$ for each $q$. Then, by construction $H^q$ is $\mathcal{F}^{(0)}_{\widehat{T}^{q-1}}$-measurable and we have $\mathbb{A}_t=b_n^{\frac{3}{4}}\sum_q H^qM(\widehat{J}^{q})_t$. This implies that the process $\mathbb{A}_t$ is a locally square-integrable martingale with respect to the filtration $\mathbf{F}$ and its predictable quadratic variation is given by
$\langle\mathbb{A}\rangle_t=b_n^{\frac{3}{2}}\sum_q |H^q|^2\langle M\rangle (\widehat{J}^{q})_t.$
Since $|H^q|\lesssim k_n$, we have $\langle\mathbb{A}\rangle_T=O_p(b_n^{1/2})=o_p(1)$. Consequently, the Lenglart inequality completes the proof of the lemma.

(c) By using [H1](iv)-(v) and [SH2](iv) instead of [H1](i)-(ii) and [SH2](ii) respectively, we can adopt an argument similar to the above for the proof. Note that $\check{I}^i_t$ is $\mathcal{H}^n_{\widehat{S}^i}$-adapted and $\check{J}^j_t$ is $\mathcal{H}^n_{\widehat{T}^j}$-adapted. This can be shown in a similar manner to the proof of Lemma \ref{check}.
\end{proof}

The last lemma is a version of Proposition 4.4 of \cite{Koike2012phy}, which deals with the condition (III):
\begin{lem}\label{HYlem12.6and12.8}
Suppose that $[\mathrm{H}1](\mathrm{i})$--$(\mathrm{iii})$, $[\mathrm{SH}3]$, $[\mathrm{H}4]$, $[\mathrm{SH}5]$ and $[\mathrm{SH}6]$ are satisfied. Then we have $(\ref{B2})$ as $n\to\infty$ for every $t\in\mathbb{R}_+$ if
\begin{enumerate}[\normalfont (a)] 
\item $M,M'\in\{X,\mathfrak{E}^X\}$, $N,N'\in\{Y,\mathfrak{E}^Y\}$ and $[\mathrm{SH}2](\mathrm{i})$--$(\mathrm{iii})$ hold true,
\end{enumerate}
or
\begin{enumerate}[\normalfont (b)]
\item $[\mathrm{H}1](\mathrm{iv})$--$(\mathrm{v})$ and $[\mathrm{SH}2]$ hold true.
\end{enumerate}
\end{lem}

\begin{proof}
(a) We decompose the target quantity as
\begin{align*}
&\sum_{i,j,i',j'}(\bar{K}^{ij}_-\bar{K}^{i'j'}_-)\bullet\langle\bar{L}_{\alpha,\beta}^{ij}(M,N),\bar{L}_{\alpha',\beta'}^{i'j'}(M^{\prime},N^{\prime })\rangle_t-\sum_{i,j,i',j'}(\bar{K}^{ij}_-\bar{K}^{i'j'}_-)\bullet V^{iji'j'}_{\alpha,\beta;\alpha',\beta'}(M,N;M^{\prime},N^{\prime})_t\\
=&\Delta_{1,t}+\Delta_{2,t}+\Delta_{3,t}+\Delta_{4,t},
\end{align*}
where
{\small \begin{align*}
\Delta_{1,t}=&\sum_{i,j,i',j'}(\bar{K}^{ij}_-\bar{K}^{i'j'}_-)\bullet(\{\bar{M}_\alpha(\widehat{\mathcal{I}})^i_-\bar{M}^{\prime}_{\alpha'}(\widehat{\mathcal{I}})^{i'}_-\}\bullet\langle\bar{N}_\beta(\widehat{\mathcal{J}})^j,\bar{N}^{\prime}_{\beta'}(\widehat{\mathcal{J}})^{j'}\rangle)_t\\
&-\sum_{i,j,i',j'}(\bar{K}^{ij}_-\bar{K}^{i'j'}_-)\bullet(\langle\bar{M}_\alpha(\widehat{\mathcal{I}})^i,\bar{M}^{\prime}_{\alpha'}(\widehat{\mathcal{I}})^{i'}\rangle_-\bullet\langle\bar{N}_\beta(\widehat{\mathcal{J}})^j,\bar{N}^{\prime}_{\beta'}(\widehat{\mathcal{J}})^{j'}\rangle)_t,\\
\Delta_{2,t}=&\sum_{i,j,i',j'}(\bar{K}^{ij}_-\bar{K}^{i'j'}_-)\bullet(\{\bar{N}_\beta(\widehat{\mathcal{J}})^j_-\bar{N}^{\prime}_{\beta'}(\widehat{\mathcal{J}})^{j'}_-\}\bullet\langle\bar{M}_\alpha(\widehat{\mathcal{I}})^i,\bar{M}^{\prime}_{\alpha'}(\widehat{\mathcal{I}})^{i'}\rangle)_t\\
&-\sum_{i,j,i',j'}(\bar{K}^{ij}_-\bar{K}^{i'j'}_-)\bullet(\langle\bar{N}_\beta(\widehat{\mathcal{J}})^j,\bar{N}^{\prime}_{\beta'}(\widehat{\mathcal{J}})^{j'}\rangle_-\bullet\langle\bar{M}_\alpha(\widehat{\mathcal{I}})^i,\bar{M}^{\prime}_{\alpha'}(\widehat{\mathcal{I}})^{i'}\rangle)_t
\end{align*}}
and
{\small \begin{align*}
\Delta_{3,t}=&\sum_{i,j,i',j'}(\bar{K}^{ij}_-\bar{K}^{i'j'}_-)\bullet(\{\bar{M}_\alpha(\widehat{\mathcal{I}})^i_-\bar{N}^{\prime}_{\beta'}(\widehat{\mathcal{J}})^{j'}_-\}\bullet\langle\bar{N}_\beta(\widehat{\mathcal{J}})^j,\bar{M}^{\prime}_{\alpha'}(\widehat{\mathcal{I}})^{i'}\rangle)_t\\
&-\sum_{i,j,i',j'}(\bar{K}^{ij}_-\bar{K}^{i'j'}_-)\bullet(\langle\bar{M}_\alpha(\widehat{\mathcal{I}})^i,\bar{N}^{\prime}_{\beta'}(\widehat{\mathcal{J}})^{j'}\rangle_-\bullet\langle\bar{N}_\beta(\widehat{\mathcal{J}})^j,\bar{M}^{\prime}_{\alpha'}(\widehat{\mathcal{I}})^{i'}\rangle)_t,\\
\Delta_{4,t}=&\sum_{i,j,i',j'}(\bar{K}^{ij}_-\bar{K}^{i'j'}_-)\bullet(\{\bar{N}_\beta(\widehat{\mathcal{J}})^j_-\bar{M}^{\prime}_{\alpha'}(\widehat{\mathcal{I}})^{i'}_-\}\bullet\langle\bar{M}_\alpha(\widehat{\mathcal{I}})^i,\bar{N}^{\prime}_{\beta'}(\widehat{\mathcal{J}})^{j'}\rangle)_t\\
&-\sum_{i,j,i',j'}(\bar{K}^{ij}_-\bar{K}^{i'j'}_-)\bullet(\langle\bar{N}_\beta(\widehat{\mathcal{J}})^j,\bar{M}^{\prime}_{\alpha'}(\widehat{\mathcal{I}})^{i'}\rangle_-\bullet\langle\bar{M}_\alpha(\widehat{\mathcal{I}})^i,\bar{N}^{\prime}_{\beta'}(\widehat{\mathcal{J}})^{j'}\rangle)_t.
\end{align*}}

Consider $\Delta_{1,t}$ first. By the use of associativity and linearity of integration, we can rewrite $\Delta_{1,t}$ as
\begin{align*}
\Delta_{1,t}=\sum_{i,j,i',j'}(\bar{K}^{ij}_-\bar{K}^{i'j'}_-)\bullet\left(\bar{M}^{ii'}_-\bullet\langle\bar{N}_\beta(\widehat{\mathcal{J}})^j,\bar{N}^{\prime }_{\beta'}(\widehat{\mathcal{J}})^{j'}\rangle\right)_t,
\end{align*}
where $\bar{M}^{ii'}=\bar{M}_\alpha(\widehat{\mathcal{I}})^i\bar{M}^{\prime }_{\alpha'}(\widehat{\mathcal{I}})^{i'}-\langle\bar{M}_\alpha(\widehat{\mathcal{I}})^i, \bar{M}^{\prime }_{\alpha'}(\widehat{\mathcal{I}})^{i'}\rangle$. Moreover, by an argument similar to the proof of Lemma 4.3 of \cite{Koike2012phy}, we can show that
\begin{align*}
(\bar{K}^{ij}_-\bar{K}^{i'j'}_-)\bullet\{\bar{M}^{ii'}_-\bullet\langle\bar{N}_\beta(\widehat{\mathcal{J}})^j,\bar{N}^{\prime}_{\beta'}(\widehat{\mathcal{J}})^{j'}\rangle\}_t
=\bar{K}^{ij}\bar{K}^{i'j'}\{\bar{M}^{ii'}_-\bullet\langle\bar{N}_\beta(\widehat{\mathcal{J}})^j,\bar{N}^{\prime}_{\beta'}(\widehat{\mathcal{J}})^{j'}\rangle\}_t.
\end{align*}
Therefore, we obtain
\begin{align*}
\Delta_{1,t}
&=\sum_{i,j,i',j'}\bar{K}^{ij}\bar{K}^{i'j'}\{\bar{M}^{ii'}_-\bullet\langle\bar{N}_\beta(\widehat{\mathcal{J}})^j,\bar{N}^{\prime}_{\beta'}(\widehat{\mathcal{J}})^{j'}\rangle\}_t\\
&=k_n^4\sum_{p,q,p'=1}^\infty c_{\alpha,\beta}(p,q)c_{\alpha',\beta'}(p',q)M^{p,p'}_-\bullet\{\widehat{J}^{q}_-\bullet\langle N,N^{\prime}\rangle\}_t,
\end{align*}
hence it is sufficient to show that
$\tilde{\Delta}_{1,t}:=\sum_{p,q,p'=1}^\infty c_{\alpha,\beta}(p,q)c_{\alpha',\beta'}(p',q)M^{p,p'}_-\bullet\{\widehat{J}^{q}_-\bullet\langle N,N^{\prime}\rangle\}_t=o_p(b_n^{1/2}).$

Since $M^{p,p'}_s=0$ if $s\leq\widehat{S}^{p\vee p'-1}$, we have
\begin{align*}
\tilde{\Delta}_{1,t}-\sum_q\sum_{p:p<q-1}\sum_{p':p'<q-1}c_{\alpha,\beta}(p,q)c_{\alpha',\beta'}(p',q)M^{p,p'}_-\bullet\{\widehat{J}^{q}_-\bullet\langle N,N^{\prime}\rangle\}_t
=O_p(\sqrt{k_n}\bar{r}_n|\log b_n|)=o_p(b_n^{1/2})
\end{align*}
by Lemma \ref{maest}(b), $(\ref{C3})$ and [H4].
Moreover, by an argument similar to the proof of $(\ref{psi0})$, we can show that
\begin{equation}\label{mimic2}
\tilde{\Delta}_{1,t}=\sum_q\sum_{p:p<q-1}\sum_{p':p'<q-1}(\psi_{\alpha,\beta})^n_{q-p}(\psi_{\alpha',\beta'})^n_{q-p'}M^{p,p'}_-\bullet\{\widehat{J}^{q}_-\bullet\langle N,N^{\prime}\rangle\}_t+o_p(b_n^{1/2}).
\end{equation}
Therefore, integration by parts yields
\begin{equation*}
\tilde{\Delta}_{1,t}=\sum_q\sum_{p:p<q-1}\sum_{p':p'<q-1}(\psi_{\alpha,\beta})^n_{q-p}(\psi_{\alpha',\beta'})^n_{q-p'}M^{p,p'}_t\langle N,N^{\prime}\rangle(\widehat{J}^{q})_t+o_p(b_n^{1/2}).
\end{equation*}
Now we separately consider the following two cases:

\noindent\textit{Case} 1: $N=N'=(M^Y)^{\upsilon_n}$. 
First, by an argument similar to the proof of $(\ref{synchro})$ (using Lemma \ref{maest}(b) instead of $(\ref{absmod2})$) we can show that
\begin{equation*}
\tilde{\Delta}_{1,t}=\sum_q\sum_{p:p<q-1}\sum_{p':p'<q-1}(\psi_{\alpha,\beta})^n_{q-p}(\psi_{\alpha',\beta'})^n_{q-p'}M^{p,p'}_t\langle N,N^{\prime}\rangle(\Gamma^{q})_t+o_p(b_n^{1/2}).
\end{equation*}
Next we show that
\begin{equation}\label{estSA3}
\tilde{\Delta}_{1,t}=\sum_q\sum_{p:p<q-1}\sum_{p':p'<q-1}(\psi_{\alpha,\beta})^n_{q-p}(\psi_{\alpha',\beta'})^n_{q-p'}M^{p,p'}_t\langle N,N^{\prime}\rangle'_{R^{q-1}}|\Gamma^{q}(t)|+o_p(b_n^{1/2}).
\end{equation}
Since $(\ref{SA4})$ and [SH3] yield
\begin{align*}
&E\left[\big|\langle N,N^{\prime}\rangle(\Gamma^q)_t-\langle N,N^{\prime}\rangle'_{R^{q-1}}|\Gamma^q(t)|\right]\\
\leq &E\left[\int_{R^{q-1}(t)}^{R^{q-1}(t)+2\bar{r}_n}E\left[\big|\langle N,N^{\prime}\rangle'_u-\langle N,N^{\prime}\rangle'_{R^{q-1}}|\big|\Big|\mathcal{F}_{R^{q-1}}\right]\mathrm{d}u\right]
\lesssim b_n^{\frac{3}{2}\xi'-\lambda}
\end{align*}
for any $\lambda>0$, we have
\begin{align*}
&E\left[\left|\sum_{q}\sum_{p:p< q-1}\sum_{p':p'<q-1}(\psi_{\alpha,\beta})^n_{q-p}(\psi_{\alpha',\beta'})^n_{q-p'}M^{p,p'}_t\left\{\langle N,N^{\prime}\rangle(\Gamma^{q})_t-\langle N,N^{\prime}\rangle'_{R^{q-1}}|\Gamma^{q}(t)|\right\}\right|\right]\\
\lesssim &k_n\bar{r}_n|\log b_n|E\left[\sum_{q}\big|\langle N,N^{\prime}\rangle(\Gamma^q)_t-\langle N,N^{\prime}\rangle'_{R^{q-1}}|\Gamma^q(t)|\big|\right]
\lesssim b_n^{\frac{5}{2}\xi'-\frac{3}{2}-\lambda-\gamma}|\log b_n|
\end{align*}
by Lemma \ref{maest}(a) and $(\ref{SC3})$. Since $\frac{5}{2}\xi'-2>\frac{1}{12}$ by [H4] and $\gamma<\frac{1}{12}$, we can take $\lambda\in\left(0,\frac{5}{2}\xi'-2-\gamma\right)$ in the above and thus $(\ref{estSA3})$ holds true. On the other hand, since
\begin{align*}
&E_0\left[\left|\sum_q\sum_{p:p<q-1}\sum_{p':p'<q-1}(\psi_{\alpha,\beta})^n_{q-p}(\psi_{\alpha',\beta'})^n_{q-p'}M^{p,p'}_t\langle N,N^{\prime}\rangle'_{R^{q-1}}\left\{|\Gamma^{q}(t)|-|\Gamma^q|1_{\{R^{q-1}\leq t\}}\right\}\right|\right]\\
\lesssim& k_n\bar{r}_n|\log b_n||\Gamma^{N^n_t+1}|=O_p(b_n^{\xi'+1/2}|\log b_n|)=o_p(b_n^{1/2})
\end{align*}
by Lemma \ref{maest}(c) and [H1](i), we obtain
\begin{equation*}
\tilde{\Delta}_{1,t}=\sum_{q=1}^{N^n_t+1}\sum_{p:p<q-1}\sum_{p':p'<q-1}(\psi_{\alpha,\beta})^n_{q-p}(\psi_{\alpha',\beta'})^n_{q-p'}M^{p,p'}_t\langle N,N^{\prime}\rangle'_{R^{q-1}}|\Gamma^{q}|+o_p(b_n^{1/2}).
\end{equation*}
Moreover, by an argument similar to the proof of $(\ref{usefularg})$ we can conclude
\begin{equation*}
\tilde{\Delta}_{1,t}=b_n\sum_{q=1}^{N^n_t+1}\sum_{p:p<q-1}\sum_{p':p'<q-1}(\psi_{\alpha,\beta})^n_{q-p}(\psi_{\alpha',\beta'})^n_{q-p}M^{p,p'}_t\langle N,N^{\prime}\rangle'_{R^{q-1}}G_{R^{q-1}}+o_p(b_n^{1/2}).
\end{equation*}
Further, an argument similar to the proof of $(\ref{endshift})$ yields
\begin{equation*}
\tilde{\Delta}_{1,t}-b_n\sum_{q}\sum_{p:p<q-1}\sum_{p':p'<q-1}(\psi_{\alpha,\beta})^n_{q-p}(\psi_{\alpha',\beta'})^n_{q-p'}M^{p,p'}_t\langle N,N^{\prime}\rangle'_{R^{q-1}\wedge t}G_{R^{q-1}\wedge t}
=O_p(\bar{r}_n|\log b_n|)=o_p(b_n^{1/2}).
\end{equation*}

Now we show that
\begin{equation}\label{lhg}
\tilde{\Delta}_{1,t}=b_n\sum_{q}\sum_{p:p<q-1}\sum_{p':p'<q-1}(\psi_{\alpha,\beta})^n_{q-p}(\psi_{\alpha',\beta'})^n_{q-p'}M^{p,p'}_t F_{R^{(q-k'_n)_+}}+o_p(b_n^{1/2}),
\end{equation}
where $k'_n=2k_n+1$ and $F=\langle N,N^{\prime}\rangle'G$. 
By an argument similar to the proof of $(\ref{contG})$, we can prove
\begin{align*}
&b_n\sum_{q}\sum_{p:p<q-1}\sum_{p':p'<q-1}(\psi_{\alpha,\beta})^n_{q-p}(\psi_{\alpha',\beta'})^n_{q-p'}M^{p,p'}_t\left(F_{R^{q-1}\wedge t}-F_{R^{(q-k'_n)_+}}\right)\\
=&O_p(\left\{b_n^{\left(\xi'-\frac{1}{2}\right)\left(\frac{3}{2}-\lambda\right)-\gamma}+b_n^{\xi'}\right\}|\log b_n|)
\end{align*}
for any $\lambda>0$.
Since $\gamma<\frac{3}{2}\left(\xi'-\frac{5}{6}\right)$, we can take $\lambda$ such that $\left(\xi'-\frac{1}{2}\right)\left(\frac{3}{2}-\lambda\right)-\gamma>\frac{1}{2}$ in the above. Thus we conclude that $(\ref{lhg})$ holds true. Now we have
\begin{align*}
\tilde{\Delta}_{1,t}&=b_n\sum_{p'}H(1)^{p'} M'(\widehat{I}^{p'})_t
+b_n\sum_{p}H(2)^{p} M(\widehat{I}^{p})_t
+2 b_n\sum_pH(3)^p M^{p,p}_t
+o_p(b_n^{1/2}),
\end{align*}
where
\begin{align*}
H(1)^{p'}&=\sum_{q:q>p'+1}(\psi_{\alpha',\beta'})^n_{q-p'}\left[\sum_{p:p<p'}(\psi_{\alpha,\beta})^n_{q-p}M(\widehat{I}^p)_t\right]F_{R^{(q-k'_n)_+}},\\
H(2)^{p}&=\sum_{q:q>p+1}(\psi_{\alpha,\beta})^n_{q-p}\left[\sum_{p':p'<p}(\psi_{\alpha',\beta'})^n_{q-p'}M'(\widehat{I}^{p'})_t\right]F_{R^{(q-k'_n)_+}}
\end{align*}
and $H(3)^p=\sum_{q:q>p}(\psi_{\alpha,\beta})^n_{q-p}(\psi_{\alpha',\beta'})^n_{q-p} F_{R^{(q-k'_n)_+}}$. Since $H(1)^{p'}$ is $\mathcal{F}_{\widehat{S}^{p'-1}}$-measurable, we have
\begin{align*}
\sum_{p'}E\left[\left| b_n^{1/2} H(1)^{p'} M(\widehat{I}^{p'})_t\right|^2\big|\mathcal{F}_{\widehat{S}^{p'-1}}\right]
=b_n\sum_{p'}\left| H(1)^{p'}\right|^2 E\left[\langle M\rangle(\widehat{I}^{p'})_t\big|\mathcal{F}_{\widehat{S}^{p'-1}}\right].
\end{align*}
Moreover, Lemma \ref{maest}(a), [SH2](i), [SH3] and the fact that $\psi_{\alpha',\beta'}$ is equal to 0 outside of $(-2,2)$ yield $E_0[| H(1)^{p'}|^2]\lesssim k_n^2\cdot k_n\bar{r}_n|\log b_n|$, hence we obtain
\begin{align*}
E_0\left[\sum_{p'}E\left[\left| b_n^{1/2}H(1)^{p'} M(\widehat{I}^{p'})_t\right|^2\big|\mathcal{F}_{\widehat{S}^{p'-1}}\right]\right]
=O_p(k_n\bar{r}_n|\log b_n|)=o_p(1)
\end{align*}
by $(\ref{C3})$ and the fact that $E\left[\langle M\rangle(\widehat{I}^{p'})_t\big|\mathcal{F}_{\widehat{S}^{p'-1}}\right]$ is $\mathcal{F}^{(0)}$-measurable. Therefore, Lemma \ref{useful} implies that $b_n^{1/2}\sum_{p'}H(1)^{p'} M(\widehat{I}^{p'})_t=o_p(1)$. Similarly we can show  $b_n^{1/2}\sum_{p}H(2)^{p} M(\widehat{I}^{p})_t=o_p(1)$ and $b_n^{1/2}\sum_{p}H(3)^{p} M^{p,p}_t$ $=o_p(1)$.
Consequently, we conclude that $\tilde{\Delta}_{1,t}=o_p(b_n^{1/2})$.

\noindent\textit{Case} 2: $N=N'=(\mathfrak{E}^Y)^{\upsilon_n}$. In this case we have $\langle N,N'\rangle_t=\frac{1}{k_n^2}\sum_{q=1}^\infty\Psi^{22}_{\widehat{T}^q}1_{\{\widehat{T}^q\leq t\}}$ due to [SH2](i), hence we have 
\begin{equation*}
\tilde{\Delta}_{1,t}=\frac{1}{k_n^2}\sum_q\sum_{p:p<q}\sum_{p':p'<q}(\psi_{\alpha,\beta})^n_{q-p}(\psi_{\alpha',\beta'})^n_{q-p'}M^{p,p'}_t\Psi^{22}_{\widehat{T}^q}1_{\{\widehat{T}^q\leq t\}}+o_p(b_n^{1/2}).
\end{equation*}
Therefore, an argument similar to the latter half of Case 1 yields $\tilde{\Delta}_{1,t}=o_p(b_n^{1/2})$.

Consequently, we conclude that $\Delta_{1,t}=o_p(k_n^4\cdot b_n^{1/2})$. Similarly we can also show that $\Delta_{2,t}=o_p(k_n^4\cdot b_n^{1/2})$.

Next we consider $\Delta_{3,t}$. By the use of associativity and linearity of integration, we have
\begin{align*}
\Delta_{3,t}=&\sum_{i,j,i',j'}(\bar{K}^{ij}_-\bar{K}^{i'j'}_-)\bullet(\bar{L}^{ij'}_-\bullet\langle\bar{M}^{\prime}_{\alpha'}(\widehat{\mathcal{I}})^{i'},\bar{N}_\beta(\widehat{\mathcal{J}})^j\rangle)_t,
\end{align*}
where $\bar{L}^{ij}=\bar{M}_\alpha(\widehat{\mathcal{I}})^i\bar{N}^{\prime }_{\beta'}(\widehat{\mathcal{J}})^{j'}-\langle\bar{M}_{\alpha}(\widehat{\mathcal{I}})^i, \bar{N}^{\prime }_{\beta'}(\widehat{\mathcal{J}})^{j'}\rangle$. Therefore, by an argument similar to the above we obtain
$\Delta_{3,t}
=k_n^4\sum_{p,q,p',q'=1}^\infty c_{\alpha,\beta}(p,q)c_{\alpha',\beta'}(p',q')L^{p,q'}_-\bullet\{\widehat{I}^{p'}_-\widehat{J}^{q}_-\bullet\langle M',N\rangle\}_t.$
Since $\widehat{I}^{p'}\cap\widehat{J}^{q}=\emptyset$ if $|p'-q|>1$, we obtain
$\Delta_{3,t}
=k_n^4\sum_{p,q'}\sum_{p',q:|p'-q|\leq 1} c_{\alpha,\beta}(p,q)c_{\alpha',\beta'}(p',q')L^{p,q'}_-\bullet\{\widehat{I}^{p'}_-\widehat{J}^{q}_-\bullet\langle M',N\rangle\}_t.$
Hence, Lemma \ref{maest}, the Lipschitz continuity of $\alpha'$ and the fact that $c_{\alpha,\beta}(p,q)=0$ if $|p-q|\geq 2k_n$ yield
\begin{align*}
&E_0\left[\left|\Delta_{3,t}-k_n^4\sum_{p,q'}\sum_{p',q:|p'-q|\leq 1} c_{\alpha,\beta}(p,q)c_{\alpha',\beta'}(q,q')L^{p,q'}_-\bullet\{\widehat{I}^{p'}_-\widehat{J}^{q}_-\bullet\langle M',N\rangle\}_t\right|\right]\\
\lesssim &k_n^4\cdot k_n\cdot\sqrt{k_n}\bar{r}_n|\log b_n|\cdot k_n^{-1}
=o_p(k_n^4\cdot b_n^{1/2}),
\end{align*}
and thus we obtain
$k_n^{-4}\Delta_{3,t}
=\sum_{p,q,q'=1}^\infty c_{\alpha,\beta}(p,q)c_{\alpha',\beta'}(q,q')L^{p,q'}_-\bullet\{\widehat{J}^{q}_-\bullet\langle M',N\rangle\}_t+o_p(b_n^{1/2}).$
Then, an argument similar to the above yields $k_n^{-4}\Delta_{3,t}=o_p(b_n^{1/2})$. Similarly we can also show that $k_n^{-4}\Delta_{4,t}=o_p(b_n^{1/2})$, hence we complete the proof of (a).

(b) Similar to the proof of (a) (note that $\check{I}^i_t$ is $\mathcal{H}^n_{\widehat{S}^i}$-adapted and $\check{J}^j_t$ is $\mathcal{H}^n_{\widehat{T}^j}$-adapted.).
\end{proof}

\begin{proof}[\bf\upshape Proof of Theorem \ref{mainthm}]
First, the condition (III) immediately follows from Lemma \ref{HYlem12.6and12.8}. Next, Lemma \ref{HYlem13.1and13.2} and integration by parts imply that the condition (I) is satisfied. Finally, since for any locally squared-integrable martingales $L,M,N$ and any $\alpha,\beta\in\Upsilon$ we have
\begin{align*}
&\langle\mathbf{M}_{\alpha,\beta}(M,N)^n,L\rangle_t\\
=&\frac{1}{(\psi_{HY}k_n)^2}\left[\sum_{i,j=1}^{\infty}\bar{K}^{ij}_-\bullet\{\bar{M}_\alpha(\widehat{\mathcal{J}})^j_-\bullet\overline{[N,L]}_\beta(\widehat{\mathcal{I}})^i\}_t+\sum_{i,j=1}^{\infty}\bar{K}^{ij}_-\bullet\{\bar{N}_\beta(\widehat{\mathcal{J}})^j_-\bullet\overline{[M,L]}_\alpha(\widehat{\mathcal{I}})^i\}_t\right]
\end{align*}
due to Lemma 4.3 of \cite{Koike2012phy}, Lemma \ref{HYlem13.1and13.2} yields the condition (II). Consequently, we obtain the desired result by Lemma \ref{jacod2}. 
\end{proof}

%% file: hitting/hitting_avar.tex
\section{Proof of Lemma \ref{localest}}\label{prooflocalest}

Exactly as in the previous section, we can use a localization procedure for the proof, and which allows us to assume the conditions [SH3], [SH5]--[SH6], $(\ref{absmod})$ and $(\ref{SA4})$.

First we prove two lemmas about the point process generated by the refresh times.
\begin{lem}\label{localC3}
Suppose that $[\mathrm{H}1](\mathrm{i})$ and $[\mathrm{K}_\rho]$ for some $\rho\in(1,2]$ hold true. Let $(H^n)$ be a sequence of stochastic processes, and suppose that $H^n$ is $\mathbf{H}^n$-adapted for each $n$ and $\sup_{0\leq s\leq t}|H^n_s|$ is tight as $n\to\infty$ for any $t>0$. Then we have 
\begin{equation*}
\sup_{0\leq s\leq t}\left| b_n\sum_{k=1}^{N^n_s+1}H^n_{R^{k-1}}-\sum_{k=1}^{N^n_s+1}\frac{H^n_{R^{k-1}}}{G^n_{R^{k-1}}}|\Gamma^k|\right|=O_p(b_n^{1-1/\rho})
\end{equation*}
as $n\to\infty$ for any $t>0$.
\end{lem}

\begin{proof}
Since the assumptions yield
\begin{align*}
\sum_{k=1}^{N^n_t+1}E\left[\left| b_n^{\frac{1}{\rho}-1}\frac{H^n_{R^{k-1}}}{G^n_{R^{k-1}}}|\Gamma^k|\right|^\rho\big|\mathcal{H}^n_{R^{k-1}}\right]
=b_n\sum_{k=1}^{N^n_t+1}\frac{H^n_{R^{k-1}}}{G^n_{R^{k-1}}}G(\rho)^n_t=O_p(1),
\end{align*}
by Lemma \ref{useful} we obtain $\sup_{0\leq s\leq t}b_n^{\frac{1}{\rho}-1}| \sum_{k=1}^{N^n_s+1}H^n_{R^{k-1}}|\Gamma^k|/G^n_{R^{k-1}}-b_n\sum_{k=1}^{N^n_s+1}H^n_{R^{k-1}}G(1)^n_{R^{k-1}}/G^n_{R^{k-1}}|=O_p(1)$. Evidently we have $\sup_{0\leq s\leq t}b_n^{\frac{1}{\rho}-1}|b_n\sum_{k=1}^{N^n_s+1}H^n_{R^{k-1}}G(1)^n_{R^{k-1}}/G^n_{R^{k-1}}-b_n\sum_{k=1}^{N^n_s+1}H^n_{R^{k-1}}|=o_p(1)$, hence we obtain the desired result.
\end{proof}

\begin{lem}\label{tightN}
Suppose that $[\mathrm{H}1](\mathrm{i})$ and $[\mathrm{K}_{4/3}]$ hold true. Suppose also that $h_n^{-1}b_n^{1/4}\to0$ as $n\to\infty$. Then $\sup_{0\leq s\leq t}h_n^{-1}b_n(N^n_s-N^n_{(s-h_n)_+})$ is tight as $n\to\infty$ for any $t>0$.
\end{lem}

\begin{proof}
Since
$\sum_{k=N^n_{(s-h_n)_+}+2}^{N^n_s+1}|\Gamma^k|/G^n_{R^{k-1}}
\leq \left(h_n+\sup_{0\leq u\leq t}|\Gamma^{N^n_u+1}|\right)\sup_{0\leq u\leq t}G_u^{-1}$
for any $s\in[0,t]$, the desired result follows from the assumptions, Lemma \ref{supGamma} and Lemma \ref{localC3}. 
\end{proof}

Next we consider the asymptotic properties of the estimators for the noise covariance matrix.
\begin{lem}\label{noiserep}
Suppose that $[\mathrm{H}1]$, $[\mathrm{SH}3]$, $[\mathrm{H}4]$, $[\mathrm{SH}6]$, $(\ref{absmod})$ and $(\ref{SA4})$ are satisfied. Then
{\small \begin{align}
&\sup_{0\leq s\leq t}\left|\gamma^{n}_s(1)^{11}-\frac{1}{k_n^2}\sum_{k=1}^{N^{n,1}_s+1}\left\{\Psi^{11}_{\widehat{S}^{k-1}}+b_n^{-1}[\underline{X}]'_{\widehat{S}^{k-1}}|\check{I}^{k}|\right\}\right|=o_p(b_n^{1/4}),\label{noiserep11}\\
&\sup_{0\leq s\leq t}\left|\gamma^{n}_s(1)^{22}-\frac{1}{k_n^2}\sum_{k=1}^{N^{n,2}_s+1}\left\{\Psi^{22}_{\widehat{T}^{k-1}}+b_n^{-1}[\underline{Y}]'_{\widehat{T}^{k-1}}|\check{J}^{k}|\right\}\right|=o_p(b_n^{1/4}),\label{noiserep22}\\
&\sup_{0\leq s\leq t}\left|\gamma^{n}_s(1)^{12}-\frac{1}{k_n^2}\sum_{k=1}^{N^n_s+1}\left\{\Psi^{12}_{R^{k-1}}1_{\{\widehat{S}^{k}=\widehat{T}^{k}\}}+b_n^{-1}[\underline{X},\underline{Y}]'_{R^{k-1}}|\check{I}^{k}*\check{J}^{k}|\right\}\right|=o_p(b_n^{1/4})\label{noiserep12}
\end{align}}
as $n\to\infty$ for every $t>0$.
\end{lem}

\begin{proof}
We can consider each of $(\ref{noiserep11})$ and $(\ref{noiserep22})$ as a special case of $(\ref{noiserep12})$ by taking $X=Y$ and $\widehat{S}^k=\widehat{T}^k$, hence it is sufficient to prove $(\ref{noiserep12})$. Furthermore, by symmetry it is sufficient to show that
\begin{align*}
\sup_{0\leq s\leq t}\left|\widetilde{\gamma}^{n}_s(1)^{12}-\frac{1}{k_n^2}\sum_{k=1}^{N^n_s+1}\left\{\Psi^{12}_{R^{k-1}}1_{\{\widehat{S}^{k}=\widehat{T}^{k}\}}+b_n^{-1}[\underline{X},\underline{Y}]'_{R^{k-1}}|\check{I}^{k}*\check{J}^{k}|\right\}\right|=o_p(b_n^{1/4}),
\end{align*}
where $\widetilde{\gamma}^{n}_s(1)^{12}=-\frac{1}{k_n^2}\sum_{k:R^{k+1}\leq t}(\mathsf{X}_{\widehat{S}^k}-\mathsf{X}_{\widehat{S}^{k-1}})(\mathsf{Y}_{\widehat{T}^{k+1}}-\mathsf{Y}_{\widehat{T}^{k}})$.

First, by $(\ref{C3})$, $(\ref{SA4})$, $(\ref{absmod})$, [SH3], [H4] and [SH6], we have
\begin{align*}
\sup_{0\leq s\leq t}\left|\widetilde{\gamma}^n_s(1)^{12}-\left\{-\sum_{k}(\mathfrak{U}^X(\widehat{I}^k)_s-\mathfrak{U}^X(\widehat{I}^{k-1})_s)(\mathfrak{U}^Y(\widehat{J}^{k+1})_s-\mathfrak{U}^Y(\widehat{J}^{k})_s)\right\}\right|=o_p(b_n^{1/4}).
\end{align*}
Next, integration by parts yields
\begin{align*}
&-\sum_{k}(\mathfrak{U}^X(\widehat{I}^k)_s-\mathfrak{U}^X(\widehat{I}^{k-1})_s)(\mathfrak{U}^Y(\widehat{J}^{k+1})_s-\mathfrak{U}^Y(\widehat{J}^{k})_s)\\
=&-\sum_{k}\left\{L(\mathfrak{U}^X,\mathfrak{U}^Y)^{k,k+1}_s-L(\mathfrak{U}^X,\mathfrak{U}^Y)^{k,k}_s-L(\mathfrak{U}^X,\mathfrak{U}^Y)^{k-1,k}_s+L(\mathfrak{U}^X,\mathfrak{U}^Y)^{k-1,k+1}_s\right\}\\
& +\sum_{k}[\mathfrak{E}^X,\mathfrak{E}^Y](\widehat{I}^k\cap\widehat{J}^k)_s
+\frac{b_n^{-1}}{k_n^2}\sum_{k}[\underline{\mathfrak{X}},\underline{\mathfrak{Y}}](\widehat{I}^k\cap\widehat{J}^k)_s\\
=:&\mathbb{A}_{1,s}+\mathbb{A}_{2,s}+\mathbb{A}_{3,s}.
\end{align*}
Combining martingale properties with $(\ref{C3})$, $(\ref{SA4})$, $(\ref{absmod})$ and [SH4]-[SH6], we obtain $\sup_{0\leq s\leq t}|\mathbb{A}_{1,s}|=o_p(b_n^{1/4})$. On the other hand, $(\ref{C3})$, [SH6] and the Doob inequality imply that
$\mathbb{A}_{2,s}=\frac{1}{k_n^2}\sum_k\Psi^{12}_{R^k}1_{\{\widehat{S}^k=\widehat{T}^k\leq s\}}+O_p(b_n^{1/2})$ uniformly in $s\in[0,t]$.
Therefore, by arguments similar to the proofs of $(\ref{shiftA})$ and $(\ref{shiftGamma})$ we obtain $\mathbb{A}_{2,s}=\frac{1}{k_n^2}\sum_{k=1}^{N^n_s+1}\Psi^{12}_{R^{k-1}}1_{\{\widehat{S}^k=\widehat{T}^k\}}+o_p(b_n^{1/4})$ uniformly in $s\in[0,t]$. By applying a similar argument to $\mathbb{A}_{3,s}$, we complete the proof of the lemma.
\end{proof}

Finally we consider the asymptotic property of the estimator $\Xi[f]^n$ for the asymptotic variance due to the endogenous noise. For this purpose we first analyze the more general quantity $\Xi_{\alpha,\beta}(V,W)^n$. For any semimartingale $V,W$ and any $\alpha,\beta\in\Upsilon$, we introduce an infeasible version of this quantity: 
\begin{align*}
\widetilde{\Xi}_{\alpha,\beta}(V,W)^n_t=\frac{1}{k_n}\sum_{i=1}^\infty\bar{V}(\widehat{\mathcal{I}})^i_t\bar{W}(\widehat{\mathcal{J}})^i_t,\qquad t\in\mathbb{R}_+.
\end{align*}
Moreover, we define the processes $\mathbb{M}^{(1)}_{\alpha,\beta}(V,W)^n$ and $\mathbb{M}^{(2)}_{\alpha,\beta}(V,W)^n$ by
\begin{align*}
\mathbb{M}^{(1)}_{\alpha,\beta}(V,W)^n_t&=\sum_{p,q:q-k_n<p<q-1}(\phi_{\alpha,\beta})^n_{q-p} V(\widehat{I}^{p})_-\bullet W(\widehat{J}^{q})_t,\\
\mathbb{M}^{(2)}_{\alpha,\beta}(V,W)^n_t&=\sum_{p,q:q-k_n<p<q-1}(\phi_{\beta,\alpha})^n_{q-p} W(\widehat{J}^{p})_-\bullet V(\widehat{I}^{q})_t
\end{align*}
and set $\mathbb{M}_{\alpha,\beta}(V,W)^n=\mathbb{M}^{(1)}_{\alpha,\beta}(V,W)^n+\mathbb{M}^{(2)}_{\alpha,\beta}(V,W)^n$. Then we have the following results:
\begin{lem}\label{Xirep}
Suppose that $[\mathrm{H}1]$, $[\mathrm{SH}3]$, $[\mathrm{H}4]$, $[\mathrm{SH}5]$--$[\mathrm{SH}6]$, $(\ref{absmod})$ and $(\ref{SA4})$ are satisfied. Let $V\in\{M^X,\mathfrak{E}^X,$ $\mathfrak{M}^{\underline{X}},A^X,\mathfrak{A}^{\underline{X}}\}$, $W\in\{M^Y,\mathfrak{E}^Y,\mathfrak{M}^{\underline{Y}},A^Y,\mathfrak{A}^{\underline{Y}}\}$ and $\alpha,\beta\in\Upsilon$. Then
\begin{align*}
b_n^{-1/4}\left\{\widetilde{\Xi}_{\alpha,\beta}(V,W)^n_s-\mathbb{M}_{\alpha,\beta}(V,W)^n_s-\phi_{\alpha,\beta}(0)[V,W]_s\right\}\xrightarrow{ucp}0
\end{align*}
as $n\to\infty$.
\end{lem}

\begin{proof}
Fix a $t>0$. Since
\begin{align*}
\widetilde{\Xi}_{\alpha,\beta}(V,W)^n_s
=\frac{1}{k_n}\sum_{p,q:|q-p|<k_n}\sum_{i=(p\vee q-k_n+1)\vee 1}^{p\wedge q}\alpha^n_{p-i}\beta^n_{q-i}V(\widehat{I}^{p})_s W(\widehat{J}^{q})_s,
\end{align*}
by integration by parts we can decompose the target quantity as
\begin{align*}
&\widetilde{\Xi}_{\alpha,\beta}(V,W)^n_s\\
=&\frac{1}{k_n}\sum_{p,q:|q-p|<k_n}\sum_{i=(p\vee q-k_n+1)\vee 1}^{p\wedge q}\alpha^n_{p-i}\beta^n_{q-i}\left\{V(\widehat{I}^{p})_-\bullet W(\widehat{J}^{q})_s+W(\widehat{J}^{q})_-\bullet V(\widehat{I}^{p})_s+[V,W](\widehat{I}^p\cap\widehat{J}^{q})_s\right\}\\
=:&\mathbb{B}_{1,s}+\mathbb{B}_{2,s}+\mathbb{B}_{3,s}.
\end{align*}

First consider $\mathbb{B}_{1,s}$. Noting that $V(\widehat{I}^{p})_s=0$ if $s\leq\widehat{S}^{p-1}$, we have
\begin{align*}
\sup_{0\leq s\leq t}\left|\mathbb{B}_{1,s}-\frac{1}{k_n}\sum_{p,q:q-k_n<p<q-1}\sum_{i=(q-k_n+1)\vee 1}^{p}\alpha^n_{p-i}\beta^n_{q-i}V(\widehat{I}^{p})_-\bullet W(\widehat{J}^{q})_s\right|=o_p(b_n^{1/4}).
\end{align*}
Moreover, since
\begin{align*}
\sum_{\begin{subarray}{c}
p,q:q-k_n<p<q-1\\
q\geq k_n
\end{subarray}}\sum_{i=(q-k_n+1)\vee 1}^{p}\alpha^n_{p-i}\beta^n_{q-i}V(\widehat{I}^{p})_-\bullet W(\widehat{J}^{q})_s
=\sum_{\begin{subarray}{c}
p,q:q-k_n<p<q-1\\
q\geq k_n
\end{subarray}}\sum_{i=q-p}^{k_n-1}\alpha^n_{i-(q-p)}\beta^n_{i}V(\widehat{I}^{p})_-\bullet W(\widehat{J}^{q})_s,
\end{align*}
we obtain
$\mathbb{B}_{1,s}=\sum_{p,q:q-k_n<p<q-1}(\phi_{\alpha,\beta})^n_{q-p} V(\widehat{I}^{p})_-\bullet W(\widehat{J}^{q})_s+o_p(b_n^{1/4})$ uniformly in $s\in[0,t]$ by using the Lipschitz continuity of $\alpha,\beta$ and the martingale property of $W$ if $W\in\{M^Y,\mathfrak{E}^Y,\mathfrak{M}^{\underline{Y}}\}$. Similarly we can show that
$\mathbb{B}_{2,s}=\sum_{p,q:q-k_n<p<q-1}(\phi_{\beta,\alpha})^n_{q-p} W(\widehat{J}^{p})_-\bullet V(\widehat{I}^{q})_s+o_p(b_n^{1/4})$ uniformly in $s\in[0,t]$.

Finally, since $\widehat{I}^p\cap\widehat{J}^q=\emptyset$ if $|q-p|>1$, 
by using the Lipschitz continuity of $\alpha$ and $\beta$ and the fact that $\alpha(x)=\beta(x)=0$ if $x\notin(0,1)$, we obtain
\begin{align*}
\mathbb{B}_{3,s}
=\frac{1}{k_n}\sum_{p}\sum_{i=(p-k_n+1)\vee 1}^{p}\alpha^n_{p-i}\beta^n_{p-i}[V,W](\widehat{I}^p)_s+O_p(k_n^{-1})
\end{align*}
uniformly in $s\in[0,t]$. Since $k_n^{-1}\sum_{i=(p-k_n+1)\vee 1}^{p}\alpha^n_{p-i}\beta^n_{p-i}=\phi_{\alpha,\beta}(0)+O_p(k_n^{-1})$ uniformly in $p\geq k_n$ by the Lipschitz continuity of $\alpha$ and $\beta$, we conclude that $\sup_{0\leq s\leq t}|\mathbb{B}_{3,s}-\phi_{\alpha,\beta}(0)[V,W]_s|=o_p(b_n^{1/4})$. Thus, we compete the proof.
\end{proof}

\begin{lem}\label{pqvarg}
Suppose that $[\mathrm{H}1]$, $[\mathrm{SH}3]$, $[\mathrm{H}4]$, $[\mathrm{SH}5]$--$[\mathrm{SH}6]$, $(\ref{absmod})$ and $(\ref{SA4})$ are satisfied. Let $M\in\{M^X,\mathfrak{E}^X,$ $\mathfrak{M}^{\underline{X}}\}$, $N\in\{M^Y,\mathfrak{E}^Y,\mathfrak{M}^{\underline{Y}}\}$ and $\alpha,\beta\in\Upsilon$. Then
\begin{enumerate}[\normalfont (a)]

\item $\sup_{0\leq t\leq T}|\mathbb{M}_{\alpha,\beta}(A^1,N)^n_t|=o_p(b_n^{1/4})$, $\sup_{0\leq t\leq T}|\mathbb{M}_{\alpha,\beta}(M,A^2)^n_t|=o_p(b_n^{1/4})$ and $\sup_{0\leq t\leq T}|\mathbb{M}_{\alpha,\beta}(A^1,A^2)^n_t|$ $=o_p(b_n^{1/4})$ as $n\to\infty$ for any $A^1\in\{A^X,\mathfrak{A}^{\underline{X}}\}$, $A^2\in\{A^Y,\mathfrak{A}^{\underline{Y}}\}$ and any $T>0$.

\item $\sup_{0\leq s\leq t}|\mathbb{M}_{\alpha,\beta}(M,N)^n_s|=O_p(b_n^{1/4})$ as $n\to\infty$ for any $t>0$.

\end{enumerate}
\end{lem}

\begin{proof}
(a) First, $\mathbb{M}^{(1)}_{\alpha,\beta}(A^1,N)^n$ is obviously a locally square-integrable martingale and we can easily prove $\langle\mathbb{M}^{(1)}_{\alpha,\beta}(A^1,N)^n\rangle_t=o_p(b_n^{1/2})$. Thus we have $\sup_{0\leq t\leq T}|\mathbb{M}^{(1)}_{\alpha,\beta}(A^1,N)^n_t|=o_p(b_n^{1/4})$ by the Lenglart inequality. On the other hand, since the quantity $\mathbb{M}^{(2)}_{\alpha,\beta}(A^1,N)^n$ has asymptotically a structure similar to that of the process $\widetilde{\mathbb{II}}$ defined in Section \ref{proofmainthm} (see Eq.~$(\ref{psi0})$), we can adapt an argument similar to that of the proof of Lemma \ref{HYlem13.1and13.2}. Hence we obtain $\sup_{0\leq t\leq T}|\mathbb{M}^{(2)}_{\alpha,\beta}(A^1,N)^n_t|=o_p(b_n^{1/4})$, and thus we conclude that $\sup_{0\leq t\leq T}|\mathbb{M}_{\alpha,\beta}(A^1,N)^n_t|=o_p(b_n^{1/4})$. Similarly we can show the other claims.

(b) Since
\begin{align*}
\langle\mathbb{M}^{(1)}_{\alpha,\beta}(M,N)^n\rangle_t=\sum_q\sum_{p:q-k_n<p<q-1}\sum_{p':q-k_n<p'<q-1}\phi_{\alpha,\beta}\left(\frac{q-p}{k_n}\right)\phi_{\alpha,\beta}\left(\frac{q-p'}{k_n}\right) M(\widehat{I}^{p})_-M(\widehat{I}^{p'})_-\bullet \langle N\rangle(\widehat{J}^{q})_t,
\end{align*}
$\langle\mathbb{M}^{(1)}_{\alpha,\beta}(M,N)^n\rangle_t$ has asymptotically a structure similar to that of $\widetilde{\Delta}_{1,t}$ defined in Section \ref{proofmainthm} (see Eq.~$(\ref{mimic2})$). Consequently, we can adapt an argument similar to that of the proof of Lemma \ref{HYlem12.6and12.8}, hence we obtain
\begin{align*}
\langle\mathbb{M}^{(1)}_{\alpha,\beta}(M,N)^n\rangle_t&=\sum_{p,q:q-k_n<p<q-1}\phi_{\alpha,\beta}\left(\frac{q-p}{k_n}\right)^2\langle M\rangle(\widehat{I}^{p})_t\langle N\rangle(\widehat{J}^{q})_t+o_p(b_n^{1/2}).
\end{align*}
Therefore, by an argument similar to the proof of Lemma 4.6 of \cite{Koike2012phy} we can show that $b_n^{-1/2}\langle\mathbb{M}^{(1)}_{\alpha,\beta}(M,N)^n\rangle_t$ converges to a random variable in probability. In particular, $\langle\mathbb{M}^{(1)}_{\alpha,\beta}(M,N)^n\rangle_t=O_p(b_n^{1/2})$. Similarly we can prove $\langle\mathbb{M}^{(2)}_{\alpha,\beta}(M,N)^n\rangle_t=O_p(b_n^{1/2})$, and thus the Lenglart inequality yields the desired result.
\end{proof}

Now we can prove a lemma about the asymptotic property of the estimator $\Xi[f]^n$.
\begin{lem}\label{Xifrep}
Suppose that $[\mathrm{H}1]$, $[\mathrm{SH}3]$, $[\mathrm{H}4]$, $[\mathrm{SH}5]$--$[\mathrm{SH}6]$, $(\ref{absmod})$ and $(\ref{SA4})$ are satisfied. Then
\begin{align*}
\sup_{0\leq s\leq t}\left|\Xi[f]^n_s-\frac{\|f'\|^2_2}{\theta}\left\{\sum_{k=1}^{N^{n,1}_s+1}[\underline{X},Y]_{\widehat{S}^{k-1}}|\check{I}^{k}|-\sum_{k=1}^{N^{n,2}_s+1}[X,\underline{Y}]_{\widehat{T}^{k-1}}|\check{J}^{k}|\right\}\right|=O_p(b_n^{1/4})
\end{align*}
as $n\to\infty$ for any $t>0$.
\end{lem}

\begin{proof}
Note that $\phi_{f,f'}(0)=\phi_{f',f''}(0)=0$ and $\phi_{f,f''}(0)=-\phi_{f',f'}(0)=-\|f\|^2_2$ by integration by parts. Since we can easily prove $\Xi_{f',f}(\mathsf{X},\mathsf{Y})^n_s=\widetilde{\Xi}_{f',f}(X,Y)^n_s+\widetilde{\Xi}_{f'',f}(\mathfrak{U}^X,Y)^n_s+\widetilde{\Xi}_{f',f'}(X,\mathfrak{U}^Y)^n_s+\widetilde{\Xi}_{f'',f'}(\mathfrak{U}^X,\mathfrak{U}^Y)^n_s+o_p(b_n^{-1/4})$ uniformly in $s\in[0,t]$, Lemma \ref{Xirep}--\ref{pqvarg} and the fact that $k_n\sqrt{b_n}=\theta+o(b_n^{1/4})$ imply that
\begin{align*}
\sup_{0\leq s\leq t}\left|\Xi_{f',f}(\mathsf{X},\mathsf{Y})^n_s-\frac{\|f\|^2_2}{\theta}\sum_{p}\left\{[\underline{X},Y](\check{I}^p)_s-[X,\underline{Y}](\check{J}^p)_s\right\}\right|=O_p(b_n^{1/4}).
\end{align*}
Then, by arguments similar to the proofs of $(\ref{shiftA})$ and $(\ref{shiftGamma})$ we obtain
\begin{equation}\label{Xiffprime}
\sup_{0\leq s\leq t}\left|\Xi_{f',f}(\mathsf{X},\mathsf{Y})^n_s-\frac{\|f\|^2_2}{\theta}\left\{\sum_{k=1}^{N^{n,1}_s+1}[\underline{X},Y]_{\widehat{S}^{k-1}}|\check{I}^{k}|-\sum_{k=1}^{N^{n,2}_s+1}[X,\underline{Y}]_{\widehat{T}^{k-1}}|\check{J}^{k}|\right\}\right|=O_p(b_n^{1/4}).
\end{equation}
In a similar manner we can show the equation obtained by replacing $\Xi_{f',f}(\mathsf{X},\mathsf{Y})^n_s$ with $-\Xi_{f,f'}(\mathsf{X},\mathsf{Y})^n_s$ in $(\ref{Xiffprime})$. Consequently, we obtain the desired result.
\end{proof}

\begin{proof}[\bf\upshape Proof of Proposition \ref{localest}]

(a) The claim immediately follows from Theorem \ref{mainthm}.

(b) First, since
\begin{align*}
\left[\sum_{k=1}^{N^{n,1}_s+1}-\sum_{k=1}^{N^{n,1}_{(s-h_n)_+}+1}\right]\left\{\Psi^{11}_{\widehat{S}^{k-1}}+b_n^{-1}[\underline{X}]'_{\widehat{S}^{k-1}}|\check{I}^{k}|\right\}
=\sum_{k=1}^{N^{n,1}_s+1}1_{\{\widehat{S}^{k-1}>(s-h_n)_+\}}\left\{\Psi^{11}_{\widehat{S}^{k-1}}+b_n^{-1}[\underline{X}]'_{\widehat{S}^{k-1}}|\check{I}^{k}|\right\},
\end{align*}
Lemma \ref{useful}, \ref{localC3} and \ref{noiserep} yield
$\partial\gamma^n_s(1)^{11}=\frac{1}{k_n^2 h_n}\sum_{k=1}^{N^{n,1}_s+1}1_{\{\widehat{S}^{k-1}>(s-h_n)_+\}}\overline{\Psi}^{11}_{\widehat{S}^{k-1}}+o_p(1).$
Then, noting that $|N^n_s-N^{n,1}_s|\leq 1$, again using Lemma \ref{localC3}, we obtain
$\partial\gamma^n_s(1)^{11}=\frac{b_n^{-1}}{k_n^2 h_n}\sum_{k=1}^{N^{n}_s+1}1_{\{\widehat{S}^{k-1}>(s-h_n)_+\}}\overline{\Psi}^{11}_{\widehat{S}^{k-1}}|\Gamma^k|/G^n_{\widehat{S}^{k-1}}$ $+o_p(1).$
Now, [H1] and the dominated convergence theorem imply that
$\partial\gamma^n_s(1)^{11}=\frac{b_n^{-1}}{k_n^2 h_n}\int_{(s-h_n)_+}^s\overline{\Psi}^{11}_{u}/G^n_{u}\mathrm{d}u+o_p(1).$
Since $b_n^{-1}/k_n^2\to\theta^{-2}$ and $h_n^{-1}\int_{(s-h_n)_+}^s\overline{\Psi}^{11}_{u}/G^n_{u}\mathrm{d}u\to\overline{\Psi}^{11}_{s-}/G^n_{s-}$ a.s., we conclude that $\partial\gamma^n_s(1)^{11}\to^p\theta^{-2}\overline{\Psi}^{11}_{s-}/G^n_{s-}$ as $n\to\infty$. The tightness of $\sup_{0\leq s\leq t}|\partial\gamma^n_s(1)^{11}|$ follows from Lemma \ref{tightN} and \ref{noiserep}. Similarly we can also show that the claims about the others respectively.

(c) An argument similar to the proof of (b) with using Lemma \ref{Xifrep} instead of Lemma \ref{noiserep} imply the desired result.
\end{proof}

%% file: hitting/hitting_dependent.tex
\section{Proof of Theorem \ref{depCLT}}\label{proofdepCLT}

\begin{lem}
Suppose $S^i=T^i$ for every $i$. Suppose also that $(\ref{C3})$, $[\mathrm{H}3]$--$[\mathrm{H}6]$, $(\ref{weakdep})$ and $(\ref{depmodel})$ are satisfied. Then 
$b_n^{-1/4}\{\widehat{PHY}(\mathsf{X},\mathsf{Y})^n-\widehat{PHY}(\mathsf{X}',\mathsf{Y}')^n\}\xrightarrow{ucp}0$
as $n\to\infty$, where $\mathsf{X}'_{S^i}=X_{S^i}+\tilde{\lambda}^1_0\epsilon^X_{S^i}+\tilde{\mu}^1_0b_n^{-1/2}\underline{X}(I^i)$ and $\mathsf{Y}'_{S^i}=Y_{S^i}+\tilde{\lambda}^2_0\epsilon^Y_{S^i}+\tilde{\mu}^2_0b_n^{-1/2}\underline{Y}(I^i)$.
\end{lem}

\begin{proof}
By a localization procedure, we can replace [H3] and [H5]--[H6] with [SH3] and [SH5]--[SH6] respectively. Furthermore, a localization argument similar to that in the first part of Section 6 of \cite{Koike2012phy} allows us to assume $(\ref{SA4})$ and that there exists a positive number $C$ such that
\begin{equation}\label{boundN}
N^n_t\leq C b_n^{-1}\qquad\textrm{for any}\ t\in\mathbb{R}_+\textrm{ and any }n\in\mathbb{N}.
\end{equation}

Set $\tilde{\lambda}^1_u=\sum_{v=u}^\infty\lambda^1_v$ for each $u\in\mathbb{Z}_+$. We define the random variable $\tilde{\epsilon}^X_{i}$ by $\tilde{\epsilon}^X_i=\sum_{u=0}^i\tilde{\lambda}^1_{u+1}\epsilon^X_{S^{i-u}}$ for every $i$. 
Then we have
\begin{align*}
\tilde{\epsilon}^X_{i}-\tilde{\epsilon}^X_{i-1}
=\sum_{u=0}^i\tilde{\lambda}^1_{u+1}\epsilon^X_{S^{i-u}}-\sum_{u=1}^i\tilde{\lambda}^1_{u}\epsilon^X_{S^{i-u}}
=\sum_{u=0}^i(\tilde{\lambda}^1_{u+1}-\tilde{\lambda}^1_{u})\epsilon^X_{S^{i-u}}+\tilde{\lambda}^1_{0}\epsilon^X_{S^i}
=-\sum_{u=0}^i\lambda^1_u\epsilon^X_{S^{i-u}}+\tilde{\lambda}^1_{0}\epsilon^X_{S^i},
\end{align*}
hence we obtain $\sum_{u=0}^i\lambda_u\epsilon^X_{S^{i-u}}=\tilde{\lambda}_{0}\epsilon^X_{S^i}-(\tilde{\epsilon}^X_{i}-\tilde{\epsilon}^X_{i-1}).$
Combining this formula with Abel's partial summation formula, we obtain
\begin{align*}
&\sum_{p=0}^{k_n-1}\Delta(g)^n_p\left(\sum_{u=0}^{i+p}\lambda^1_u\epsilon^X_{S^{i+p-u}}\right)
=\tilde{\lambda}^1_{0}\sum_{p=0}^{k_n-1}\Delta(g)^n_p\epsilon^X_{S^{i+p}}-\sum_{p=0}^{k_n-1}\Delta(g)^n_p(\tilde{\epsilon}^X_{i+p}-\tilde{\epsilon}^X_{i+p-1})\\
=&\tilde{\lambda}^1_{0}\sum_{p=0}^{k_n-1}\Delta(g)^n_p\epsilon^X_{S^{i+p}}+\sum_{p=0}^{k_n-1}\Delta^2(g)^n_{p}(\tilde{\epsilon}^X_{i+p}-\tilde{\epsilon}^X_{i-1})
=\tilde{\lambda}^1_{0}\sum_{p=0}^{k_n-1}\Delta(g)^n_p\epsilon^X_{S^{i+p}}+\sum_{p=0}^{k_n-1}\Delta^2(g)^n_{p}\tilde{\epsilon}^X_{i+p}
\end{align*}
for every $i$, where $\Delta^2(g)^n_p=\Delta(g)^n_{p+1}-\Delta(g)^n_p$ (note that $\sum_{p=0}^{k_n-1}\Delta^2(g)^n_{p}=0$). Since
\begin{align*}
\sum_{l=0}^\infty\left| E\left[\tilde{\epsilon}^X_{i+p}\tilde{\epsilon}^X_{i+p+l}\right]\right|
=\sum_{l=0}^\infty\left|\sum_{u=0}^i\tilde{\lambda}^1_{u+1}\tilde{\lambda}^1_{u+l+1}\Psi^{11}_{S^{i-u}}\right|
\lesssim\left(\sum_{u=0}^\infty |\tilde{\lambda}^1_{u+1}|\right)^2
\end{align*}
and $\sum_{u=1}^{\infty}|\tilde{\lambda}^1_u|\leq \sum_{u=1}^{\infty}\sum_{v=u}^{\infty}|\lambda^1_v|=\sum_{v=1}^{\infty}v|\lambda^1_v|<\infty$, we have
\begin{align*}
E\left[\left|\sum_{p=0}^{k_n-1}\Delta^2(g)^n_{p}\tilde{\epsilon}^X_{i+p}\right|^2\right]
=\sum_{p,p'=0}^{k_n-1}\Delta^2(g)^n_{p}\Delta^2(g)^n_{p'}E_0\left[\tilde{\epsilon}^X_{i+p}\tilde{\epsilon}^X_{i+p'}\right]
\leq\frac{2}{k_n^4}\sum_{p=0}^{k_n-1}\sum_{l=0}^\infty\left| E_0\left[\tilde{\epsilon}^X_{i+p}\tilde{\epsilon}^X_{i+p+l}\right]\right|
\lesssim k_n^{-3}
\end{align*}
uniformly in $i$. Similarly, we can show that
\begin{align*}
E\left[\left|b_n^{-1/2}\sum_{p=0}^{k_n-1}\Delta(g)^n_p\left(\sum_{u=0}^{i+p}\mu^1_u\underline{X}(I^{i+p-u})\right)-b_n^{-1/2}\tilde{\mu}^1_0\sum_{p=0}^{k_n-1}\Delta(g)^n_p\underline{X}(I^{i+p})\right|^2\right]\lesssim k_n^{-3}b_n^{\xi'-1}
\end{align*}
uniformly in $i$, where $\underline{X}(I^k)=\underline{X}_{S^k}-\underline{X}_{S^{k-1}}$ for each $k$. Therefore, we have $E[|\overline{\mathsf{X}}_g(\mathcal{I}^i)-\overline{\mathsf{X}'}_g(\mathcal{I}^i)|^2]\lesssim k_n^{-3}b_n^{\xi'-1}$ uniformly in $i$. Hence the Schwarz inequality, $(\ref{SA4})$, [SH3], [H4], [SH5] and $(\ref{boundN})$ imply that $E[\sup_{0\leq s\leq t}|\widehat{PHY}(\mathsf{X},\mathsf{Y})^n_s-\widehat{PHY}(\mathsf{X}',\mathsf{Y})^n_s|]\lesssim b_n^{\xi'-1/2}=o(b_n^{1/4})$. Similarly we can show that $b_n^{-1/4}\{\widehat{PHY}(\mathsf{X}',\mathsf{Y})^n$ $-\widehat{PHY}(\mathsf{X}',\mathsf{Y}')^n\}\xrightarrow{ucp}0$, and thus we complete the proof.
\end{proof}

\begin{proof}[\bf\upshape Proof of Theorem \ref{depCLT}]
By applying Theorem \ref{mainthm} to $\widehat{PHY}(\mathsf{X}',\mathsf{Y}')^n$ in the above, we obtain the desired result.
\end{proof}

%% file: hitting/hitting_mrc.tex
\section{Proof of Theorem \ref{mrcthm}}\label{proofmrc}

First, we can easily show that $b_n^{-1/4}(\Xi_{g,g}(\mathsf{X},\mathsf{Y})^n-\{\widetilde{\Xi}_{g,g}(X,Y)^n+\widetilde{\Xi}_{g',g}(\mathfrak{U}^X,Y)^n+\widetilde{\Xi}_{g,g'}(X,\mathfrak{U}^Y)^n+\widetilde{\Xi}_{g',g'}(\mathfrak{U}^X,\mathfrak{U}^Y)^n\})\xrightarrow{ucp}0$ as $n\to\infty$. Therefore, noting that integration by parts yields $\phi_{g,g'}(0)=\phi_{g',g}(0)=0$, we have $b_n^{-1/4}\{\operatorname{MRC}(\mathsf{X},\mathsf{Y})^n-\mathbb{M}^n\}\xrightarrow{ucp}0$ as $n\to\infty$ from Lemma \ref{Xirep}, Lemma \ref{pqvarg}(a) and the proof of Lemma \ref{noiserep}, where $\mathbb{M}^n=\mathbb{M}_{g,g}(M^X,M^Y)^n+\mathbb{M}_{g',g}(\widetilde{\mathfrak{U}}^X,M^Y)^n+\mathbb{M}_{g,g'}(M^X,\widetilde{\mathfrak{U}}^Y)^n+\mathbb{M}_{g',g'}(\widetilde{\mathfrak{U}}^X,\widetilde{\mathfrak{U}}^Y)^n$. Since $\mathbb{M}^n$ has a structure similar to that of $\widetilde{\mathbf{M}}^n$ defined in Appendix \ref{proofmainthm} (see also the proof of Lemma \ref{pqvarg}), we can adopt an argument similar to the proof of Theorem \ref{mainthm} and conclude that the claim holds true. \hfill $\Box$

%% file: hitting/hitting_acknowledgements.tex

\section*{Acknowledgements}

\addcontentsline{toc}{section}{Acknowledgements}

This work was supported by Grant-in-Aid for JSPS Fellows.
The author is grateful to Professor Nakahiro Yoshida for valuable and helpful discussions.